\documentclass[oneside,english]{amsart}
\usepackage[T1]{fontenc}
\usepackage[latin9]{inputenc}
\usepackage{amsbsy}
\usepackage{amstext}
\usepackage{amsthm}
\usepackage{amssymb}
\usepackage[all]{xy}
\PassOptionsToPackage{normalem}{ulem}
\usepackage{ulem}

\makeatletter
\numberwithin{figure}{section}
\theoremstyle{plain}
\newtheorem{thm}{\protect\theoremname}
  \theoremstyle{plain}
  \newtheorem{prop}[thm]{\protect\propositionname}
  \theoremstyle{definition}
  \newtheorem{defn}[thm]{\protect\definitionname}
  \theoremstyle{plain}
  \newtheorem{lem}[thm]{\protect\lemmaname}
  \theoremstyle{remark}
  \newtheorem{rem}[thm]{\protect\remarkname}
  \theoremstyle{remark}
  \newtheorem{claim}[thm]{\protect\claimname}
  \theoremstyle{plain}
  \newtheorem{cor}[thm]{\protect\corollaryname}
  \theoremstyle{definition}
  \newtheorem{example}[thm]{\protect\examplename}
  \theoremstyle{remark}
  \newtheorem*{claim*}{\protect\claimname}

\usepackage{amsfonts}\usepackage{amsthm}\usepackage{enumerate}\usepackage{mathrsfs}
\usepackage{cancel}\usepackage{url}
\usepackage{MnSymbol}\usepackage{hyperref}
\usepackage{bbm}

\input xy\huge
\xyoption{all}

\DeclareFontFamily{OT1}{pzc}{}
\DeclareFontShape{OT1}{pzc}{m}{it}{<-> s * [1.10] pzcmi7t}{}
\DeclareMathAlphabet{\mathpzc}{OT1}{pzc}{m}{it}

\newcommand{\rr}{\mathbb{R}}
\newcommand{\oo}{\mathcal{O}}

\newcommand{\pp}{\mathbb{P}}

\newcommand{\ff}{\mathcal{F}}

\newcommand{\hh}{\mathcal{H}}
\newcommand{\cc}{\mathbb{C}}

\newcommand{\zz}{\mathbb{Z}}
\newcommand{\qq}{\mathbb{Q}}
\newcommand{\ii}{\mathcal{I}}
\newcommand{\lc}{\mathcal{L}}
\newcommand{\xx}{\mathcal{X}}
\newcommand{\yy}{\mathcal{Y}}
\newcommand{\mm}{\mathcal{M}}

\newcommand{\im}{\operatorname{Im}}

\newcommand{\Hom}{\mathpzc{Hom}}

\newcommand{\nn}{\mathbb{N}}

\newcommand{\bu}{\mathfrak{bu}}

\newcommand{\manc}{\textbf{Man}^c}

\newcommand{\tb}{\mathbb{T}}

\newcommand{\zc}{\mathcal{Z}}
\newcommand{\sfr}{\mathfrak{s}}
\newcommand{\basic}{\mathfrak{b}}
\newcommand{\tc}{\mathcal{T}}
\newcommand{\ts}{\mathscr{T}}
\newcommand{\aru}[2]{\ar[#1]^-{#2}}
\newcommand{\ard}[2]{\ar[#1]_-{#2}}

\newcommand{\lf}{\mathsf{l}}
\newcommand{\kf}{\mathsf{k}}

\newcommand{\wc}{\mathcal{W}}
\usepackage{tikz} 
\usetikzlibrary{matrix,arrows}
\usepackage{tensor}
\newcommand{\fibp}[2]{\tensor[_{#1\thinspace}]{\times}{_{\thinspace #2}}}
\newcommand{\ev}{\operatorname{ev}}

\newcommand{\evi}{\operatorname{evi}}
\newcommand{\Sym}{\operatorname{Sym}}
\newcommand{\For}{\operatorname{For}}
\newcommand{\Or}{\operatorname{Or}}
\newcommand{\ed}{\operatorname{ed}}
\newcommand{\id}{\operatorname{id}}
\newcommand{\pt}{\operatorname{pt}}
\newcommand{\pr}{\operatorname{pr}}

\newcommand{\Aut}{\operatorname{Aut}}

\newcommand{\sstar}{\smallstar}
\newcommand{\sgn}{\operatorname{sgn}}

\newcommand{\jc}{\mathcal{J}}
\newcommand{\no}{n_\text{odd}}
\newcommand{\vect}[1]{\boldsymbol{\mathbf{#1}}}
\newcommand{\vp}{\vect \phi}
\newcommand{\vpb}{\vect \phi_\bullet}
\newcommand{\pc}{\mathcal{P}}
\newcommand{\cont}{\operatorname{Cont}}
\newcommand{\cnt}{\operatorname{cnt}}
\newcommand{\mv}{\operatorname{mv}}
\newcommand{\rc}{\mathcal{R}}
\usepackage{tikz}
\newcommand*\circled[1]{\tikz[baseline=(char.base)]{
            \node[shape=circle,draw,inner sep=2pt] (char) {#1};}}
\newcommand{\vc}{\mathcal{V}}
\newcommand{\codim}{\operatorname{cd}}

\makeatother

\usepackage{babel}
  \providecommand{\claimname}{Claim}
  \providecommand{\corollaryname}{Corollary}
  \providecommand{\definitionname}{Definition}
  \providecommand{\examplename}{Example}
  \providecommand{\lemmaname}{Lemma}
  \providecommand{\propositionname}{Proposition}
  \providecommand{\remarkname}{Remark}
\providecommand{\theoremname}{Theorem}

\begin{document}

\title{Fixed-point Localization for $\rr\pp^{2m}\subset\cc\pp^{2m}$}

\author{Amitai Netser Zernik}
\begin{abstract}
We derive a fixed-point formula for integrals on moduli spaces of
stable maps to projective spaces of even dimension. This gives a formula
for the equivariant open Gromov-Witten invariants of $\rr\pp^{2m}\subset\cc\pp^{2m}$,
and the structure constants of the equivariant $A_{\infty}$ algebra
$\operatorname{End}{}_{Fuk^{\tb}\left(\cc\pp^{2m}\right)}\left(\rr\pp^{2m}\right)$
with bulk deformations. The formula involves contributions from Givental's
correlators for the closed theory and the descendent integrals of
discs, and specializes to give a new expression for the Welschinger
count of real rational curves in the plane passing through some real
and conjugation invariant pairs of points.
\end{abstract}

\maketitle
\tableofcontents{}

\section{Introduction}

\subsection{Overview}

In this paper we derive a fixed-point formula for the equivariant
open Gromov-Witten theory of $\rr\pp^{2m}\subset\cc\pp^{2m}$, as
a special case of a more general integration technique on moduli spaces
of discs. 

The action of a torus group $\tb$ on the moduli space of \emph{closed}
stable maps $\overline{\mm}_{0,l}\left(\cc\pp^{n},d\right)$ can be
used to reduce the computation of an integral on $\overline{\mm}_{0,l}\left(\cc\pp^{n},d\right)$
to tractable calculations near the fixed points of the $\tb$-action
\cite{original-loc,kontsevich-loc}. This technique has become a
general, powerful tool which has found many applications.

Fixed-point localization has also been applied to \emph{open }Gromov-Witten
theory, which is concerned with integrals on the moduli spaces parameterizing
stable maps of discs. Some notable examples include Katz and Liu \cite{katz+liu},
Pandharipande, Solomon and Walcher \cite{disc-enumeration}, Georgieva
and Zinger \cite{penka+zinger} and Tehrani and Zinger\footnote{The last two works study \emph{real }invariants, but these are closely
related to disc invariants, as we will see.} \cite{tehrani+zinger}. Typically, moduli spaces of discs have boundary,
so the geometric setup in these works is constrained by the requirement
that the boundary contributions vanish. 

Our approach is different, in that the boundary contributions are
not required to \emph{vanish} but are rather \emph{canceled }by boundary
contributions from integrals on additional spaces. One novel feature
is that this allows one to introduce boundary constraints. The construction
of these additional spaces is closely related to the homological perturbation
lemma for $A_{\infty}$ algebras. 

On the same note, let $\operatorname{End}{}_{Fuk^{\tb}\left(\cc\pp^{2m}\right)}\left(\rr\pp^{2m}\right)$
denote the quantum deformation of the De Rham differential graded
algebra $\left(\Omega\left(\rr\pp^{2m}\right),d,\wedge\right)$ over
the Novikov ring with bulk and equivariant deformation parameters
included. This algebra admits a natural rigid minimal model  and
our formula computes the structure constants of this model explicitly.

\subsection{Equivariant open Gromov-Witten invariants}

Fix some $m\geq1$ and let $\left(X,L\right)=\left(\cc\pp^{2m},\rr\pp^{2m}\right)$.
We review the definition of equivariant open Gromov-Witten invariants
for $L\subset X$ from \cite{equiv-OGW-invts}. See $\S$\ref{subsec:Review-of-Resolutions}
for more details.

For $\basic=\left(k,l,\beta\right)\in\zz_{\geq0}^{3}$ such that $k+2l+3\beta$
is an odd integer $\geq3$, we construct a sequence of orbifolds with
corners and maps
\[
\left(\widetilde{\mm}_{b}^{r}\to\mm_{\basic}^{r}\to\check{\mm}_{b}^{r}\right)_{r\geq0}
\]
We have 
\[
\widetilde{\mm}_{\basic}^{0}=\mm_{\basic}^{0}=\check{\mm}_{\basic}^{0}=\overline{\mm}_{0,k,l}\left(X,L,\beta\right),
\]
the moduli space of stable discs of degree $\beta\in H_{2}\left(X,L\right)$
with $k$ boundary marked points and $l$ interior marked points.
More generally, $\mm_{\basic}^{r}$ is a disjoint union of products
of moduli spaces, modeled on trees with $r$ oriented edges. The endpoints
of the edges correspond to $2r$ additional boundary marked points.
$\check{\mm}_{\basic}^{r}$ is obtained from $\mm_{\basic}^{r}$ by
forgetting the markings corresponding to the tails of the edges, and
$\widetilde{\mm}_{\basic}^{r}$ is obtained from $\mm_{\basic}^{r}$
by blowing up the locus where two endpoints of an edge map to the
same point in $L$. All of the spaces and maps come equipped with
a natural action of the rank $m$ torus $\tb=U\left(1\right)^{m}$.

We use these spaces to define an integration map 
\[
\int_{\basic}:\Omega_{\basic}\to\rr\left[\lambda_{1},...,\lambda_{m}\right]
\]
Here $\left(\Omega_{b},D\right)$ is the complex of \emph{extended
forms}:\emph{ }an extended form is specified by a sequence $\left(\check{\omega}_{r}\right)_{r\geq0}$,
where each $\check{\omega}_{r}$ is a $\tb$-invariant $\rr\left[\lambda_{1},...,\lambda_{m}\right]$-valued
form on $\check{\mm}_{b}^{r}$, subject to some coherence conditions.
The Cartan-Weil differential
\[
D=d-\sum_{j=1}^{m}\lambda_{j}\iota_{\xi_{j}}
\]
acts on $\Omega_{\basic}$ level-wise. By definition,
\begin{equation}
\int_{b}\omega=\sum_{r\geq0}\frac{1}{r!}\int_{\widetilde{\mm}_{b}^{r}}\tilde{\omega}_{r}\cdot\Lambda^{\boxtimes r}\label{eq:integration}
\end{equation}
where $\Lambda$ is an equivariant form on the diagonal blow up 
\begin{equation}
\widetilde{L\times L}=\left(L\times L\backslash\Delta\right)\bigcup_{S\left(N_{\Delta}\right)\times\left(0,\infty\right)}S\left(N_{\Delta}\right)\times[0,\infty),\label{eq:L x L blow up}
\end{equation}
representing a homotopy retraction, and $\tilde{\omega}_{r}$ is the
pullback of $\check{\omega}_{r}$ to $\widetilde{\mm}_{\basic}^{r}$.
The main result about this construction is Stokes' theorem,
\begin{equation}
\int_{b}D\omega=0.\label{eq:stokes theorem}
\end{equation}
For $k,\beta\in\zz_{\geq0}$ and $\boldsymbol{l}=\left(l_{0},...,l_{2m}\right)\in\zz_{\geq0}^{2m+1}$
with $\sum\boldsymbol{l}=l$, we define the equivariant open Gromov-Witten
invariant

\[
I_{m}\left(k,\boldsymbol{l},\beta\right)=\int_{\basic}\omega,
\]
where $\omega$ is given by pulling back equivariant closed forms
along the interior and boundary evaluation maps,
\begin{equation}
{\check{\omega}_{r}=\prod_{0\leq d\leq2m}\,\prod_{\boldsymbol{l}\left[:d\right]+1\leq a\leq\boldsymbol{l}\left[:d+1\right]}\left(\evi{}_{a}^{\basic,r}\right)^{*}\mathcal{H}^{d}\,\cdot\,\prod_{1\leq i\leq k}\left(\ev{}_{i}^{\basic,r}\right)^{*}p_{0}^{L}}.\label{eq:omega of ogw invts}
\end{equation}
Here we've set $\boldsymbol{l}\left[:d\right]:=\left(\sum_{j=0}^{d-1}l_{j}\right)$
and let $\hh$ and $p_{0}^{L}$ denote equivariant forms representing
the hyperplane class in $\cc\pp^{2m}$ and the point class in $\rr\pp^{2m}$,
respectively (see $\S$\ref{subsec:RP^2m in CP^2m}). 

If we assume $m=1$, $\deg\check{\omega}=\dim\check{\mm}_{\basic}^{0}$,
and we take $\mathcal{H}$ to be conjugation invariant, we can show
the contribution of $\check{\mm}_{\basic}^{r}$ for $r>0$ vanishes,
and the contribution of $\check{\mm}_{\basic}^{0}$ matches the invariants
defined by Solomon \cite{JT}. By a reflection principle argument
\cite{JT}, these invariants are shown to be twice Welschinger's signed
count \cite{welschinger} of real rational planar curves through a
generic configuration of $k$ real points and $l$ conjugate pairs
of complex points. In sum, we have
\begin{equation}
W_{k,l}=\frac{1}{2}I_{1}\left(k,\left(0,0,l\right),\frac{k+2l+1}{3}\right),\label{eq:welschinger-solomon}
\end{equation}
so the Welschinger invariants are instances of the open Gromov-Witten
invariants.

\subsection{\label{subsec:open fp formula}Statement of main results}

Our geometric result can be stated as follows. 
\begin{thm}
\label{thm:general fp intro}Let $\omega\in\Omega_{\basic}$ satisfy
$D\omega=0$. Then
\end{thm}
\begin{equation}
\int_{\basic}\omega=\sum_{\vp\in\mathcal{W}}\frac{\xi_{\vp}}{\operatorname{Aut}\vp}\int_{\check{F}_{\vp}}\frac{\check{\omega}_{\vp}}{e\left(\check{N}_{\vp}\right)}.\label{eq:extended fp formula abridged}
\end{equation}
Here $\coprod_{\vp\in\mathcal{W}}\check{F}_{\vp}$ is a finite cover
of the $\tb$-fixed points $C^{\tb}$ of a clopen component $C\subset\coprod_{r\geq0}\check{\mm}_{\basic}^{r}$
parameterizing configurations where all the edges connect an odd degree
disc to an even degree disc. $e\left(\check{N}_{\vp}\right)$ is
an equivariant Euler form for the normal bundle to $\check{F}_{\vp}\to\check{\mm}_{\basic}^{r=r\left(\vp\right)}$,
subject to certain boundary conditions. $\xi_{\vp}$ represents the
contribution of $\Lambda^{\boxtimes r}$. See Theorem \ref{thm:general fixed point formula}
and Remark \ref{rem:sum over orbits} for more details.

We turn now to a purely algebraic statement, succinctly expressing
the equivariant open Gromov-Witten invariants in terms of known generating
functions. Hereafter, when a formal variable is introduced we use
a number in overbrace to denote its $\zz$-grading. Let $\cc\left[\vect\lambda\right]=\cc\left[\overbrace{\lambda_{1}}^{2},...,\overbrace{\lambda_{m}}^{2}\right]$
denote the equivariant $\tb$-cohomology of a point. Let 
\[
R=\cc\left[\vect\lambda\right]\left[\left[\eta_{0},...,\overbrace{\eta_{i}}^{2-2i},...,\eta_{2m}\right]\right].
\]
We define a generating function
\[
F_{OGW}\in R\left[\left[\overbrace{q}^{2m+1},\overbrace{s}^{1-2m}\right]\right]
\]
for the equivariant open Gromov-Witten invariants by 

\[
F_{OGW}=\sum_{k,\boldsymbol{l},\beta}I\left(k,\boldsymbol{l},\beta\right)\,\left(\sqrt{-1}\right)^{\left(m+1\right)w_{1}\left(\beta\right)}\,q^{\beta}\,\frac{\vect\eta{}^{\boldsymbol{l}}}{\boldsymbol{l}!}\,\frac{s^{k}}{k!}
\]
where sum ranges over all $k,\beta\in\zz_{\geq0}$ and $\boldsymbol{l}\in\zz_{\geq0}^{2m+1}$,
${\frac{\vect\eta^{\boldsymbol{l}}}{\boldsymbol{l}!}:=\frac{\eta_{0}^{l_{0}}\cdot....\cdot\eta_{2m}^{l_{2m}}}{l_{0}!\cdots l_{2m}!}}$,
and
\[
w_{1}\left(\beta\right)=\beta\mod2\in\left\{ 0,1\right\} .
\]
$F_{OGW}$ is supported in degree $3-2m$. 

Let 
\[
F_{0}^{o}\in\rr\left[\left[\overbrace{s_{0}}^{1},t_{0},t_{1},...,\overbrace{t_{i}}^{2-2i},...\right]\right]
\]
denote the generating function for the descendent integrals of discs,
developed in \cite{JRR} (we review the definition in $\S$\ref{subsec:descendent integrals}).
We introduce auxiliary formal variables $u$ and $\nu_{a,d}$ for
$a=1,...,2m$ and $d>0$ a positive integer.

We set $\alpha_{0}=0$ and, for $1\leq i\leq m$, $\alpha_{i}=\lambda_{i}$
and $\alpha_{2m+1-i}=-\lambda_{i}$. We denote by $Z_{a}\left(\frac{\alpha_{a}}{d}\right)$
the value of the quantum correlator with constraints, see $\S$\ref{subsec:Constraint-correlators},
which encapsulates the Gromov-Witten theory of closed curves in $\cc\pp^{2m}$.
\begin{thm}
\label{thm:open localization}Let 
\[
E\in\mathcal{R}\left[\left[u^{-1},u,\nu_{a,d},q,\eta_{0},...,\eta_{2m},s,s_{0},t_{0},t_{1},...\right]\right]
\]
 be defined by 
\begin{multline*}
E=\exp\left(\sum_{a\neq0}\sum_{d>0}Z_{a}\left(\frac{\alpha_{a}}{d}\right)\partial_{\nu_{a,d}}\right)\\
\exp\left(\left(\sqrt{-1}\right)^{\left(1+m\right)}\sum_{a\neq0}\sum_{d\mbox{ odd}}uq^{d}\nu_{a,d/2}\frac{d^{d-1}}{d!\left(2\alpha_{a}\right)^{d}}\frac{\exp\left(\frac{d}{2}\cdot\frac{\lambda_{1}\cdots\lambda_{m}}{u\alpha_{a}}\cdot\partial_{s_{0}}\right)}{\prod_{0\leq c\leq\frac{d-1}{2},b\neq a,\overline{a}}\left(\frac{2c-d}{d}\alpha_{a}-\alpha_{b}\right)}\right)\\
\exp\left(\sum_{a\neq0}\sum_{d>0}q^{2d}\nu_{a,d}\frac{\left(-1\right)^{d}d^{2d-1}\cdot\left(\prod_{a\in\left\{ 1,...2m\right\} }\alpha_{a}\right)\cdot\left(\sum_{i\geq0}\left(-\frac{d}{\alpha_{a}}\right)^{i+1}\partial_{t_{i}}\right)}{\left(d!\right)^{2}\alpha_{a}^{2d}\prod_{0\leq c\leq d,b\neq a,0}\left(\frac{c}{d}\alpha_{a}-\alpha_{b}\right)}\right)\\
\exp\left(\frac{u}{\lambda_{1}\cdots\lambda_{m}}\exp\left(s\cdot\lambda_{1}\cdots\lambda_{m}\cdot\partial_{s_{0}}\right)\exp\left(\eta_{0}\partial_{t_{0}}\right)F_{0}^{o}\right)
\end{multline*}
Here  $\sum_{a\neq0}$ means $a$ ranges over $1,...,2m$. 

The OGW-generating function is given by

\textbf{
\begin{equation}
F_{OGW}=\partial_{u}|_{u=0}\log\left(E|_{s_{0}=0,\nu_{a,d}=0,t_{0}=0,...,t_{i}=0,...}\right).\label{eq:ogw loc formula}
\end{equation}
}
\end{thm}
Here the logarithm is defined by Taylor expansion around 1:
\begin{equation}
\log x=\left(x-1\right)-\frac{\left(x-1\right)^{2}}{2}+\frac{\left(x-1\right)^{3}}{3}-\cdots.\label{eq:log taylor expansion}
\end{equation}
The proof of this theorem is given in $\S$\ref{subsec:Equivariant-open-Gromov-Witten}.

The formulae in the theorem are nothing more than a compact
description of a diagrammatic sum; the four factors comprising $E$
generate four types of labeled vertices, with the inner exponents
generating edges of various types incident to these vertices. The
substitution $E|_{s_{0}=0,\nu_{a,d}=0,t_{0}=0,...,t_{i}=0,...}$ places
some restrictions on the labels, and $\partial_{u}|_{u=0}\log\left(-\right)$
replaces the sum over all diagrams with a sum over tree diagrams. 

As a special case, this gives a combinatorial expression for the Solomon-Welschinger
invariants (\ref{eq:welschinger-solomon}). We've used a computer
to compute all of the invariants with $k+2l\leq17$ and checked the
results agree with values computed using WDVV, see \cite{klm}. For
example, we have 
\[
W_{17,0}=2,845,440.
\]
Another validation for this formula comes from the obvious fact 
\[
\int_{\basic}\omega=0\text{ when }\deg\omega<\dim\check{\mm}_{\basic}^{0},
\]
though the individual fixed-point contributions to Theorem \ref{thm:general fp intro}
need not vanish. In \cite{descendents+loc} we show that such relations
are in fact sufficient to determine the intersection theory of discs.
Simple examples of both types of sanity checks are given in $\S$\ref{subsec:Examples-and-applications.}.

\subsection{The failure of naive localization}

Corollary \ref{cor:fixed point formula} reproves the classical fixed-point
formula of Atiyah and Bott \cite{AB}. It shows that if $\xx$ is
a closed orbifold and $\omega$ is a closed equivariant form, then
\[
\int_{\xx}\omega=\int_{F}\frac{\omega|_{F}}{e\left(N\right)}
\]
where $F\subset\xx$ denotes the fixed-point suborbifold and $e\left(N\right)$
is the equivariant Euler to the normal bundle of $N$. 

Set $m=1$, so $\left(X,L\right)=\left(\cc\pp^{2},\rr\pp^{2}\right)$,
and consider 
\[
1=W_{2,0}=\frac{1}{2}I_{1}\left(2,\left(0,0,0\right),1\right),
\]
the number of lines through a pair of distinct points in the plane.
This invariant is given by an integral on ${\overline{\mm}_{0,2,0}\left(1\right)}$.
$S^{1}$ acts on $\left(\cc\pp^{2},\rr\pp^{2}\right)$ by 
\[
\begin{pmatrix}1 & 0 & 0\\
0 & \cos\theta & -\sin\theta\\
0 & \sin\theta & \cos\theta
\end{pmatrix}
\]
and thus on $\overline{\mm}_{0,2,0}\left(1\right)$, but it is easy
to see that $\overline{\mm}_{0,2,0}\left(1\right)$ has \emph{no }fixed
points, which would seem to imply $W_{2,0}$ vanishes. In reality,
the fixed points fail to ``see'' the line because 
\[
\partial\overline{\mm}_{0,2,0}\left(1\right)=\overline{\mm}_{0,1,0}\left(1\right)\times_{L}\overline{\mm}_{0,0,3}\left(0\right)\neq\emptyset.
\]
Theorem \ref{thm:general fp intro} recovers the correct answer by
incorporating the fixed points of the resolution ${\check{\mm}_{\basic}^{1}=\overline{\mm}_{0,0,0}\left(1\right)\times\overline{\mm}_{0,3,0}\left(0\right)}$
of $\partial\overline{\mm}_{0,2,0}\left(1\right)$. These fixed points
can be schematically represented by 
\begin{equation}
\circled{1}\to\circled{0}:\label{eq:simple diagram}
\end{equation}
where
\begin{itemize}
\item $\circled{1}$ depicts a fixed point of $\overline{\mm}_{0,0,0}\left(1\right)$,
represented by a degree one disc with no markings, whose boundary
maps to the line at infinity $\rr\pp^{1}\subset L$ (there are two
such fixed-points)
\item $\to$ depicts an oriented edge, which gives the propagator factor
$\xi_{\vp}$ in (\ref{eq:extended fp formula abridged})
\item $>\circled{0}:$ depicts a fixed point of $\overline{\mm}_{0,0,3}\left(0\right)$,
represented by a degree zero disc carrying the two original markings
as well as a marking for the head of the edge $\to$. Such a disc
must map to the origin $p_{0}\in L$.
\end{itemize}
See Example \ref{exa:lines thru 2pts} for the full computation. More
generally, even though we can choose $\omega$ so 
\[
W_{k,l}=\int_{\basic}\omega=\int_{\check{\mm}_{\basic}^{0}}\omega_{0}
\]
with contributions only from the $r=0$ term in (\ref{eq:integration}),
the fixed-point formula (\ref{eq:extended fp formula abridged}) for
$W_{k,l}$ will involve contributions from $\check{\mm}_{\basic}^{r}$
with $r$ arbitrarily large, representing contributions of the $r$-codimension
corners $\partial^{r}\overline{\mm}_{0,k,l}\left(\beta\right)$.

\subsection{Outline of proof}

The structure of the paper is as follows. In Section \ref{sec:statement}
we review and adapt to our needs some related theories. In particular
we discuss quantum constraint correlators and the intersection theory
of discs.

In Section \ref{sec:fp for generalized} we discuss the resolution
fixed points, and use this to formulate Theorem \ref{thm:general fp intro}
more precisely as Theorem \ref{thm:general fixed point formula}.
We then compute some examples, and deduce the generating function
statement, Theorem \ref{thm:open localization}.

The next three sections are devoted to the proof of Theorem \ref{thm:general fp intro}.
Consider 
\[
I\left(t\right)=\int_{\basic}e^{tD\eta}\omega=\sum_{r\geq0}\frac{1}{r!}\int_{\widetilde{\mm}_{\basic}^{r}}e^{tD\widetilde{\eta}_{r}}\,\widetilde{\omega}_{r}\,\Lambda^{\boxtimes r}.
\]
On the one hand, by Stokes' theorem (\ref{eq:stokes theorem}), this
is independent of $t$. On the other hand, for suitable $\eta$ we
show that 
\begin{equation}
\int_{\basic}\omega=I\left(0\right)=\lim_{t\to\infty}I\left(t\right)=\sum_{r\geq0}\frac{1}{r!}\int_{\widetilde{N}_{\basic}^{r}}e^{D\tilde{\eta}_{r}^{\left(2\right)}}\cdot\tilde{\rho}^{*}\left(\tilde{\omega}_{r}\boxtimes\Lambda^{\boxtimes r}\right),\label{eq:sing loc abridged}
\end{equation}
see Theorem \ref{thm:singular localization formula}. Here $\widetilde{N}_{\basic}^{r}$
is a blow up of the normal bundle $N_{\basic}^{r}$ to $F_{\basic}^{r}\to\mm_{\basic}^{r}$,
where $F_{\basic}^{r}=\left(\check{\mm}_{\basic}^{r}\right)^{\tb}\times_{\check{\mm}_{\basic}^{r}}\mm_{\basic}^{r}$
is the inverse image of the fixed points of $\check{\mm}_{\basic}^{r}$.
Using a tubular neighborhood, $\widetilde{N}_{\basic}^{r}$ becomes
an open subset of $\widetilde{\mm}_{\basic}^{r}$. $\tilde{\eta}_{r}^{\left(2\right)}$
is the quadratic part of $\check{\eta}_{r}$, and $\tilde{\rho}$
is a certain retraction. The computation of the limit (\ref{eq:sing loc abridged})
uses steepest descent analysis, but the singularities of $\Lambda$
demand special attention. We discuss a fairly general formalism for
such computations in Section \ref{sec:steepest descent} (as a special
case, we obtain the Atiyah-Bott localization formula for closed orbifolds).
The derivation of (\ref{eq:sing loc abridged}) is carried out in
Section \ref{sec:singular loc}. 

The final step in the proof of Theorem \ref{thm:general fp intro}
is carried out in Section \ref{sec:Resummation}. Note that (\ref{eq:sing loc abridged})
already reduces $\int_{\basic}\omega$ to local quantities near the
fixed points of $\coprod_{r}\check{\mm}_{\basic}^{r}$, but the integrals
are hard to evaluate in this form (see $\S$\ref{subsec:problems with singular}
for a detailed explanation). This is related to the singularities
of $\left(\tilde{\rho}\right)^{*}\Lambda^{\boxtimes r}$ near the
diagonal, which we now discuss. Consider a fixed-point component 
\[
\check{F}_{T}=\prod_{v\in T_{0}}\overline{\mm}_{0,k_{v},l_{v}}\left(\beta_{v}\right)^{\tb}\subset\left(\check{\mm}_{\basic}^{r}\right)^{\tb}
\]
corresponding to a tree $T=\left(T_{0},T_{1}\right)$ with vertex
set $T_{0}$ and $r$ oriented edges $T_{1}$. Consider $\left(v_{\text{tail}},v_{\text{head}}\right)\in T_{1}$.
We must have $\beta_{v_{\text{head}}}=0\mod2$ whereas $\beta_{v_{\text{tail}}}$
may be odd or even, so
\[
T_{1}=T_{\text{odd-even}}\coprod T_{\text{even-even}}.
\]
The two boundary markings corresponding to an odd-even edge $e\in T_{\text{odd-even}}$
are separated in $L$. In contrast, both endpoints of an even-even
edge map to the unique real fixed point $p_{0}\in L$. It follows
that

\[
\pm\tilde{\rho}^{*}\Lambda^{\boxtimes r}=\left(\tilde{\rho}^{*}\Lambda^{\boxtimes T_{\text{odd-even}}}\right)\boxtimes\left(\tilde{\rho}^{*}\Lambda^{\boxtimes T_{\text{even-even}}}\right)=R\boxtimes S
\]
where $R=\tilde{\rho}^{*}\Lambda^{\boxtimes T_{\text{odd-even}}}$
is regular, in the sense that it is pulled back along the blow up
map $\widetilde{N}_{\basic}^{r}\to N_{\basic}^{r}$, while
\[
S=\tilde{\rho}^{*}\Lambda^{\boxtimes T_{\text{even-even}}}=\theta_{0}^{\boxtimes T_{\text{even-even}}}
\]
is singular: $\theta_{0}$ is an angular form for
\[
S\left(T_{p_{0}}L\right)\simeq S\left(N_{\Delta/L\times L}|_{p_{0}}\right),
\]
giving the rescaled limit of $\Lambda$ near $\left(p_{0},p_{0}\right)\in L\times L$. 

Every $e\in T_{\text{even-even}}$ corresponds to a boundary component
$\partial^{T}\check{F}_{T'}$ of some other fixed-point component
$F_{T'}$. $T'$ is obtained from $T$ by contracting $e$. Let us
fix some tree $\underline{T}$ with only odd-even edges, and consider
the total contribution $C_{\underline{T}}$ to (\ref{eq:sing loc abridged})
of the fixed-point components $\left\{ \check{F}_{T}|\cnt T=\underline{T}\right\} $
of trees $T$ whose odd-even skeleton, obtained by contracting all
even-even edges, is $\underline{T}$. The central result of Section
\ref{sec:Resummation}, Theorem \ref{thm:resummation}, says that
$C_{\underline{T}}$ can be regularized. Roughly speaking, the integrals
on $\widetilde{N}_{\basic}^{r}|_{\check{F}_{T}}$ are replaced by
integrals on 
\[
\partial^{T}\check{F}_{\underline{T}}\times\left[0,1\right]^{T_{\text{even-even}}},
\]
where $\partial^{T}\check{F}_{\underline{T}}$ is the codimension
$\left|T_{\text{even-even}}\right|$ corner of $\check{F}_{\underline{T}}$
corresponding to $T$. Thus we find that
\[
C_{\underline{T}}=\int_{P_{\underline{T}}}\Upsilon
\]
where the domain
\[
P_{\underline{T}}=\bigcup_{c\geq0}\partial^{c}\check{F}_{\underline{T}}\times_{\Sym\left(c\right)}\left[0,1\right]^{c}
\]
is obtained by gluing $\partial\check{F}_{\underline{T}}\times\left[0,1\right]$
to $\check{F}_{\underline{T}}$ along $\partial\check{F}_{\underline{T}}\times\left\{ 0\right\} $,
then filling in the missing squares at the codimension two corners,
then the missing cubes at codimension 3 corners, and so on. The integrand
$\Upsilon$ satisfies boundary conditions, making it a well-defined
cohomology class relative to $\partial P_{\underline{T}}$ . In particular,
$C_{\underline{T}}$ is invariant under $\omega\to\omega+D\epsilon$
and the choice of $\Lambda$. Theorem \ref{thm:resummation} allows
us to replace the sum $\sum_{\underline{T}}C_{\underline{T}}$ over
odd-even trees $\underline{T}$, appearing on the right hand side
of (\ref{eq:sing loc abridged}), with the right hand side of (\ref{eq:extended fp formula abridged}),
and conclude the proof of Theorem \ref{thm:general fp intro}.

Some definitions and basic results about orbifolds with corners are
given in the Appendix.

\subsection{Quantum $A_{\infty}$ deformations and the superpotential}

Though we will not use the $A_{\infty}$ formalism in this paper,
it is closely related to what we do, and in this subsection we briefly
outline this relationship. The equivariant open Gromov-Witten invariants
$\left\{ I_{m}\left(k,\boldsymbol{l},\beta\right)\right\} $ encode
the structure constants of a natural, rigid minimal model for the
Fukaya endomorphism algebra of $\rr\pp^{2m}\subset\cc\pp^{2m}$ with
bulk and equivariant deformations. The question of whether or not
the Calabi-Yau structure is given by the standard pairing on this
minimal model requires further study. We will mention a condition
on $\Lambda$ that should guarantee this.

Consider the $\left(\zz\oplus\zz/2\right)$-graded $R$-module
\[
C=\overbrace{\Omega}^{\left(\zz,0\right)}\left(\rr\pp^{2m},\overbrace{R}^{\left(2\zz,0\right)}\otimes_{\zz}\left(\overbrace{\zz}^{\left(0,0\right)}\oplus\overbrace{\Or\left(T\rr\pp^{2m}\right)}^{\left(0,1\right)}\right)\right).
\]
so the $\zz$ component of the grading, denoted $\codim x$, is the
total grading of differential forms with values in the graded ring
$R=\rr\left[\vect\lambda\right]\left[\left[\vect\eta\right]\right]$,
where now $\deg\lambda_{i}=\left(2,0\right)$ and $\deg\eta_{i}=\left(2-2i,0\right)$.

$C$ is equipped with a pairing

\begin{equation}
\left\langle x,y\right\rangle =\left(-1\right)^{\epsilon\left(x,y\right)}\int x\wedge y,\;\epsilon\left(x,y\right)=\codim x+\codim x\cdot\codim y.\label{eq:pairing}
\end{equation}
For $\left(k,\beta\right)\in\zz_{\geq0}^{2}\backslash\left\{ \left(1,0\right)\right\} $
we define an operation 
\[
m_{k,\beta}:C^{\otimes k}\to C\left[2-k-\mu\left(\beta\right),\mu\left(\beta\right)\mod2\right],\;\mu\left(\beta\right)=\left(2m+1\right)\cdot\beta
\]
by

\begin{align*}
m_{k,\beta}\left(x_{1},...,x_{k}\right):= & \sum_{\boldsymbol{l}\in\zz_{\geq0}^{2m+1}}\frac{\boldsymbol{\eta}^{\boldsymbol{l}}}{\boldsymbol{l}!}\left(\sqrt{-1}\right)^{\star}\times\\
\times & \ev_{0*}^{\left(k,\sum\boldsymbol{l},\beta\right)}\left(\ev_{1}^{*}x_{1}\,\ev_{2}^{*}x_{2}\,\cdots\,\ev_{k}^{*}x_{k}\,\prod_{d=0}^{2m}\prod_{\boldsymbol{l}\left[:d\right]+1\leq a<\boldsymbol{l}\left[:d+1\right]}\evi_{a}^{*}\hh^{i}\right)
\end{align*}
where ${\sqrt{-1}^{*}=\left(-1\right)^{\sum_{i=1}^{k}i\left(\codim x_{i}+1\right)}\cdot\begin{cases}
1 & \beta=0\mod2\\
\sqrt{-1} & \beta=1\mod2
\end{cases}}$, $\mathcal{H}$ is an equivariant hyperplane form $\S$\ref{subsec:RP^2m in CP^2m},
and the pushforward is defined using the local system map $\check{\jc}_{\left(k,\sum\boldsymbol{l},\beta\right)}^{0}$
(\ref{eq:check J}). We set 
\[
m_{1,0}=D=d-\sum_{j=1}^{m}\lambda_{j}\iota_{\xi_{j}},
\]
the Cartan-Weil differential. The operations $\left\{ m_{k,\beta}\right\} $
define a $\left(G=\left(1,\left(2m+1\right)\right)\cdot\zz_{\geq0}\right)$-gapped
unital cyclic twisted $A_{\infty}$ algebra over the complete local
ring $\left(R,\mathfrak{m}=R\cdot\left(\eta_{0},...,\eta_{2m}\right)\right)$\footnote{This is defined similarly to the case $R$ is discrete, except we
require only that $m_{0,0}\in\mathfrak{m}$ in place of $m_{0,0}=0$;
this suffices to establish convergence in the proof of the homological
perturbation lemma. See \cite{twA8} for more details.}. This type of algebra was studied in \cite{twA8}, where a homological
perturbation lemma is proved. This allows us to classify all such
algebras on $C$, by transferring the structure to cohomology:
\[
HC=R\cdot\overbrace{1}^{0,0}\oplus R\cdot\overbrace{\pt}^{2m,1}.
\]
Note (\ref{eq:pairing}) induces a pairing on $HC$. A unital cyclic
twisted $G$-gapped $A_{\infty}$ algebra structure on $HC$ is specified
by a collection of operations 
\[
\left\{ \mu_{k,\beta}:HC^{\otimes k}\to HC\left[2-k-\mu\left(\beta\right),\mu\left(\beta\right)\mod2\right]\right\} .
\]
Following Fukaya \cite{fukaya-superpotential}, we define the superpotential\footnote{Note we do \emph{not }consider an inhomogeneous term, and below we
omit the Gromov-Witten invariants $I_{m}\left(0,\boldsymbol{l},\beta\right)$
with $k=0$, to streamline the presentation.}
\[
\Phi\left(\left\{ \mu_{k,\beta}\right\} \right):=\sum_{k\geq0,\beta\geq0}\frac{1}{k+1}\left\langle \mu_{k,\beta}\left(p_{0}^{L},\hdots,p_{0}^{L}\right),p_{0}^{L}\right\rangle q^{\beta}s^{k+1}.
\]
$\Phi\left(\left\{ \mu_{k,\beta}\right\} \right)\in s\cdot R\left[\left[q,s\right]\right]$
is supported in degree $3-2m$. 
\begin{prop}
\label{prop:rigidity}(a) The superpotential induces a bijection 
\[
\left\{ \text{unital cyclic }A_{\infty}\text{ structures }\left\{ \mu_{k,\beta}\right\} \right\} \to\left\{ \Phi\in s\cdot R\left[\left[q,s\right]\right]\big|\begin{array}{c}
\Phi\mod s^{2}\in\mathfrak{m}\text{ and}\\
\deg\Phi=3-2m
\end{array}\right\} .
\]
(b) if $\left\{ f_{k,\beta}\right\} :\left(HC,\left\{ \mu_{k,\beta}\right\} \right)\to\left(HC,\left\{ \mu_{k,\beta}'\right\} \right)$
is a unital cyclic isomorphism between two such algebras, then $f_{k,\beta}=\id$.
\end{prop}
\begin{proof}
This follows from elementary considerations which use the definition
of cyclic symmetry and unitality directly.
\end{proof}
In other words, the homotopy category of unital cyclic $A_{\infty}$
algebra structures on $HC$ is discrete and parameterized by the superpotential,
which is just the generating function for the nontrivial structure
constants. These structure constants are not subject to any constraints.
Here we consider the structure constants which are forced to vanish,
as well as ${m_{2,0}\left(p_{0}^{L},1\right)=m_{2,0}\left(1,p_{0}^{L}\right)=p_{0}^{L}}$,
to be trivial.

Consider the unital cyclic minimal model $\left(HC,\left\{ m_{k,\beta}^{can}\right\} \right)$,
which is homotopy equivalent to $\left(C,\left\{ m_{k,\beta}\right\} \right)$
through a unital cyclic morphism. We can also consider the unital
cyclic $A_{\infty}$ algebra $\left(HC,\left\{ m_{k,\beta}^{!}\right\} \right)$
whose structure constants are given by the open Gromov-Witten invariants,
so
\[
\Phi\left(HC,\left\{ m_{k,\beta}^{!}\right\} \right)=F_{OGW}^{s>0}=F_{OGW}-\left(F_{OGW}\mod s\right).
\]

\begin{prop}
\label{prop:unital isomorphism}There exists a unital (but not necessarily
cyclic) $A_{\infty}$ isomorphism
\[
\left(HC,m_{k,\beta}^{!}\right)\xrightarrow{f}\left(HC,m_{k,\beta}^{can}\right).
\]
\end{prop}
Before we prove this proposition, we make a few remarks. First, note
that $\left\{ m_{k,\beta}^{!}\right\} $ is well-defined independent
of all choices: invariance on choice of the forms $p_{0}^{L}$ and
$\mathcal{H}$ follows from Stokes' theorem (\ref{eq:stokes theorem}),
and the more subtle invariance on $\Lambda$ follows from the fixed-point
formula (this is Corollary \ref{cor:Lambda independence}).

The notion of Calabi-Yau structure \cite[\S 10.2]{k+sA8} should extend
to the twisted algebra case, and then Proposition \ref{prop:rigidity}
implies $\left\{ m_{k,\beta}^{can}\right\} $ is also well-defined
independent of all choices. The question of whether $m_{k,\beta}^{!}=m_{k,\beta}^{can}$
depends on whether we can represent a unital cyclic homotopy retraction
by a kernel $\Lambda$ on $\widetilde{L\times L}$, as we now explain.

Consider the inclusion $i:HC\to C$ and the projection $\pi:C\to HC$,
defined using the unit and point classes $1,p_{0}^{L}\in C$ and the
pairing $\left\langle -,-\right\rangle $ in the obvious way. In order
to transfer the unital cyclic structure to $HC$, we need an operator
$h:C\to C\left[-1\right]$ satisfying 

\begin{equation}
\left\langle h\,x,y\right\rangle =\left(-1\right)^{\codim x}\left\langle x,h\,y\right\rangle ,\label{eq:self adjoint}
\end{equation}
\begin{equation}
D\,h+h\,D=\id_{C}-i\circ\pi\text{ and}\label{eq:retraction}
\end{equation}
\begin{equation}
h^{2}=0\qquad h\,i=0\qquad\pi\,h=0.\label{eq:side conditions}
\end{equation}

By construction, $-\Lambda|_{S\left(N_{\Delta}\right)}$ is an equivariant
angular form for the sphere bundle of the normal bundle to the diagonal
$N_{\Delta}$, and $D\Lambda=\widetilde{\pr}_{2}^{*}p_{0}^{L}$, where
for $i=1,2$, $\widetilde{\pr}_{i}:\widetilde{L\times L}\to L\times L\to L$
denotes the projection. It follows that if we define ${h_{\Lambda}:C\to C\left[-1\right]}$
by
\begin{equation}
h_{\Lambda}\,x=\widetilde{\pr}_{2*}\left(\left(\Lambda+\tau^{*}\Lambda\right)\widetilde{\pr}_{1}^{*}x\right),\label{eq:smooth retraction}
\end{equation}
where $\tau:\widetilde{L\times L}\to\widetilde{L\times L}$ denotes
the swap, then $h=h_{\Lambda}$ satisfies (\ref{eq:self adjoint},\ref{eq:retraction}),
but we do not know if we can make it satisfy (\ref{eq:side conditions})
also. The modified operator 
\[
h=h_{\Lambda}\,D\,h_{\Lambda}
\]
satisfies all of the required conditions (\ref{eq:self adjoint},\ref{eq:retraction},\ref{eq:side conditions}),
and we use it to prove the existence of $\left\{ m_{k,\beta}^{can}\right\} $
(of course it may not be of the form (\ref{eq:smooth retraction})).

To some extent, Proposition \ref{prop:unital isomorphism} guided
our construction of the resolutions $\mm_{\basic}^{r}$, their orientations,
and the integration map. Thus, the reader may find that the proof
below explains and motivates the constructions discussed in $\S$\ref{subsec:Review-of-Resolutions}.
\begin{proof}
[Proof of Proposition \ref{prop:unital isomorphism} (Sketch)]Since
$h_{\Lambda}$ satisfies (\ref{eq:retraction}), we can use it to
construct a minimal model $\left(HC,\left\{ m_{k,\beta}'\right\} \right)$
for $\left(C,m_{k,\beta}\right)$. This is done as in \cite{fukayacyclic},
except we need to adjust the signs for the twisted case\footnote{Note without the side conditions (\ref{eq:side conditions}), we cannot
work with the geometric series representation of the bar coalgebra
differential, and must use summation over trees. Compare \cite[Remark 29]{twA8}.}. In fact, $\left(HC,\left\{ m_{k,\beta}'\right\} \right)$ is cyclic,
since $h_{\Lambda}$ is self-adjoint (\ref{eq:self adjoint}) (see
\cite[Lemma 10.3]{fukayacyclic}), but the equivalence
\[
\left(HC,\left\{ m_{k,\beta}^{'}\right\} \right)\xrightarrow{f'}\left(C,\left\{ m_{k,\beta}\right\} \right)
\]
may not respect the pairing, i.e. it may not be a cyclic morphism
in the sense of \cite[Definition 19]{twA8}. We have $\left\langle h_{\Lambda}1,x\right\rangle =\left\langle 1,h_{\Lambda}x\right\rangle =0$
since the De Rham degree of $h_{\Lambda}x$ is less than $2m$, so
$\left(HC,\left\{ m_{k,\beta}'\right\} \right)$ and $f'$ are unital.
It follows that the roof
\[
\left(HC,\left\{ m_{k,\beta}'\right\} \right)\xrightarrow{f'}\left(C,\left\{ m_{k,\beta}\right\} \right)\xleftarrow{f^{can}}\left(HC,\left\{ m_{k,\beta}^{can}\right\} \right)
\]
is represented by a unital isomorphism $f$. To complete the proof,
we must show that $m'_{k,\beta}=m_{k,\beta}^{!}$. It is enough to
check this for the non-trivial structure constants, so we will now
prove that for all $k,\beta\geq0$ with $\left(k+1\right)+\beta=1\mod2$
we have
\begin{equation}
\left\langle m_{k,\beta}'\left(p_{0}^{L},...,p_{0}^{L}\right),p_{0}^{L}\right\rangle =\left\langle m_{k,\beta}^{!}\left(p_{0}^{L},...,p_{0}^{L}\right),p_{0}^{L}\right\rangle .\label{eq:equal potential}
\end{equation}

Note that $m_{k,\beta}'$ is given by a sum over labeled ribbon trees
$\Gamma$ with $k+1$ external edges, one of which is called the root,
and whose vertices $\Gamma_{0}$ are labeled by integers $\beta_{v},l_{v}\geq0$
so that $\sum_{v\in\Gamma_{0}}\beta_{v}=\beta$ and $k_{v}+2l_{v}+3\beta_{v}\geq3$,
where $k_{v}$ is the valency of $v$. It follows from the projection
formula that the contribution $C_{\Gamma}$ of $\Gamma$ to $\left\langle m_{k,\beta}'\left(p_{0}^{L},...,p_{0}^{L}\right),p_{0}^{L}\right\rangle $
is given by an integral over the blow up of 
\[
\prod_{v\in\Gamma_{0}}\overline{\mm}_{0,k_{v},l_{v}}^{main}\left(\beta_{v}\right)
\]
(more precisely, over ${\left(\widetilde{L\times L}\right)^{\Gamma_{1}}\times_{\left(L\times L\right)^{\Gamma_{1}}}\prod_{v\in\Gamma_{0}}\overline{\mm}_{0,k_{v},l_{v}}^{main}\left(\beta_{v}\right)}$),
where $\overline{\mm}_{0,k_{v},l_{v}}^{main}\left(\beta_{v}\right)\subset\overline{\mm}_{0,k_{v},l_{v}}\left(\beta_{v}\right)$
is the component of the moduli space where the markings are required
to appear around the boundary of the disc in the cyclic order specified
by $\Gamma$. Since $\left(k+1\right)+\beta=1\mod2$, there's a natural
orientation on the edges $\Gamma_{1}$ so that $k_{v}^{\text{in}}+\beta_{v}=1\mod2$,
where $k_{v}^{\text{in}}$ denotes the number of incoming edges. The
contribution of trees $\Gamma$ with vertices which become unstable
when we forget the tails of edges cancel in pairs (this is just a
property of $m_{2,0}$, the analog of (\ref{eq:or rev involution})).
Fix some set $\underline{\Gamma}$ of ribbon trees which differ only
by the cyclic ordering of the edges near the vertices. We can think
of $\underline{\Gamma}$ as a non-ribbon labeled rooted tree. Pick
an arbitrary numbering of the $r=\left|\underline{\Gamma}_{1}\right|$
edges, an arbitrary partition of $\left\{ 1,...,l\right\} $ into
parts of sizes $\left\{ l_{v}\right\} $, and an arbitrary labeling
of the external edges by $\left\{ 1,...,k+1\right\} $ so that $1$
is the root. This gives a labeled tree $\tc\in\ts_{\basic}^{r}$,
$\basic=\left(k+1,l=\sum l_{v},\beta\right)$ in the sense of Definition
\ref{def:labeled tree}, and it is not hard to see that 
\[
\frac{1}{k+1}\cdot\sum_{\Gamma\in\underline{\Gamma}}C_{\Gamma}=\frac{1}{\left(k+1\right)!}\frac{1}{r!}\frac{1}{l!}\int_{\widetilde{\mm}_{\tc}}\tilde{\omega}_{r}\,\Lambda^{\boxtimes r},
\]
at least up to signs. Here $\tilde{\omega}_{r}$ is given by (\ref{eq:omega of ogw invts}),
and $\widetilde{\mm}_{\tc}\subset\widetilde{\mm}_{\basic}^{r}$ is
the clopen component corresponding to $\tc$, see $\S$\ref{subsec:Resolutions}.
The local system map $\jc_{\basic}^{r}$ was computed recursively
so that (\ref{eq:J_b^r coherence}) holds. Using this and its equivariance
w.r.t. the $\Sym\left(r\right)$ action (\ref{eq:J_s^rho equivariance})
and the behavior under permutation of boundary markings near a vertex
\cite[Eq (29), (31)]{equiv-OGW-invts}, we find that the signs are
in fact the same. Equation (\ref{eq:equal potential}) follows.
\end{proof}
We remark that originally, we used $\Phi\left(HC,\left\{ m_{k,\beta}^{can}\right\} \right)$
as the definition of open Gromov-Witten invariants, motivated by the
work of Solomon and Tukachinsky \cite{jake+sara} on the odd-dimensional
projective spaces. They define invariants for $\rr\pp^{2m+1}\subset\cc\pp^{2m+1}$
by constructing recursively a weak bounding cochain directly on $\left(C,\left\{ m_{k,\beta}\right\} \right)$.
One way to construct such a weak bounding cochain is to take the
image of the point class\footnote{strictly speaking, since a weak bounding cochain has to have degree
one, we need to take $b=p_{0}^{L}\cdot\epsilon^{-2m+1}$ where $\epsilon$
is a generator of the Novikov ring, see \cite{twA8}.} under the $A_{\infty}$ morphism $\left(HC,\left\{ m_{k,\beta}^{can}\right\} \right)\to\left(C,\left\{ m_{k,\beta}\right\} \right)$.
This way, the choice of operator $h$ fixes all the choices involved
in the recursive construction of the weak bounding cochain.

\subsection{Acknowledgments}

I am deeply grateful to my teacher, Jake Solomon, for suggesting the
problem and contributing important ideas, advice, and encouragement.
I'd also like to thank Mohammed Abouzaid, Rahul Pandharipande, Amit
Solomon, Ran Tessler, Lior Yanovski and Jingyu Zhao for engaging conversations
and useful suggestions. I was partially supported by ERC starting
grant 337560, ISF Grant 1747/13 and NSF grant DMS-1128155.

\section{\label{sec:statement}Related Theories}

\subsection{\label{subsec:RP^2m in CP^2m}Equivariant cohomology of $\rr\pp^{2m}\subset\cc\pp^{2m}$}

We set some notation. The unitary group $U\left(2m+1\right)$ acts
on $X=\cc\pp^{2m}$. The action of the subgroup of diagonal matrices,
$\tb^{\cc}=U\left(1\right)^{2m+1}<U\left(2m+1\right)$, is given by
\[
X=\cc\pp^{2m}=\pp_{\cc}\left(V_{0}\oplus V_{1}\oplus\cdots\oplus V_{2m}\right).
\]
where $\left(u_{0},...,u_{2m}\right)$ acts on $V_{i}$ by $z\mapsto u_{i}\cdot z$.
$\mathbb{T}^{\cc}$-equivariant cohomology is defined over the ring
\[
H_{\mathbb{T}^{\cc}}^{\bullet}=H_{\mathbb{T}^{\cc}}^{\bullet}\left(\mbox{pt}\right)=\rr\left[\alpha_{0},...,\alpha_{2m}\right],\;\deg\alpha_{i}=2.
\]
Let $\mathcal{H}=c_{1}\left(\mathcal{O}\left(1\right)\right)\in\Omega\left(X;\rr\left[\vec{\alpha}\right]\right)^{\tb^{\cc}}$
denote an equivariant form representing the Chern class of the dual
to the tautological line bundle, so
\begin{equation}
H_{\tb^{\cc}}^{\bullet}\left(X\right)=\rr\left[\left[\mathcal{H}\right],\alpha_{0},....,\alpha_{2m}\right]/\left(\prod_{i=0}^{2m}\left(\left[\mathcal{H}\right]-\alpha_{i}\right)\right).\label{eq:cohomology of CP^2m}
\end{equation}

Let $c:\cc^{2m+1}\to\cc^{2m+1}$ be the involution given by 
\begin{equation}
\left(z_{0},...,z_{2m}\right)\mapsto\left(\overline{z}_{0},\overline{z}_{2m},\overline{z}_{2m-1},...,\overline{z}_{1}\right).\label{eq:affine conjugation}
\end{equation}
It induces involutions $c_{X}:X\to X$  and ${c_{U}:U\left(2m+1\right)\to U\left(2m+1\right)}$,
the latter given by conjugation.We define $O\left(2m+1\right)$ and
$\tb$ to be the $c_{U}$-fixed subgroups of $U\left(2m+1\right)$
and $\tb^{\cc}$, respectively. We identify $U\left(1\right)^{m}\simeq\tb$
by 
\[
U\left(1\right)^{m}\ni\left(u_{1},...,u_{m}\right)\mapsto\mbox{diag}\left(1,u_{1},....,u_{m},\overline{u}_{m},...,\overline{u}_{1}\right)\in\tb<\tb^{\cc}<U\left(2m+1\right)
\]
so that the homomorphism $\tb\to\tb^{\cc}$ corresponds to the map
\begin{equation}
H_{\tb^{\cc}}^{\bullet}=\rr\left[\alpha_{0},...,\alpha_{2m}\right]\overset{\rho_{\tb}}{\longrightarrow}H_{\tb}^{\bullet}=\rr\left[\lambda_{1},...,\lambda_{m}\right]\label{eq:weight loss}
\end{equation}
given by $\alpha_{0}\mapsto0$, $\alpha_{i}\mapsto\lambda_{i}$ and
$\alpha_{2m+1-i}\mapsto-\lambda_{i}$, for $1\leq i\leq m$. Let $W_{0}=\rr$
denote the trivial representation of $\tb$, and let $W_{i}=\rr^{2}$
denote the real representations of $\tb$, on which $\left(e^{\sqrt{-1}t_{1}},...,e^{\sqrt{-1}t_{m}}\right)$
acts by 
\[
\begin{pmatrix}\cos t_{i} & -\sin t_{i}\\
\sin t_{i} & \cos t_{i}
\end{pmatrix}.
\]
As $\zz/2\times\tb$ representations we have 
\[
V_{0}\simeq\cc\otimes W_{0}
\]
and for $1\leq i\leq m$,
\[
V_{i}\oplus V_{2m+1-i}\simeq\cc\otimes W_{i},
\]
where the $\zz/2$ action is the usual conjugation action on the $\cc$
factor. We have
\[
L=X^{\zz/2}=\pp_{\rr}\left(W_{0}\oplus W_{1}\oplus\cdots\oplus W_{m}\right).
\]
The induced action of $O\left(2m+1\right)$ on $L$ is transitive.

We identify ${H_{2}\left(X;\zz\right)=\zz}$ using the complex structure
and  ${H_{2}\left(X,L\right)=\zz}$ so that $H_{2}\left(X\right)\to H_{2}\left(X,L\right)$
corresponds to multiplication by $+2$.

We let 
\begin{equation}
\mathcal{R}=S^{-1}\rr\left[\lambda_{1},...,\lambda_{m}\right]\label{eq:graded rational functions}
\end{equation}
(respectively, $\mathcal{R}^{\cc}=S_{\cc}^{-1}\rr\left[\alpha_{0},....,\alpha_{2m}\right]$)
denote the localization of $H_{\tb}^{\bullet}$ (resp., $H_{\tb^{\cc}}^{\bullet}$)
by the multiplicative subset of nonzero homogeneous polynomials.

\subsection{\label{subsec:Constraint-correlators}Constraint correlators}

The information in the \emph{closed }Gromov-Witten theory of $\cc\pp^{2m}$
relevant for the computation of the equivariant OGW invariants is
encapsulated by \emph{constraint correlators}. These are generating
functions fashioned after the quantum correlators used by Givental
\cite{givental}. Our approach follows \cite{panda->givental}, with
two main differences. First, we allow any number of constrained marked
points, by introducing a formal variable $\eta_{i}$ that records
the number of marked points carrying $\mathcal{H}^{i}$. Second, although
the recursion specifying the correlators is computed over $\mathcal{R}^{\cc}$
to avoid non-rigid maps (see Remark \ref{rem:moving fixed points}),
we work with partially localized rings that admit an extension of
the ring homomorphism $\rr\left[\alpha_{0},...,\alpha_{2m}\right]\xrightarrow{\rho_{\tb}}\rr\left[\lambda_{1},...,\lambda_{m}\right]$.

Consider the two rings

\[
R_{\text{rat}}^{\cc}=M^{-1}\rr\left[\alpha_{0},...,\alpha_{2m},\overbrace{\hbar}^{2}\right]\left[\left[\overbrace{Q}^{4m+2},\eta_{0},...,\overbrace{\eta_{i}}^{2-2i},...,\eta_{2m}\right]\right],
\]

\[
R_{\text{an}}^{\cc}=\rr\left[\alpha_{0},...,\alpha_{2m},\hbar\right]\left[\left[Q,\boldsymbol{\eta},\hbar^{-1}\right]\right].
\]
Here $M\subset\rr\left[\boldsymbol{\alpha},\hbar\right]$ is the multiplicative
subset generated by $\hbar$ and $\hbar+\frac{\alpha_{i}-\alpha_{j}}{n}$
for all $i\neq j$ and $n\geq1$. We think of $R_{\text{rat}}^{\cc}$
as a subring of $R_{\text{an}}^{\cc}$ using Taylor expansion in powers
of $\hbar^{-1}$. Recall $S\subset\rr\left[\boldsymbol{\alpha}\right]$
denotes the multiplicative subset of homogeneous nonzero polynomials.
For $i\neq j$ and $n\geq1$, the map
\[
\hbar\mapsto\frac{\alpha_{i}-\alpha_{j}}{n}
\]
extends to an injective, degree-preserving ring homomorphisms $S^{-1}R_{\text{rat}}^{\cc}\to\mathcal{R}^{\cc}\left[\left[Q,\boldsymbol{\eta}\right]\right]$.
This is the motivation for considering the subring $R_{\text{rat}}^{\cc}\subset R_{\text{an}}^{\cc}$.

For non-negative integers $d,n$ with $d+n\geq3$ we let 
\[
\overline{\mm}_{0,n}\left(X,d\right)
\]
denote the moduli space parameterizing stable maps to $X$ of genus
0 and degree $d$, with $n$ marked points. $\tb^{\cc}$ acts on $\overline{\mm}_{0,n}\left(d\right)$
by translating maps. 

By abuse of notation, for $0\leq a\leq2m$ we let $p_{a}=\pp_{\cc}\left(V_{a}\right)\in X$
denote the $\tb^{\cc}$-fixed-point, as well as the equivariant Poincare
dual form
\begin{equation}
p_{a}=\prod_{i\neq a}\left(\mathcal{H}-\alpha_{i}\right)\label{eq:equivariant point class}
\end{equation}

\begin{defn}
For $0\leq a\leq2m$ the \emph{constraint correlator }$Z_{a}^{\cc}\left(\hbar\right)\in R_{\text{an}}^{\cc}$
is defined by 

\begin{equation}
Z_{a}^{\cc}\left(\hbar\right)=Z_{a}^{0}\left(\hbar\right)+\sum_{d>0}Q^{d}\sum_{n\geq0}\int_{\overline{\mm}_{0,n+1}\left(X,d\right)}\frac{1}{n!}\prod_{i=1}^{n}\left(ev_{i}^{*}\left(\sum_{j=0}^{2m}\eta_{j}\mathcal{H}^{j}\right)\right)\cdot\frac{\ev{}_{n+1}^{*}p_{a}}{\hbar-\psi_{n+1}}\label{eq:correlator def}
\end{equation}
where $\psi_{n+1}=c_{1}^{\tb^{\cc}}\left(\mathbb{L}_{n+1}\right)$,
the first Chern class of the cotangent line to the last marked point,
and we set
\begin{equation}
Z_{a}^{0}\left(\hbar\right)=\sum_{n\geq0}\hbar^{1-n}\frac{1}{n!}\left(\sum_{j=0}^{2m}\eta_{j}\alpha_{a}^{j}\right)^{n}.\label{eq:Z^0}
\end{equation}
\end{defn}

We use closed localization to compute the image of $Z_{a}^{\cc}\left(\hbar\right)$
in $S^{-1}R_{\text{an}}^{\cc}$. 
\begin{prop}
\label{prop:correlators recursion}We have $Z_{a}^{\cc}\left(\hbar\right)\in S^{-1}R_{\text{rat}}^{\cc}\subset S^{-1}R_{\text{an}}^{\cc}$,
and the following equations hold for $0\leq a\leq2m$.

\begin{multline}
Z_{a}^{\cc}\left(\hbar\right)=Z_{a}^{0}\left(\hbar\right)+\\
+\sum\frac{1}{n!}\left(\sum_{i=0}^{2m}\eta_{i}\alpha_{a}^{i}\right)^{n}\prod_{b\neq a}\left(\alpha_{b}-\alpha_{a}\right)^{s}\left(\frac{\hbar^{-1}}{s!}\prod_{r=1}^{s}\frac{d_{r}}{\alpha_{a}-\alpha_{j_{r}}}\right)\left(\frac{d_{1}}{\alpha_{a}-\alpha_{j_{1}}}+\cdots+\frac{d_{s}}{\alpha_{a}-\alpha_{j_{s}}}+\hbar^{-1}\right)^{s+n-2}\times\\
\times\prod_{r=1}^{s}\frac{\left(-1\right)^{d_{r}}d_{r}^{2d_{r}-1}Q^{d_{r}}}{\left(d_{r}!\right)^{2}\left(\alpha_{a}-\alpha_{j_{r}}\right)^{2d_{r}}\prod_{\begin{array}{c}
0\leq w\leq d_{r}\\
k\neq j_{r},a
\end{array}}\left(\frac{w}{d_{r}}\alpha_{a}+\frac{d_{r}-w}{d_{r}}\alpha_{j_{r}}-\alpha_{k}\right)}Z_{j_{r}}^{\cc}\left(\frac{\alpha_{j_{r}}-\alpha_{a}}{d_{r}}\right)+\\
+\sum_{j\neq a,d>0}\frac{\prod_{b\neq a}\left(\alpha_{a}-\alpha_{b}\right)}{\hbar+\frac{\alpha_{a}-\alpha_{j}}{d}}\cdot\frac{\left(-1\right)^{d}d^{2d-1}Q^{d}}{\left(d!\right)^{2}\left(\alpha_{a}-\alpha_{j_{r}}\right)^{2d}\prod_{\begin{array}{c}
0\leq w\leq d\\
k\neq j,a
\end{array}}\left(\frac{w}{d}\alpha_{a}+\frac{d-w}{d}\alpha_{j}-\alpha_{k}\right)}Z_{j}^{\cc}\left(\frac{\alpha_{j}-\alpha_{i}}{d}\right).\label{eq:closed recursion}
\end{multline}
where the undecorated sum ranges over $s,n\in\zz_{\geq0}$ satisfying
$s+n+1\geq3$, and over pairs of $s$-tuples $\left(j_{1},...,j_{s}\right)\in\left\{ 0,...,\hat{a},...,2m\right\} ^{s}$
and $\left(d_{1},...,d_{s}\right)\in\zz_{>0}^{s}$. 

In other words, $Z^{\cc}=\left(Z_{a}^{\cc}\left(\hbar\right)\right)_{a=0}^{2m}$
is the unique fixed point of a map 
\[
f:\left(\mathcal{R}_{\hbar}^{\cc}\left[\left[Q,\eta_{0},...,\eta_{2m}\right]\right]\right)^{2m+1}\to\left(\mathcal{R}_{\hbar}^{\cc}\left[\left[Q,\eta_{0},...,\eta_{2m}\right]\right]\right)^{2m+1}
\]
which is $S^{-1}R_{\text{rat}}^{\cc}\left\langle Q,\boldsymbol{\eta}\right\rangle $-adically
contracting, so $Z^{\cc}=\lim_{n\to\infty}f^{n}\left(\mathbf{0}\right)$.
\end{prop}
\begin{proof}
This is a simple variation of the arguments given in \cite[Section 3]{panda->givental},
to accommodate additional marked points with constraints.
\end{proof}
In practice, using the recursion $f$ is inefficient. One can use
$\Sym\left(s\right)$ invariance to reduce the computation of a coefficient
of $Z_{a}^{\cc}$ to a sum over certain isomorphism types of labeled
trees, in a straightforward manner. Of course, the same sum can be
obtained by directly applying the fixed-point formula to (\ref{eq:correlator def});
we're just using algebra to concisely summarize some well-known computations,
see for example \cite[Chapter 27]{clay-MS}.

The following result says that $Z_{a}^{\cc}\left(\hbar\right)$ is
integral, so we can change coefficients $\rr\left[\vect\alpha\right]\xrightarrow{\rho_{\tb}}\rr\left[\vect\lambda\right]$.
\begin{lem}
We have $Z_{a}^{\cc}\left(\hbar\right)\in R_{\text{rat}}^{\cc}\subset S^{-1}R_{\text{rat}}^{\cc}$.
\end{lem}
\begin{proof}
Fix some monomial $Q^{d}\eta_{0}^{r_{0}}\cdots\eta_{2m}^{r_{2m}}$
and consider the coefficient $C$ of this monomial in $Z_{a}^{\cc}\left(\hbar\right)$.
By definition, $C\in\rr\left[\vect\alpha,\hbar\right]\left[\left[\hbar^{-1}\right]\right]$,
and by Proposition \ref{prop:correlators recursion} it is in $S^{-1}M^{-1}\rr\left[\vect\alpha,\hbar\right]$.
We thus reduce to proving that the intersection of these two rings
in $S^{-1}\rr\left[\vect\alpha,\hbar\right]\left[\left[\hbar^{-1}\right]\right]$
satisfies
\[
M^{-1}\rr\left[\vect\alpha,\hbar\right]=S^{-1}M^{-1}\rr\left[\vect\alpha,\hbar\right]\cap\rr\left[\vect\alpha,\hbar\right]\left[\left[\hbar^{-1}\right]\right].
\]
Indeed, the right-hand-side consists of elements $C\in\rr\left[\vect\alpha,\hbar\right]\left[\left[\hbar^{-1}\right]\right]$
such that there exists some $q\in M$ so that $p:=q\cdot C\in S^{-1}\rr\left[\vect\alpha,\hbar\right]$;
but then $p\in\rr\left[\vect\alpha,\hbar\right]$ and $C=p/q\in M^{-1}\rr\left[\vect\alpha,\hbar\right]$.
\end{proof}
Consider

\[
R_{\text{rat}}=\check{M}^{-1}\rr\left[\lambda_{1},...,\lambda_{m},\hbar\right]\left[\left[\overbrace{q}^{2m+1},\boldsymbol{\eta}\right]\right]
\]
where $\check{M}^{-1}\subset\rr\left[\boldsymbol{\lambda},\hbar\right]$
is the multiplicative subset generated by $\hbar+\rho_{\tb}\left(\frac{\alpha_{i}-\alpha_{j}}{n}\right)$
for all $i\neq j$ and $n\geq1$ (i.e., by $\hbar,\hbar\pm\frac{\lambda_{i}}{n},\hbar\pm\frac{2\lambda_{i}}{n}$
and $\hbar\pm\frac{\lambda_{i}\pm\lambda_{j}}{n}$ for all $i\neq j$
and $n\geq1$). We denote by
\[
R_{\text{rat}}^{\cc}\xrightarrow{\widehat{\rho}_{\tb}}R_{\text{rat}}
\]
the unique extension of $\rho_{\tb}:\rr\left[\boldsymbol{\alpha}\right]\to\rr\left[\boldsymbol{\lambda}\right]$
that sends $Q\mapsto q^{2},\eta_{i}\mapsto\eta_{i}$ and $\hbar\mapsto\hbar$. 
\begin{defn}
We define 
\[
Z_{a}\left(\hbar\right)=\hat{\rho}_{\tb}\left(Z_{a}^{\cc}\left(\hbar\right)\right).
\]
\end{defn}
\begin{rem}
\label{rem:moving fixed points}Theorem \ref{thm:open localization}
uses $Z_{a}\left(\hbar\right)$, but the full $\tb^{\cc}$ action
is useful for the computation of the correlators $Z_{a}\left(\hbar\right)$,
as Proposition \ref{prop:correlators recursion} shows.

Although $Z_{a}^{\cc}\left(\hbar\right)\in R_{\text{rat}}^{\cc}$
so we can make sense of $\hat{\rho}_{\tb}\left(Z_{a}^{\cc}\left(\hbar\right)\right)$,
individual summands in the recursion (\ref{eq:closed recursion})
will not in general lie in $R_{\text{rat}}^{\cc}$. Indeed, we see
that some of the recursion denominators may vanish if we apply $\hat{\rho}_{\tb}$;
also, the substitution $Z_{a}\left(\rho_{\tb}\left(\frac{\alpha_{i}-\alpha_{j}}{n}\right)\right)$
need not be well-defined.

On the other hand, the $\hbar^{-1}$-Taylor expansion of $Z_{a}\left(\hbar\right)$
\emph{is} given by a sum of equivariant $\tb$-integrals, obtained
by applying $\rho_{\tb}$ to the integrand of (\ref{eq:correlator def}),
and one can compute these integrals directly using $\tb$ localization.
What is happening is that different $\tb^{\cc}$-fixed-point components
of $\overline{\mm}_{0,n+1}\left(X,d\right)$ may be part of a single
fixed-point component for the $\tb$-action. For a prototypical example,
let $F_{1}$ and $F_{2}$ be two isolated $\tb^{\cc}$ fixed-points
in $\overline{\mm}_{0,2}\left(\cc\pp^{2},2\right)$. $F_{1}$ is represented
by a holomorphic map $\cc\pp^{1}\to\cc\pp^{2}$ which is a degree
2 cover of the line $\overline{p_{1}p_{2}}$, branched over $p_{1}$
and $p_{2}$. $F_{2}$ is represented by an injective holomorphic
map $\Sigma\to\cc\pp^{2}$ where $\Sigma$ has two irreducible components
covering the lines $\overline{p_{1}p_{0}}$ and $\overline{p_{0}p_{2}}$,
respectively. For both $F_{1},F_{2}$ we let the marked points $1$
and $2$ map to $p_{1}$ and $p_{2}$, respectively. There is a single
$\tb$-fixed-point component $F_{12}\subset\overline{\mm}_{0,2}\left(\cc\pp^{2},2\right)$
containing $F_{1}$ and $F_{2}$..

Consider some equivariant integral $\int_{\overline{\mm}_{0,2}\left(\cc\pp^{2},2\right)}\omega$.
If we try to apply $\rho_{\tb}$ to $C_{1}=\int_{F_{1}}^{\tb^{\cc}}\frac{\omega}{E_{F_{1}}}$,
the $\tb^{\cc}$-localization contribution of $F_{1}$, we see that
$\rho_{\tb}$ sends the Euler factor in the denominator corresponding
to the tangent direction $T_{F_{1}}F_{12}\subset T_{F_{1}}\overline{\mm}_{0,2}\left(\cc\pp^{2},2\right)$
to zero, and a similar singularity appears in the contribution $C_{2}$
of $F_{2}$. On the other hand, $\rho_{\tb}\left(C_{1}+C_{2}\right)$
\emph{is} well-defined, since $C_{1}+C_{2}$ can be interpreted as
applying $\tb^{\cc}$ localization to evaluate $\int_{F_{12}}^{\tb}\frac{\omega}{E_{F_{12}}}$,
the $\tb$-fixed-point contribution of $F_{12}$.
\end{rem}
\begin{lem}
\label{lem:Z_a sub well-defined}$Z_{a}\left(\frac{\rho_{\tb}\left(\alpha_{a}\right)}{d}\right)$
is a well-defined element of $\check{M}^{-1}\rr\left[\vect\lambda,\hbar\right]$
for all $a\neq0,d\geq1$. 
\end{lem}
\begin{proof}
This follows from inspection of the denominators in Proposition \ref{prop:correlators recursion}.
\end{proof}

\subsection{\label{subsec:descendent integrals}Descendent integrals of discs}

We now review some aspects of the open intersection theory of discs,
as defined and computed in \cite{JRR}. We reformulate the results
in terms of differential forms rather than multisections.

Let $\kf$ and $\lf$ be finite \emph{sets}, such that $\left|\kf\right|+2\left|\lf\right|\geq3$.
Let $\overline{\mm}_{0,\kf,\lf}$ denote the moduli space of stable
discs with boundary and interior marked points, in bijection with
$\kf$ and $\lf$, respectively. It is a smooth manifold with corners.
For $i\in\lf$ let $\mathbb{L}_{i}\to\overline{\mm}_{0,\kf,\lf}$
denote the cotangent space to the interior point marked $i$; $\mathbb{L}_{i}$
is a complex vector bundle of rank one. Consider 
\[
E=\bigoplus_{i\in\lf}\mathbb{L}_{i}^{\oplus a_{i}}
\]
for some $a_{i}\geq0$. It is a rank $n=2\cdot\left(\sum_{i\in\lf}a_{i}\right)$
vector bundle, canonically oriented by the complex orientation. Assume
that 
\begin{equation}
n=\dim\overline{\mm}_{0,\kf,\lf}=2\left|\lf\right|+\left|\kf\right|-3\label{eq:rank=00003Ddim}
\end{equation}
Let $\pi:S\left(E\right)\to\overline{\mm}_{0,\kf,\lf}$ denote the
associated sphere bundle. Recall (see \cite{bott+tu}) a form $\theta\in\Omega^{n-1}\left(S\left(E\right)\right)$
is called \emph{an angular form} \emph{for $E$ }if $\pi_{*}\theta\equiv1$
and $d\theta=-\pi^{*}e$ for some form $e\in\Omega^{n}\left(\overline{\mm}_{0,\kf,\lf}\right)$,
called the \emph{Euler form associated with $\theta$. }

We now introduce boundary conditions on $\theta$. We have
\[
\partial\overline{\mm}_{0,\kf,\lf}=\coprod\overline{\mm}_{0,\kf'\coprod\left\{ \star'\right\} ,\lf'}\times\overline{\mm}_{0,\kf''\coprod\left\{ \star''\right\} ,\lf''}
\]
where the disjoint union ranges over all pairs of partitions $\kf=\kf'\coprod\kf''$
and $\lf=\lf'\coprod\lf''$; such that $\left|\kf'\right|+1+2\left|\lf'\right|\geq3$,
$\left|\kf''\right|+1+2\left|\lf''\right|\geq3$, and $\left|\kf'\right|$
is odd. We call a pair of partitions as above a \emph{boundary specification}.

Fix a boundary specification and let $B=\overline{\mm}_{0,\kf'\coprod\left\{ \star'\right\} ,\lf'}\times\overline{\mm}_{0,\kf''\coprod\left\{ \star''\right\} ,\lf''}$
denote the corresponding clopen component of $\partial\overline{\mm}_{0,\kf,\lf}$.
Since $\left|\kf'\right|$ is odd we must have $\left|\kf'\right|+1+2\left|\lf'\right|\geq4$
so there's a well-defined forgetful map 
\begin{equation}
B=\overline{\mm}_{0,\kf'\coprod\left\{ \star'\right\} ,\lf'}\times\overline{\mm}_{0,\kf''\coprod\left\{ \star''\right\} ,\lf''}\overset{\mbox{For}_{B}}{\longrightarrow}B_{-}=\overline{\mm}_{0,\kf',\lf'}\times\overline{\mm}_{0,\kf''\coprod\left\{ \star''\right\} ,\lf''}.\label{eq:jrr forgetful}
\end{equation}
For $x\in\lf'$ (respectively, $x\in\lf''$) we denote by $\mathbb{L}'_{x}$
(resp. $\mathbb{L}_{x}''$) the pullback to $B_{-}$ of the corresponding
cotangent line on $\overline{\mm}_{0,\kf',\lf'}$ (resp. $\overline{\mm}_{0,\kf''\coprod\star'',\lf''}$).
We define a complex vector bundle on $B_{-}$ by 
\begin{equation}
E_{-}=\bigoplus_{x\in l'}\left(\mathbb{L}'_{x}\right)^{\oplus a_{x}}\oplus\bigoplus_{x\in l''}\left(\mathbb{L}''_{x}\right)^{\oplus a_{x}}.\label{eq:cotangent direct sum}
\end{equation}
There's a natural isomorphism
\[
E|_{B}\to\mbox{For}_{B}^{*}E_{-}.
\]
Taking the disjoint union over all $B$ we obtain a cartesian map
of vector bundles 
\begin{equation}
\partial E\xrightarrow{f}E'\mbox{ over }\partial\mbox{\ensuremath{\overline{\mm}}}_{0,\kf,\lf}\to\mbox{\ensuremath{\overline{\mm}}}_{0,\kf,\lf}'\label{eq:descendent bdry forget}
\end{equation}
with $\dim\overline{\mm}'_{0,\kf,\lf}<\dim\partial\overline{\mm}_{0,\kf,\lf}$.
\begin{defn}
An angular form $\theta\in\Omega^{n-1}\left(S\left(E\right)\right)$
will be called \emph{canonical }if there exists some $\theta'\in\Omega^{n-1}\left(S\left(E'\right)\right)$
such that 
\[
\left(i_{\overline{\mm}_{0,k,l}}^{\partial}\right)^{*}\theta=f^{*}\theta'.
\]
\end{defn}
\begin{claim}
\label{claim:canangular form existence}A canonical angular form exists.
\end{claim}
\begin{proof}
This is similar to the proof of Proposition \ref{prop:splitting construction},
and is omitted. The corresponding statement for multisections is proven
in \cite{JRR}.
\end{proof}
\begin{defn}
\label{def:open descendent integrals}Let $k,l$ and $a_{1},...,a_{l}$
be non-negative integers. If (\ref{eq:rank=00003Ddim}) holds, the
\emph{open descendent integral }is defined by 
\begin{equation}
\left\langle \tau_{a_{1}}\cdots\tau_{a_{l}}\sigma^{k}\right\rangle _{0}^{o}=2^{-\frac{k-1}{2}}\int_{\overline{\mm}_{0,\left[k\right],\left[l\right]}}e\label{eq:open descendents def}
\end{equation}
where $e$ is any Euler form associated with a canonical angular form
$\theta\in\Omega^{n-1}\left(S\left(E\right)\right)$. Otherwise, we
set $\left\langle \tau_{a_{1}}\cdots\tau_{a_{l}}\sigma^{k}\right\rangle _{0}^{o}=0$.
\end{defn}
The language of multisections was developed in \cite{branched-man},
and used in \cite{JRR}. The following lemma allows us to translate
results to the language of differential forms.
\begin{lem}
\label{lem:zeros of a multisection}Let $\pi:E\to M$ be an oriented
rank $n$ vector bundle over a compact oriented manifold with corners
of the same dimension $n$. Let $\sigma$ be a multi-valued section
of $E$, transverse to the zero section. Let $\theta$ be an angular
form for $E$ with corresponding Euler form $e$. Let $\#\mathfrak{Z}\left(s\right)$
denote the weighted signed count of the oriented zero set $\mathfrak{Z}\left(s\right)$
of $s$. We have 
\[
\#\mathfrak{Z}\left(s\right)=\int_{M}e-\int_{\partial M}\sigma^{*}\theta
\]
\end{lem}
\begin{proof}
Similar to the proof of \cite[Theorem 11.16]{bott+tu}.
\end{proof}
\begin{cor}
\label{cor:computation of open descendents}Definition \ref{def:open descendent integrals}
agrees with \cite[Definition 3.4]{JRR}, given in terms of multisections.
In particular, the open descendent integrals are integers uniquely
specified by the following two relations.

(i)
\[
\left\langle \tau_{a_{1}}\cdots\tau_{a_{l}}\sigma^{k}\right\rangle _{0}^{o}=\sum_{j}\left\langle \tau_{a_{j}-1}\prod_{i\neq j}\tau_{a_{i}}\sigma^{k}\right\rangle _{0}^{o}.
\]
(ii) In case $a_{i}\geq1$ for all $i$,
\[
\left\langle \tau_{a_{1}}\cdots\tau_{a_{l}}\sigma^{k}\right\rangle _{0}^{o}=\frac{\left(1+\sum_{i=1}^{l}\left(2a_{i}-1\right)\right)!}{\prod_{i=1}^{l}\left(2a_{i}-1\right)!!},
\]
 where for an odd integer $r$ we set $r!!=r\cdot\left(r-2\right)\cdot...\cdot1$.
\end{cor}
\begin{proof}
By Lemma \ref{lem:zeros of a multisection}, to establish the first
claim it suffices to show that if $\sigma$ is a canonical multisection
(see \cite[Definition 3.1]{JRR}) we have 
\[
\int_{\partial\overline{\mm}_{0,\left[k\right],\left[l\right]}}\sigma^{*}\theta=0.
\]
This is obvious, since for each $B=\overline{\mm}_{0,\kf'\coprod\left\{ \star'\right\} ,\lf'}\times\overline{\mm}_{0,\kf''\coprod\left\{ \star''\right\} ,\lf''}$,
$\left(\sigma^{*}\theta\right)|_{B}$ is pulled back from a lower
dimensional space, namely $B_{-}$. The relations are quoted from
\cite[Theorems 1.2, 1.4]{JRR}.
\end{proof}

\begin{rem}
In \cite{JRR}, the open descendent integrals are shown to satisfy
a plethora of relations. Fixed point localization in $\left(\cc\pp^{2m},\rr\pp^{2m}\right)$
produces many additional relations, involving also the closed invariants.
In fact, the relations from $\left(\cc\pp^{2},\rr\pp^{2}\right)$
can be used to produce a recursive algorithm for computing the open
genus zero descendent integrals, independently of Corollary \ref{cor:computation of open descendents}.
This is discussed in \cite{descendents+loc}.
\end{rem}
\begin{defn}
Let $s_{0},t_{0},t_{1},...$ be formal variables. Write $\gamma=\sum_{i=0}^{\infty}t_{i}\tau_{i}$
and $\delta=s_{0}\sigma$. Define the (open, genus zero) \emph{free
energy function }by 
\begin{equation}
F_{0}^{o}\left(s_{0},t_{0},t_{1},...\right)=\sum_{n=0}^{\infty}\frac{2^{\frac{k-1}{2}}\left\langle \gamma^{n}\delta^{k}\right\rangle _{0}^{o}}{n!k!}.\label{eq:descendent generating func}
\end{equation}
\end{defn}

\section{\label{sec:fp for generalized}Fixed-point Formula for Extended Forms}

In this section we state the fixed-point formula for extended forms,
Theorem \ref{thm:general fixed point formula}. This is more general
and more computationally effective than Theorem \ref{thm:open localization},
but requires more geometric preparation. After stating the theorem
in $\S$\ref{subsec:fp for generalized forms}, we will give explicit
formula for the fixed-point contributions to the equivariant open
Gromov-Witten invariants $\S$\ref{subsec:Equivariant-open-Gromov-Witten},
and compute a few examples $\S$\ref{subsec:Examples-and-applications.}.
We will conclude this section with a proof of Theorem \ref{thm:open localization}
using Theorem \ref{thm:general fixed point formula}. The proof of
Theorem will be taken up in the next sections.

\subsection{\label{subsec:Review-of-Resolutions}Review of resolutions}

We briefly review some notation and construction from \cite[\S 2]{equiv-OGW-invts}.

\subsubsection{Moduli spaces}

We let $\nn=\left\{ 1,2,...\right\} $. For any $S\subset\nn$, $\sstar'_{S}=\left\{ \sstar'_{i}\right\} _{i\in S}$,
similarly for $\sstar''_{S}$ and $\star'_{S},\star''_{S},\star'''_{S},...$ 

Hollow stars will be used to denote markings related to boundary nodes,
and solid stars will be used for interior nodes. 
\begin{defn}
\label{def:moduli specifications}A \emph{pre-moduli specification}
$\mathfrak{b}$ is a 3-tuple $\left(\kf,\lf,\beta\right)$ where 

\begin{itemize}
\item $\kf\subset\nn\coprod\sstar''_{\nn}$ and $\lf\subset\nn$ are finite
subsets. Elements of $\kf$ and of $\lf$ are called \emph{orienting
labels }and \emph{interior labels, }respectively.
\item $\beta$ is a non-negative integer, \emph{the degree}, which we think
of as an element of $H_{2}\left(X,L\right)$.
\end{itemize}
A \emph{basic moduli specification }is a pre-moduli specification
$\mathfrak{b}=\left(\kf,\lf,\beta\right)$ that is 

\begin{itemize}
\item \emph{stable,} meaning $k+2l+3\beta\geq3$; henceforth we use standard
Roman letters to denote the sizes of sets labeled by the corresponding
Serif letters, so $k=\left|\kf\right|$ and $l=\left|\lf\right|$.
\item \emph{orientable}, meaning 
\begin{equation}
k+\beta=1\mod2.\label{eq:orientability of moduli specification}
\end{equation}
\end{itemize}
A \emph{moduli specification }$\sfr$ is a pair $\sfr=\left(\mathfrak{b},\sigma\right)$
where $\mathfrak{b}=\left(\kf,\lf,\beta\right)$ is an orientable
pre-moduli specification and $\sigma\subset\sstar'_{\nn}$ is a finite
subset such that $k+\left|\sigma\right|+2l+3\beta\geq3$. We call
$\sigma$ the \emph{superfluous (boundary) labels,} and $\tilde{\kf}=\kf\coprod\sigma$
the \emph{boundary labels}. 

A\emph{ }moduli specification\emph{ }$\sfr=\left(\mathfrak{b},\sigma\right)$
is called \emph{sturdy }if $\mathfrak{b}$ is stable and \emph{wobbly
}otherwise.
\end{defn}
Let $\sfr=\left(\mathfrak{b},\sigma\right)$ be a moduli specification.
If $\sfr$ is sturdy, then $\mathfrak{b}$ is a basic moduli specification;
if it is wobbly, it is necessarily of the form 
\begin{equation}
\left(\left(\kf,\emptyset,0\right),\sigma\right)\text{ with \ensuremath{\left|\kf\right|}=1 and \ensuremath{\left|\sigma\right|\geq}2.}\label{eq:wobbly spec}
\end{equation}
Either way, the \emph{combined moduli specification }$\tilde{\sfr}:=\left(\tilde{\kf},\lf,\beta\right)$
is a basic moduli specification.

A 3-tuple of non-negative integers $b=\left(k,l,\beta\right)$ with
$k+2l+3\beta\geq3$ and $k+\beta=1\mod2$ may be used in place of
a basic moduli specification, taking $\basic=\left(\left[k\right],\left[l\right],\beta\right)$
where we denote $\left[k\right]=\left\{ 1,2,...,k\right\} $.

Let $\tb=U\left(1\right)^{m}$ and $\tb_{\basic}=\tb\times\Sym\left(\kf\right)\times\Sym\left(\lf\right)$.
If $\basic$ is a basic moduli specification, the moduli space of
stable disc maps

\[
\mm_{\basic}=\overline{\mm}_{0,\kf,\lf}\left(\beta\right)
\]
is a $\tb_{\basic}$-orbifold with corners. See the appendix for what
we mean by this and related notions. Henceforth, all orbifolds will
have corners unless explicitly mentioned otherwise. 

The construction of $\mm_{\basic}$ is summarized in the following
diagram (see \cite[\S 2.1]{equiv-OGW-invts}): 
\begin{equation}
\mm_{\basic}\overset{s}{\hookrightarrow}\mm_{\basic}^{::}\xrightarrow{o}\widetilde{\mm}_{\lf^{\cc},\beta}^{\zz/2}\xrightarrow{B}\mm_{\lf^{\cc},\beta}^{\zz/2}\xrightarrow{i}\mm_{\lf^{\cc},\beta}.\label{eq:open-closed relation}
\end{equation}
Here $\mm_{\lf^{\cc},\beta}=\overline{\mm}_{0,k+2l}\left(X,\beta\right)$
is the moduli space of stable genus zero closed curves marked by $\lf^{\cc}:=\kf\coprod\left(\lf\times\left\{ 1,2\right\} \right)$,
and $\mm_{\lf^{\cc},\beta}^{\zz/2}$ denotes its homotopy fixed points
under an antiholomorphic involution which conjugates the map, fixes
$\kf$ and swaps $\lf\times\left\{ 1\right\} $ and $\lf\times\left\{ 2\right\} $.
$B$ is a kind of blowup, which replaces a neighbourhood of each point
$p$ of a real simple normal crossings divisor ${\mathcal{W}\overset{D}{\hookrightarrow}\mm_{\lf^{\cc},\beta}^{\zz/2}}$
by a $2^{c}$ orthants, where $c=\left|D^{-1}\left(p\right)\right|$.
$o$ is a 2-cover corresponding (over interior points) to a choice
of orientation for the boundary of the domain disc, and $s$ is a
clopen component forming a section of the quotient map which identifies
$\left(x,1\right)\sim\left(x,2\right)$ for $x\in\lf$. $\tb_{\basic}$
acts on $\mm_{\basic}$ by translating maps and permuting labels.
This close relationship between the open and closed moduli spaces
turns out to be useful for computations and to avoid some technical
obstacles in the construction of orbifold tubular neighborhoods for
the $\tb$-action fixed points. 

For a sturdy moduli specification $\sfr=\left(\basic,\sigma\right)=\left(\left(\kf,\lf,\beta\right),\sigma\right)$
and $S\subset\nn$ we set

\[
\acute{\mm}_{\sfr}^{S}=\mm_{\left(\kf\coprod\left(\sigma\cap\sstar'_{S}\right),\lf,\beta\right)}
\]
with $\mm_{\sfr}=\acute{\mm}_{\sfr}^{\nn}$ and $\check{\mm}_{\sfr}=\acute{\mm}_{\sfr}^{\emptyset}$.
We denote by $\For_{\sfr}:\mm_{\sfr}\to\check{\mm}_{\sfr}$ the map
that forgets the superfluous markings $\sigma$. It is a b-fibration
(see the appendix). For any $S\subset\nn$ we have a decomposition
\[
\mm_{\sfr}\xrightarrow{\acute{\For}_{\sfr}^{S}}\acute{\mm}_{\sfr}^{S}\xrightarrow{\grave{\For}_{\sfr}^{S}}\check{\mm}_{\sfr}
\]
where $\acute{\For}_{\sfr}^{S}$ (respectively, $\grave{\For}_{\sfr}^{S}$)
is the map that forgets the markings $\sigma\backslash\sstar'_{S}$
(resp. $\sigma\cap\sstar'_{S}$).

We have evaluation maps 
\[
\evi_{x}^{\sfr}:\mm_{\sfr}\to X\text{ for }x\in\lf\text{ and}
\]
\[
\ev_{x}^{\sfr}:\mm_{\sfr}\to L\text{ for }x\in\kf\coprod\sigma.
\]

\subsubsection{\label{subsec:Resolutions}Resolutions}

Let $\basic=\left(\kf,\lf,\beta\right)$ be a basic moduli specification.
Let $\tb_{\basic}^{r}=\tb_{\basic}\times\Sym\left(r\right)$, the
product of a torus group with the group of permutations of $r$ elements.
Our approach for defining invariants and proving the fixed-point
formula involves the construction of \emph{resolutions, }which are
sequences of $\tb_{\basic}^{r}$-orbifolds $\mm_{\basic}^{r},\check{\mm}_{\basic}^{r}$
defined for $r=0,1,...$ with 
\[
\mm_{\basic}^{0}=\check{\mm}_{\basic}^{0}=\mm_{\basic}
\]
and $\mm_{\basic}^{r}=\check{\mm}_{\basic}^{r}=\emptyset$ for sufficiently
large $r$. The idea is simple. Recall we have 
\begin{equation}
\partial\mm_{\basic}=\coprod\mm_{\sfr'}\underset{\ev{}_{\sstar'_{1}}^{\sfr'}\qquad\ev{}_{\sstar''_{1}}^{\sfr''}}{\times}\mm_{\sfr''}\label{eq:boundary decomposition}
\end{equation}
where the disjoint union is taken over all pairs of moduli specifications
$\sfr'=\left(\left(\kf',\lf',\beta'\right),\sigma'=\left\{ \sstar'_{1}\right\} \right)$
and $\sfr''=\left(\left(\kf''=\hat{\kf}''\coprod\sstar''_{1},\lf'',\beta''\right),\sigma''=\emptyset\right)$
for
\begin{gather}
\kf=\kf'\coprod\hat{\kf}'',\;\lf=\lf'\coprod\lf'',\;\mbox{and }\mbox{\ensuremath{\beta}=\ensuremath{\beta}'+\ensuremath{\beta}''},\label{eq:decomposing moduli specifications}
\end{gather}
and where $\sstar'_{1},\sstar''_{1}$ denote two new boundary markings
(representing the special boundary points identified by the node).
Note how (\ref{eq:orientability of moduli specification}) specifies
the order of the fiber factors.

We take $\mm_{\basic}^{1}=\coprod\mm_{\sfr'}\times\mm_{\sfr''}$,
replacing the fiber products in (\ref{eq:boundary decomposition})
by products. This amounts to allowing the two discs in the nodal configuration
to move independently of each other. The space $\check{\mm}_{\basic}^{1}$
is obtained from $\mm_{\basic}^{1}$ by forgetting the superfluous
marking $\sstar'_{1}\in\sigma'$. We denote by $\For_{\basic}^{1}:\mm_{\basic}^{1}\to\check{\mm}_{\basic}^{1}$
the corresponding forgetful map. The space $\mm_{\basic}^{2}$ is
obtained in a similar fashion, replacing fiber products with products
for the clopen component 
\[
\partial_{-}\mm_{\basic}^{1}\subset\partial\mm_{\basic}^{1}
\]
lying over $\partial\check{\mm}_{\basic}^{1}$. We will see that the
complementary part of the boundary, $\partial_{+}\mm_{\basic}^{1}$,
is insignificant in the sense that it admits an orientation-reversing
involution $\tau_{\basic}^{1}:\partial_{+}\mm_{\basic}^{1}\to\partial_{+}\mm_{\basic}^{1}$.
More generally, $\mm_{\basic}^{r}$ is a disjoint union of products
of spaces modeled on labeled trees, and $\check{\mm}_{\basic}^{r}$
is obtained by forgetting $\sstar'_{\left[r\right]}$. The precise
statement is as follows (cf. \cite[Lemma 17]{equiv-OGW-invts}).
\begin{defn}
\label{def:labeled tree}A \emph{$\left(\basic,r\right)$-labeled
tree $\tc$} is a tree with oriented labeled edges $\tc_{1}=\left\{ e_{1},...,e_{r}\right\} $,
labeled vertices $\tc_{0}=\left\{ v_{1},...,v_{r+1}\right\} $ and
a map $\sfr_{\tc}$ which assigns to each $v_{i}$ a sturdy moduli
specification 
\[
\sfr_{\tc}\left(v_{i}\right)=\left(\left(\kf_{\tc}\left(v_{i}\right),\lf_{\tc}\left(v_{i}\right),\beta_{\tc}\left(v_{i}\right)\right),\sigma_{\tc}\left(v_{i}\right)\right).
\]
These are required to satisfy the following properties. 

(a) the head (respectively, the tail) of the edge $e_{j}\in\tc_{1}$
is the vertex $v_{i}$ if and only if $\sstar''_{j}\in\kf_{\tc}\left(v_{i}\right)$
(resp., iff $\sstar'_{j}\in\sigma_{\tc}\left(v_{i}\right)$).

(b) we have
\[
\coprod_{i=1}^{r+1}\kf_{\tc}\left(v_{i}\right)=\kf\coprod\sstar''_{\left[r\right]},\;\coprod_{i=1}^{r+1}\lf_{\tc}\left(v_{i}\right)=\lf,\;\sum_{i=1}^{r+1}\beta_{\tc}\left(v_{i}\right)=\beta,\text{ and }\coprod_{i=1}^{r+1}\sigma_{\tc}\left(v_{i}\right)=\sigma\coprod\sstar'_{\left[r\right]}
\]

(c) Let $1\leq a\leq r$. Remove the edges $\left\{ e_{a},e_{a+1}...,e_{r}\right\} $
from $\tc$, and let $\tc',\tc''$ denote the connected components
of the resultant forest that are joined by $e_{a}$. Then if $v_{i}$
is a vertex of $\tc'$ and $v_{j}$ is a vertex of $\tc''$ then $i<j$.

We denote the set of $\left(\basic,r\right)$-labeled trees by $\ts_{\basic}^{r}$.
\end{defn}
For $S\subset\nn$ we define spaces

\[
\acute{\mm}_{\basic}^{r,S}=\coprod_{\tc\in\ts_{\basic}^{r}}\acute{\mm}_{\tc}^{S}\text{ for }\acute{\mm}_{\tc}^{S}=\prod_{i=1}^{r+1}\acute{\mm}_{\sfr_{\tc}\left(v_{i}\right)}^{S},
\]
with $\mm_{\basic}^{r}:=\acute{\mm}_{\basic}^{r,\nn}$ and $\check{\mm}_{\basic}^{r}:=\acute{\mm}_{\basic}^{r,\emptyset}$.
Note that the products are ordered according to the vertex labels. 

We define a b-fibration
\[
\For_{\basic}^{r}:\mm_{\basic}^{r}\to\check{\mm}_{\basic}^{r}
\]
by setting $\For_{\basic}^{r}:=\coprod_{\tc\in\ts_{\basic}^{r}}\prod_{i=1}^{r+1}\For_{\sfr_{\tc}\left(v_{i}\right)}$.
This defines a decomposition of the boundary 
\[
\partial\mm_{\basic}^{r}=\partial_{+}\mm_{\basic}^{r}\coprod\partial_{-}\mm_{\basic}^{r}
\]
into ``horizontal'' and ``vertical'' components, respectively,
with induced maps 
\[
\left(\For_{\basic}^{r}\right)_{-}:\partial_{-}\mm_{\basic}^{r}\to\partial\check{\mm}_{\basic}^{r}\text{ and}
\]
\[
\left(\For_{\basic}^{r}\right)_{+}:\partial_{+}\mm_{\basic}^{r}\to\check{\mm}_{\basic}^{r}
\]
(see \cite[\S 2.2.1]{equiv-OGW-invts}). For any $S\subset\nn$ we
have a factorization
\[
\For_{\basic}^{r}=\grave{\For}_{\basic}^{r,S}\circ\acute{\For}_{\basic}^{r,S}
\]
through b-fibrations 
\[
\acute{\For}{}_{\basic}^{r,S}=\coprod_{\tc\in\ts_{\basic}^{r}}\prod_{i=1}^{r+1}\acute{\For}{}_{\sfr_{\tc}\left(v_{i}\right)}^{S},\;\grave{\For}{}_{\basic}^{r,S}=\coprod_{\tc\in\ts_{\basic}^{r}}\prod_{i=1}^{r+1}\grave{\For}{}_{\sfr_{\tc}\left(v_{i}\right)}^{S}.
\]
We denote by $\check{\pi}_{\basic}^{r}:\check{\mm}_{\basic}^{r}\to\ts_{\basic}^{r}$
the locally constant map which sends $\check{\mm}_{\tc}$ to $\tc$.
$\Sym\left(r\right)=\Sym\left(\left[r\right]\right)$ acts on $\ts_{\basic}^{r}$
by relabeling the edges $e_{i}$. This action lifts to $\mm_{\basic}^{r}$
by relabeling the special points $\sstar'_{\left[r\right]}\coprod\sstar''_{\left[r\right]}$
accordingly, and then relabeling the vertices so that condition (c)
is preserved. $\Sym\left(r\right)$ also acts on $\check{\mm}_{\basic}^{r}$
making the maps $\For_{\basic}^{r}$ and $\check{\pi}_{\basic}^{r}$
$\Sym\left(r\right)$-equivariant. The $\Sym\left(r\right)$ action
commutes with the $\tb_{\basic}$ action and makes all the maps $\tb_{\basic}^{r}$-equivariant.
We define an action of $\Sym\left(\left[r\right]\cap S\right)\times\Sym\left(\left[r\right]\backslash S\right)$
on $\acute{\mm}_{\basic}^{r,S}$ in a similar fashion, and this makes
the maps into and out of $\acute{\mm}_{\basic}^{r,S}$ ${\left(\Sym\left(\left[r\right]\cap S\right)\times\Sym\left(\left[r\right]\backslash S\right)<\Sym\left(r\right)\right)}$-equivariant.

\subsubsection{Additional structures and properties}

We have evaluation maps
\[
\acute{\ev}_{x}^{\basic,r,S}:\acute{\mm}_{\basic}^{r,S}\to L\text{ for }x\in\kf\coprod\sstar''_{\left[r\right]}\coprod\sstar'_{\left[r\right]\cap S}
\]
and
\[
\acute{\evi}_{x}^{\basic,r,S}:\acute{\mm}_{\basic}^{r,S}\to X\text{ for }x\in\lf.
\]
We set $\ev_{x}^{\basic,r}=\acute{\ev}_{x}^{\basic,r,\nn}$ and $\check{\ev}_{x}^{\basic,r}=\acute{\ev}_{x}^{\basic,r,\emptyset}$
and similarly for the interior evaluation maps, $\evi$.We construct
local system maps
\[
\ff_{\basic}^{r}:\Or\left(T\mm_{\basic}^{r}\right)\to\Or\left(T\check{\mm}_{\basic}^{r}\right)
\]
 and 
\begin{equation}
\check{\mathcal{J}}_{\basic}^{r}:\Or\left(T\check{\mm}_{\basic}^{r}\right)\to\Or\left(TL\right)^{\boxtimes\left(k+r\right)}\label{eq:check J}
\end{equation}
 lying over $\For_{\basic}^{r}$ and ${\prod_{x\in\kf\coprod\left(\sstar''_{i}\right)_{i=1}^{r}}\check{\ev}{}_{x}^{\basic,r}:\check{\mm}_{\basic}^{r}\to L^{k+r}}$,
respectively. We set 
\[
\jc_{\basic}^{r}:=\check{\jc}_{\basic}^{r}\circ\ff_{\basic}^{r}:\Or\left(T\mm_{\basic}^{r}\right)\to\Or\left(TL\right)^{\boxtimes\left(k+r\right)}.
\]

Writing
\[
\ed_{i}^{\basic,r}:=\ev{}_{\sstar_{i}'}^{\basic,i}\times\ev{}_{\sstar_{i}''}^{\basic,r}:\mm_{\basic}^{r}\to L\times L,
\]
there's a cartesian square 
\begin{equation}
\xymatrix{\partial_{-}\mm_{\basic}^{r}\ar[d]\ar[r]^{g_{\basic}^{r+1}} & \mm_{\basic}^{r+1}\ar[d]^{\ed_{r+1}^{\basic,r+1}}\\
L\ar[r]_{\Delta_{L}} & L\times L
}
.\label{eq:g defining square}
\end{equation}
which induces a local system map 
\[
\mathcal{G}_{\basic}^{r+1}:\Or\left(T\partial_{-}\mm_{\basic}^{r}\right)\to\Or\left(T\mm_{\basic}^{r+1}\right)\otimes\left(\ev{}_{\sstar''_{r+1}}^{\basic,r+1}\right)^{-1}\Or\left(TL\right)
\]
lying over $g_{\basic}^{r+1}$. Let 
\[
{\left(\jc_{\basic}^{r+1}\right)^{\to}:\mm_{\basic}^{r+1}\otimes\left(ev_{\sstar''_{r+1}}^{\basic,r+1}\right)^{-1}\Or\left(TL\right)^{\vee}\to\Or\left(TL\right)^{\boxtimes\left(k+r\right)}}
\]
be the local system map derived from ${\mathcal{J}_{\basic}^{r+1}:\mm_{\basic}^{r+1}\to Or\left(TL\right)^{\boxtimes\left(k+\left(r+1\right)\right)}}$.
$\jc_{\basic}^{r}$ is \emph{coherent,} in the sense that 
\begin{equation}
\left(\jc_{\basic}^{r+1}\right)^{\to}\circ\mathcal{G}_{\basic}^{r+1}=\mathcal{J}_{\basic}^{r}\circ\iota_{\mm_{\basic}^{r}}^{\partial}|_{\partial_{-}\mm_{\basic}^{r}}.\label{eq:J_b^r coherence}
\end{equation}
where 
\[
\iota_{\mm_{\basic}^{r}}^{\partial}:\Or\left(T\partial\mm_{\basic}^{r}\right)\to\Or\left(T\mm_{\basic}^{r}\right)
\]
is the local system map defined by the outward normal orientation
convention.

We have a $\tb_{\basic}^{r}$-equivariant involution $\tau_{\basic}^{r}:\partial_{+}\mm_{\basic}^{r}\to\partial_{+}\mm_{\basic}^{r}$
which reverses the $\jc_{\basic}^{r}$-orientation

\begin{equation}
\mathcal{J}_{\basic}^{r}\circ\left(\iota_{\mm_{\basic}^{r}}^{\partial}|_{\partial_{+}\mm_{\basic}^{r}}\right)\circ\Or\left(d\tau_{\basic}^{r}\right)=\left(-1\right)\mathcal{J}_{\basic}^{r}\circ\left(\iota_{\mm_{\basic}^{r}}^{\partial}|_{\partial_{+}\mm_{\basic}^{r}}\right)\label{eq:or rev involution}
\end{equation}
and commutes with $\For_{\basic}^{r}$. To construct it, write
\[
\partial_{+}\mm_{\basic}^{r}=\coprod_{\tc\in\ts_{\basic}^{r}}\coprod_{i=1}^{r+1}\left(\left(\prod_{j=1}^{i-1}\mm_{\sfr_{\tc}\left(v_{i}\right)}\right)\times\partial_{+}\mm_{\sfr_{\tc}\left(v_{i}\right)}\times\left(\prod_{j=i+1}^{r+1}\mm_{\sfr_{\tc}\left(v_{i}\right)}\right)\right)
\]
with 
\[
\partial_{+}\mm_{\left(\kf,\lf,\beta\right),\sigma}=\coprod\mm_{\underbrace{\left(\left(\kf',\lf',\beta'\right),\hat{\sigma}'\coprod\sstar'_{r+1}\right)}_{\sfr'}}\times_{L}\mm_{\underbrace{\left(\left(\hat{\kf}''\coprod\sstar''_{r+1},\lf'',\beta''\right),\sigma''\right)}_{\sfr''}}
\]
where the disjoint union ranges over all partitions (\ref{eq:decomposing moduli specifications})
and $\sigma=\hat{\sigma}'\coprod\sigma''$, such that $\sfr_{0}=\left(\left(\kf_{0},\emptyset,0\right),\sigma_{0}\right)$
is wobbly (see (\ref{eq:wobbly spec})) for $\sfr_{0}=\sfr'$ or $\sfr_{0}=\sfr''$,
but not both. We let $\tau_{\basic}^{r}$ act by relabeling the last
two elements of $\sigma_{0}$.

The local system maps $\ff_{\basic}^{r}$, $\check{\jc}_{\basic}^{r}$
and $\mathcal{J}_{\basic}^{r}$ are $\tb_{\basic}$-equivariant. If
we let $\tau.$ denote the action of $\tau\in\Sym\left(r\right)<\tb_{\basic}^{r}$
then

\begin{align}
\jc_{\sfr}^{\rho}\circ\Or\left(d\tau.\right) & =\sgn\left(\tau\right)\cdot\jc_{\sfr}^{\rho}\label{eq:J_s^rho equivariance}\\
\check{\jc}_{\sfr}^{\rho}\circ\Or\left(d\tau.\right) & =\sgn\left(\tau\right)\cdot\check{\jc}_{\sfr}^{\rho}\label{eq:check J equivariance}\\
\ff_{\sfr}^{\rho}\circ\Or\left(d\tau.\right) & =\Or\left(d\tau.\right)\circ\ff_{\sfr}^{\rho}\label{eq:ff equivariance}
\end{align}
where $\operatorname{sgn}\left(\tau\right)\in\left\{ \pm1\right\} $
is the sign of $\tau$.

\subsubsection{\label{subsec:Extended forms and integ}Extended forms and integration.}

For each $\tc_{+}\in\ts_{\basic}^{r+1}$ there's a unique $\tc\in\ts_{\basic}^{r}$
such that 
\[
\left(g_{\basic}^{r+1}\right)^{-1}\left(\mm_{\tc_{+}}\right)=:\partial^{\tc_{+}}\mm_{\tc}\subset\partial\mm_{\tc}
\]
this defines a map $\cnt_{\basic}^{r+1}:\ts_{\basic}^{r+1}\to\ts_{\basic}^{r}$,
which one may think of as contracting $e_{r+1}$ to some vertex $v_{i}\in\tc_{0}$.
We have 
\[
\partial_{-}\mm_{\tc}=\coprod_{\left\{ \tc_{+}|\cnt_{\basic}^{r+1}\left(\tc_{+}\right)=\tc\right\} }\partial^{\tc_{+}}\mm_{\tc}.
\]
Consider some $S\subset\left[r\right]$. We set
\[
\partial^{\tc_{+}}\mm_{\tc}^{S}=\left(\acute{\For}_{\basic}^{r,S}\right)_{-}\left(\partial^{\tc_{+}}\mm_{\tc}\right)\subset\partial\mm_{\tc}^{S}.
\]
Note that if $\tc_{1},\tc_{2}\in\ts_{\basic}^{r+1}$ differ by moving
$\sstar'_{i}$, $i\not\in S$, from the head of $e_{r+1}$ to its
tail, or vice-versa, then $\partial^{\tc_{1}}\mm_{\tc}^{S}=\partial^{\tc_{2}}\mm_{\tc}^{S}$.

Generalizing (\ref{eq:g defining square}), we find that there's a
map 
\[
\acute{g}_{\tc_{+}}^{S}:\partial^{\tc_{+}}\mm_{\tc}^{S}\to\acute{\mm}_{\tc_{+}}^{S\coprod\left\{ r+1\right\} }
\]
sitting in a cartesian square 
\[
\xymatrix{\partial^{\tc_{+}}\mm_{\tc}\ar[d]_{\left(\acute{\For}_{\basic}^{r,S}\right)_{-}}\ar[r]^{g_{\basic}^{r+1}} & \mm_{\tc_{+}}\ar[d]^{\acute{\For}_{\basic}^{r+1,S\coprod\left\{ r+1\right\} }}\\
\partial^{\tc_{+}}\check{\mm}_{\tc}\ar[r]_{\acute{g}_{\tc_{+}}^{S}} & \acute{\mm}_{\tc_{+}}^{S\coprod\left\{ r+1\right\} }
}
.
\]
We define $\check{g}_{\tc_{+}}:\partial^{\tc_{+}}\check{\mm}_{\tc}\to\check{\mm}_{\tc_{+}}$
by $\check{g}_{\tc_{+}}=\grave{\For}_{\basic}^{r+1,\left\{ r+1\right\} }\circ\acute{g}_{\tc_{+}}^{\emptyset}$. 
\begin{defn}
An \emph{extended form} $\omega$ for $\basic$ is a sequence 

\[
\omega=\left\{ \check{\omega}_{r}\in\Omega\left(\check{\mm}_{\basic}^{r};\check{\mathcal{E}}_{\basic}^{r}\otimes_{\zz}\rr\left[\vect\lambda\right]\right)^{\tb}\right\} _{r\geq0}\mbox{ for }\check{\mathcal{E}}_{\basic}^{r}:=\bigotimes_{x\in\kf}\left(\check{\ev}{}_{i}^{\basic,r}\right)^{*}\mbox{Or}\left(TL\right),
\]
satisfying the following two conditions: 

(coherence) for all $\tc_{+}\in\ts_{\basic}^{r+1}$ we have
\begin{equation}
\left(i_{\check{\mm}_{\tc}}^{\partial^{\tc_{+}}}\right)^{*}\check{\omega}_{r}=\left(\check{g}_{\tc_{+}}\right)^{*}\check{\omega}_{r+1}\label{eq:check omega coherence}
\end{equation}

(symmetry) $\check{\omega}_{r}$ is $Sym\left(r\right)$-invariant.
\end{defn}
The set of extended forms is a complex, $\left(\Omega_{b},D\right)$,
with $D$ acting level-wise by the Cartan-Weil differential $D=d-\sum_{j=1}^{m}\lambda_{j}\iota_{\xi_{j}}$,
where for $1\leq j\leq m$, $\iota_{\xi_{j}}$ denotes contraction
with $\xi_{j}$, the image under the action map of the generator $\sqrt{-1}u_{j}\frac{\partial}{\partial u_{j}}$
of the lie algebra of $U\left(1\right)^{m}$.

We construct a blow up of $\mm_{\basic}^{r}$,
\[
\widetilde{\mm}_{\basic}^{r}\xrightarrow{\bu_{\basic}^{r}}\mm_{\basic}^{r},
\]
as follows. Let 
\begin{equation}
N_{\Delta}\xrightarrow{\gamma_{\Delta}}L\times L\label{eq:diagonal tube}
\end{equation}
be a tubular neighborhood\footnote{in terms of Definition \ref{def:tubular ngbhd}, we're taking $j=\id$.}
for the diagonal $L\xrightarrow{\Delta}L\times L$.

We use this to construct the blow up of $L\times L$ at the diagonal,
\[
\widetilde{L\times L}:=S\left(N_{\Delta}\right)\times[0,\infty)\bigcup\left(L\times L\backslash\Delta\right)\xrightarrow{\bu_{\Delta}}L\times L.
\]
The map 
\[
\ed_{\basic}^{r}:=\ed_{1}^{\basic,r}\times\cdots\times\ed_{r}^{\basic,r}:\mm_{\basic}^{r}\to\left(L\times L\right)^{r}
\]
\[
\ed_{i}^{\basic,r}:=\ev_{\sstar'_{i}}^{\basic,r}\times\ev_{\sstar''_{i}}^{\basic,r}
\]
 is transverse to $\Delta^{r}\subset\left(L\times L\right)^{r}$,
and we define $\widetilde{\mm}_{\basic}^{r}$ by the cartesian square
\[
\xymatrix{\widetilde{\mm}_{\basic}^{r}\ar[r]^{\bu_{\basic}^{r}}\ar[d]_{\widetilde{\ed}_{\basic}^{r}} & \mm_{\basic}^{r}\ar[d]^{\ed_{\basic}^{r}}\\
\left(\widetilde{L\times L}\right)^{r}\ar[r]_{\bu_{\Delta}^{\times r}} & \left(L\times L\right)^{r}
}
.
\]
$\mathcal{J}_{\basic}^{r}$ induces a local system map
\[
\widetilde{\mathcal{J}}_{\basic}^{r}:Or\left(T\widetilde{\mm}_{\basic}^{r}\right)\to\Or\left(TL\right)^{\boxtimes\left(\kf\coprod\sstar''_{\left[r\right]}\right)}.
\]
We use this to define an $\rr\left[\vec{\lambda}\right]$-linear,
$\Sym\left(k\right)\times\Sym\left(l\right)$-invariant integration
map 
\[
\int_{\basic}:\Omega_{\basic}\to\rr\left[\vec{\lambda}\right],
\]
as the finite sum of integrals
\begin{equation}
\int_{\basic}\omega=\sum_{r\geq0}\frac{1}{r!}\int_{\widetilde{\mm}_{\basic}^{r}}\left(\For_{\basic}^{r}\circ\bu_{\basic}^{r}\right)^{*}\check{\omega}_{r}\cdot\left(\widetilde{\ed}_{\basic}^{r}\right)^{*}\Lambda^{\boxtimes r}.\label{eq:extended integral}
\end{equation}
Here $\Lambda\in\Omega\left(\widetilde{L\times L};\tilde{\mbox{pr}}_{2}^{*}\left(Or\left(TL\right)\right)\otimes_{\zz}\rr\left[\vec{\lambda}\right]\right)^{\tb}$
is an \emph{equivariant homotopy kernel,}
\[
\Lambda=\sigma\left(\mbox{pr}_{S\left(N_{\Delta}\right)}^{\widetilde{N}_{\Delta}}\right)^{*}\theta_{\Delta}+\bu_{\Delta}^{*}\Upsilon
\]
for $\theta_{\Delta}$ an equivariant angular form\footnote{We review the definition above Theorem \ref{thm:general fixed point formula}.}
for $S\left(N_{\Delta}\right)$, $\sigma:[0,\infty)\to[0,1]$ a smooth,
compactly supported cutoff function with $\sigma\left(0\right)=+1$,
and $\Upsilon$ chosen so that 
\begin{equation}
D\Lambda=-\widetilde{\pr}_{2}^{*}p_{0}^{L}\label{eq:D Lambda}
\end{equation}
for 
\[
p_{0}^{L}\in\Omega\left(L;\Or\left(TL\right)\otimes_{\zz}\rr\left[\vec{\lambda}\right]\right)^{\tb}
\]
a form representing the equivariant Poincare dual to $p_{0}$ in $L$.
Note we have 
\begin{equation}
\left(i_{\widetilde{L\times L}}^{\partial}\right)^{*}\Lambda=\theta_{\Delta}+\left(\pi_{\Delta}^{S\left(N_{\Delta}\right)}\right)^{*}\Upsilon|_{\Delta}.\label{eq:Lambda asymptotics}
\end{equation}

The $r=0$ summand in (\ref{eq:extended integral}) is just

\[
\int_{\overline{\mm}_{0,k,l}\left(X,L,\beta\right)}\omega_{0};
\]
one may think of the summands for $r\geq1$ as corrections accounting
for the the boundary and corners of $\overline{\mm}_{0,k,l}\left(X,L,\beta\right)$.

A key property of the integral is that it satisfies Stokes' theorem:

\begin{equation}
\int_{\basic}D\upsilon=0\label{eq:stokes}
\end{equation}
for any $\upsilon\in\Omega_{\basic}$.

\subsubsection{\label{subsec:Sorted-Odd-Even-Trees}Sorted odd-even trees}

\begin{defn}
(a) A labeled tree $\tc\in\ts_{\basic}^{r}$ will be called an\emph{
odd-even tree }if all the moduli specifications $\left(\left(\kf,\lf,\beta\right),\sigma\right)=\sfr_{\tc}\left(v\right)$
for a vertex $v\in\tc_{0}$ satisfy the following condition: if $\beta=0\mod2$
then $\sigma=\emptyset$ and if $\beta=1\mod2$ then $\kf=\emptyset$
(note in particular this means the tree is bipartite with respect
to the partition into odd degree and even degree vertices, which explains
the name). We denote the set of odd-even trees by $\ts_{\basic}^{r,0}$.

(b) An odd-even tree $\tc\in\ts_{\basic}^{r}$ will be called \emph{sorted
}if the graph spanned by the edges $1,...,a$ is connected for every
$1\leq a\leq r$, and such that if ${\left\{ \sstar'_{i},\sstar'_{j}\right\} \subset\sigma_{\tc}\left(v\right)}$
for some $i<j$ and $v\in\tc_{0}$ then $\sstar'_{a}\in\sigma_{\tc}\left(v\right)$
for all $i\leq a\leq j$.
\end{defn}
\begin{prop}
\label{prop:sorted is good}Let $\basic=\left(\kf,\lf,\beta\right)$
be a basic moduli specification. 

(a) For every $\tc\in\ts_{\basic}^{r,0}$ there exists at least one
$\tau\in\Sym\left(r\right)$ such that $\tau.\tc$ is sorted.

(b) For a sorted odd-even\textup{ $\tc\in\ts_{\basic}^{r,0}$ we have}
\begin{align*}
\check{\jc}_{\basic}^{r}|_{\mm_{\tc}} & =\left(-1\right)^{\left(1+m\right)\cdot\binom{o}{2}}\left(\prod\check{\jc}\right){}_{\basic}^{r}\\
\ff_{\basic}^{r}|_{\mm_{\tc}} & =\left(\prod\ff\right)_{\basic}^{r}
\end{align*}
where $o$ is the number of odd vertices, $\left(\prod\check{\jc}\right){}_{\basic}^{r}|_{\check{\mm}_{\tc}}$
is the composition
\[
\Or\left(T\left(\bigtimes_{i=1}^{r+1}\check{\mm}_{\sfr_{\tc}\left(v_{i}\right)}\right)\right)\simeq\bigotimes_{i=1}^{r+1}\Or\left(T\check{\mm}_{\sfr_{\tc}\left(v_{i}\right)}\right)\xrightarrow{\bigotimes_{i=1}^{r+1}\check{\jc}_{\sfr_{\tc}\left(v_{i}\right)}^{\emptyset}}\Or\left(TL\right)^{\boxtimes\left(k+r\right)},
\]
and $\left(\prod\ff\right)_{\basic}^{r}|_{\mm_{\tc}}$ is similarly
induced from $\bigotimes_{i=1}^{r+1}\ff_{\sfr_{\tc}\left(v_{i}\right)}^{\emptyset}$
using the order of the vertices.

(c) if $\tc\in\ts_{\basic}^{r,0}$ is sorted and $\beta_{\tc}\left(v_{i}\right)=1\mod2$
and $\beta_{\tc}\left(v_{j}\right)=0\mod2$ for some $1\leq i,j\leq r+1$
then $i<j$.
\end{prop}
\textbf{Warning.} In this paper $\ff_{\basic}^{r}$ and $\check{\jc}_{\basic}^{r}$
are both twisted by $\left(-1\right)^{\binom{r}{2}}$ relative to
their definition in \cite{equiv-OGW-invts}. This cleans up some signs,
compare for instance part (b) of Proposition \ref{prop:sorted is good}
and \cite[Proposition 29]{equiv-OGW-invts}. Their composition $\jc_{\basic}^{r}$
is of course the same.

\subsection{\label{subsec:closed and open fps}The open fixed points}

A \emph{local} \emph{orbifold }(or \emph{l-orbifold }for short) $\xx=\coprod_{n\geq0}\xx_{n}$
is a disjoint union of orbifolds with corners with $\dim\xx_{n}=n$.
When considering a vector bundle\emph{ }over a local orbifold we will
always assume that the total space is an (ordinary) orbifold of constant
dimension $N$, so a vector bundle $E\to\xx$ is given by a vector
bundle $E_{n}\to\xx_{n}$ of rank $N-n$ for each $n\geq0$. 

Let $\basic=\left(\kf,\lf,\beta\right)$ be a basic moduli specification.
Let $\lf^{\cc}=\kf\coprod\lf\times\left\{ 1,2\right\} $. We use
the following diagram to construct an l-orbifold with corners $\mm_{\basic}^{\tb}$
together with a closed embedding $\mm_{\basic}^{\tb}\to\mm_{\basic}$
representing the fixed-point stack (see the appendix). 
\[
\xymatrix{\mm_{\basic}^{\tb}\ar[r]\ar[d] & \left(\mm_{\basic}^{::}\right)^{\tb}\ar[d]\ar[r] & \left(\widetilde{\mm_{\lf^{\cc},\beta}^{\zz/2}}\right)^{\tb}\ar[d]\ar[r]^{\simeq} & \widetilde{\mm_{\lf^{\cc},\beta}^{\zz/2\times\tb}}\ar[dl]\ar[r]^{B'} & \mm_{\lf^{\cc},\beta}^{\zz/2\times\tb}\ar[r]\ar[d] & \mm_{\lf^{\cc},\beta}^{\tb}\ar[d]\\
\mm_{\basic}\ar[r]_{s} & \mm_{\basic}^{::}\ar[r]_{o} & \widetilde{\mm_{\lf^{\cc},\beta}^{\zz/2}}\ar[rr]_{B} &  & \mm_{\lf^{\cc},\beta}^{\zz/2}\ar[r]_{i_{\mm_{\lf^{\cc},\beta}^{\zz/2}}} & \mm_{\lf^{\cc},\beta}
}
\]
The bottom row is just (\ref{eq:open-closed relation}). The top row
is constructed from right to left. The vertical maps are the fixed-point
structure maps, and are closed embeddings (cf. Definition \ref{def:closed embedding}).
The unlabeled horizontal maps on the top row are the maps of $\tb$-fixed
points induced by the maps lying below them. The slanted map is the
blow up of $\mm_{\lf^{\cc},\beta}^{\zz/2\times\tb}\to\mm_{\lf^{\cc},\beta}^{\zz/2}$,
making the adjacent trapezoid 2-cartesian. The map labeled $\simeq$
is the natural diffeomorphism $\left(\mm^{\tb}\right)^{\sim}\simeq\left(\widetilde{\mm}\right)^{\tb}$
(cf. \cite[\S 2.1]{equiv-OGW-invts}). Let $f_{\basic}:\mm_{\basic}\to\mm_{\lf^{\cc},\beta}$
denote the composition of the maps on the bottom row so $f_{\basic}^{\tb}$
is the induced map of fixed points. Taking normal bundles of the vertical
maps we find that 
\[
N_{\basic}\simeq\left(f_{\basic}^{\tb}\right)^{*}N_{\lf^{\cc},\beta}^{\zz/2}.
\]

We have immersive \emph{closed gluing maps }

\[
\mm_{\lf_{1}\coprod\star_{1},\beta_{1}}\times_{X}\mm_{\lf_{2}\coprod\star_{2},\beta_{2}}\to\mm_{\lf_{1}\coprod\lf_{2},\beta_{1}+\beta_{2}},
\]
with normal bundle $\mathbb{L}_{\star_{1}}^{\vee}\boxtimes\mathbb{L}_{\star_{2}}^{\vee}$
where $\mathbb{L}_{\star_{i}}^{\vee}$ denotes the tangent line at
$\star_{i}$. These gluing maps are associative. In particular, we
have $\zz/2$-equivariant maps
\[
\mm_{\left(\lf_{+}\times\left\{ 1\right\} \right)\coprod\star_{1}',\beta_{+}}\times_{X}\mm_{\lf^{\cc}\coprod\left\{ \star_{1},\star_{2}\right\} ,\beta}\times_{X}\mm_{\left(\lf_{+}\times\left\{ 2\right\} \right)\coprod\star'_{2},\beta_{+}}\xrightarrow{\gamma^{\cc}}\mm_{\kf\coprod\left(\left(\lf\coprod\lf_{+}\right)\times\left\{ 1,2\right\} \right),\beta+2\beta_{+}}.
\]
Here $\lf^{\cc}=\kf\coprod\lf\times\left\{ 1,2\right\} $ and $\lf_{+}$
are finite, disjoint subsets and $\beta,\beta_{+}$ non-negative integers,
subject to the usual stability conditions.

Since the boundary divisor is normal crossing, we find that the real
codimension one locus $D_{\kf,\lf,\beta}$ used to define the blow
up $\widetilde{\mm}_{\lf^{\cc},\beta}^{\zz/2}\xrightarrow{B}\mm_{\lf^{\cc},\beta}^{\zz/2}$,
has transversal self-intersection relative to $\gamma^{\zz/2}$ (cf.
\cite[\S 2.1]{equiv-OGW-invts}). It follows that we have induced
\emph{open gluing maps} 
\[
\mm_{\kf,\lf\coprod\left\{ \star_{1}\right\} ,\beta}\times_{X}\mm_{\lf_{+}\coprod\left\{ \star'_{1}\right\} ,\beta_{+}}\xrightarrow{\gamma}\mm_{\kf,\lf\coprod\lf_{+},\beta+2\beta_{+}}.
\]
which are a closed embedding with normal bundle 
\begin{equation}
N_{\gamma}=\mathbb{L}_{\star_{1}}^{\vee}\boxtimes\mathbb{L}_{\star_{1}'}^{\vee}.\label{eq:normal to glued}
\end{equation}
The normal bundle to 
\[
\mm_{\kf,\lf\coprod\left\{ \star_{1}\right\} ,\beta}\times_{X}\mm_{\lf_{+}\coprod\left\{ \star'_{1}\right\} ,\beta_{+}}\to\mm_{\kf,\lf\coprod\left\{ \star_{1}\right\} ,\beta}\times\mm_{\lf_{+}\coprod\left\{ \star'_{1}\right\} ,\beta_{+}}
\]
is $\ev_{\star_{1}}^{*}TX$. There's a natural map of local systems
\begin{equation}
\Or\left(T\mm_{\kf,\lf\coprod\left\{ \star_{1}\right\} ,\beta}\right)\xrightarrow{\Gamma}\Or\left(T\mm_{\kf,\lf\coprod\lf_{+},\beta+2\beta_{+}}\right)\label{eq:gluing ls map}
\end{equation}
lying over $\gamma$, defined using the complex orientations for $TX,\mathbb{L}_{\star_{1}}^{\vee}\boxtimes\mathbb{L}_{\star_{1}'}^{\vee}$
and the tangent spaces to the closed moduli spaces. We have (see \cite{JT})
\[
\jc_{\left(\kf,\lf\coprod\lf_{+},\beta+2\beta_{+}\right)}\circ\Gamma=\jc_{\left(\kf,\lf\coprod\left\{ \star_{1}\right\} ,\beta\right)}.
\]
We can combine the closed and open gluing maps in different ways,
and the order of gluing does not matter. In particular, we obtain
maps
\[
\left(\mm_{\kf,\lf_{0}\coprod\left\{ \star_{1},...,\star_{s}\right\} ,\beta_{0}}\times_{X^{s}}\bigtimes_{i=1}^{s}\mm_{\lf_{i}\coprod\star'_{i},\beta_{i}}\right)_{\Sym\left(s\right)}\to\mm_{\kf,\coprod_{i=0}^{s}\lf_{i},\sum_{i=0}^{s}\beta_{i}}
\]
where subscript $\Sym\left(s\right)$ denotes the stack quotient.
Since the gluing maps are $\tb$-equivariant they induce maps of fixed
points.

Let $\sfr=\left(\basic,\sigma\right)$ be a sturdy moduli specification.
We write 
\[
\check{F}_{\sfr}:=\check{\mm}_{\sfr}^{\tb}=\mm_{\basic}^{\tb}.
\]

\begin{prop}
\label{prop:fp inverse image}(a) The fiber product $F_{\sfr}:=\For_{\sfr}^{-1}\left(\check{F}_{\sfr}\right)$
is an orbifold with corners and the map $F_{\sfr}\to\mm_{\sfr}$ is
a closed embedding.

(b) If $\beta$ is odd, the induced map $F_{\sfr}\xrightarrow{\For_{\sfr}}\check{F}_{\sfr}$
is a proper submersion.
\end{prop}
\begin{proof}
We prove (a). Using atlases this reduces to showing that if $X\xrightarrow{i}Y$
is a closed embedding of \emph{manifolds }with corners and $Y'\xrightarrow{f}Y$
is a b-submersion, then (i) the fiber product $X'=X\fibp{i}{f}Y'$
exists in $\manc$ and (ii) the pullback $X'\xrightarrow{i'}Y'$ is
a closed embedding. 

Claim (i) follows from (iii) of \cite[Lemma 37]{equiv-OGW-invts}.
We prove claim (ii). By assumption, there exists a cartesian square
\[
\xymatrix{X\ar[r]^{i}\ar[d] & Y\ar[d]^{h}\\
0\ar[r] & \rr^{N}
}
.
\]
It follows that the square
\[
\xymatrix{X'\ar[d]\ar[r] & Y'\ar[d]^{h\circ f}\\
0\ar[r] & \rr^{N}
}
\]
is cartesian and $h\circ f$ is a b-submersion, so $i'$ is a closed
immersion by Remark \ref{rem:b-submersive is enough}. The condition
of being a homeomorphism onto a closed image is clearly stable under
pullbacks, completing the proof of claim (ii).

We prove (b). Suppose $\beta$ is odd, and let $p=\left[\Sigma,\nu,u,\kappa,\lambda\right]\in\check{F}_{\sfr}$.
Since the map $H_{2}\left(X\right)\to H_{2}\left(X,L\right)$ is identified
with $\zz\xrightarrow{2\times}\zz$, there is at least one disc connected
component $\Sigma^{0}\subset\Sigma$ with 
\begin{equation}
\beta^{0}:=u_{*}\left[\Sigma^{0},\partial\Sigma^{0}\right]\label{eq:intersection degree}
\end{equation}
a positive (odd) integer. Such a configuration covers $\pp\left(W_{j}\oplus W_{0}\right)$
for some $1\leq j\leq m$, and in particular does not pass through
the only fixed point $p_{0}\in L$, so there can be no special points
on $\partial\Sigma^{0}$. Since $\partial\Sigma/\nu$ is connected
in genus zero, we conclude that $\Sigma^{0}$ is the only disc component
(the other connected components of $\Sigma$ are biholomorphic to
$S^{2}$) and $\kf=\emptyset$. We call $\Sigma^{0}$ (respectively,
$\beta^{0}$) \emph{the intersection disc (resp., intersection degree)
}of the configuration $p$.

Let $\tilde{p}=\left[\tilde{\Sigma},\tilde{\nu},\tilde{u},\tilde{\kappa},\tilde{\lambda}\right]\in F_{\sfr}$
satisfy $\For_{\sfr}\left(\tilde{p}\right)=p$, so $\tilde{\Sigma}=\Sigma\coprod\Sigma'$,
$\tilde{\nu}=\nu\coprod\nu'$, $\tilde{u}=u\coprod u'$ and $\tilde{\lambda}=\lambda$,
with $\Sigma'$ a possibly disconnected configuration connected by
boundary nodes to $\Sigma$ and $u'$ a locally constant map. Using
the infinitesimal transversality of the $O\left(2m+1\right)$ action
we see that any tangent vector $v$ at $p$ can be lifted to a tangent
vector $v'$ at $p'$, by modifying $\left(\Sigma,\nu,u\right)$ according
to $v$ and modifying the other components so that $\tilde{u}$ stays
$\tilde{\nu}$ invariant. Note $p'\in S^{c}\left(\mm_{\sfr}\right)$,
the immersed suborbifold of points of depth $c$ in $\mm_{\sfr}$,
iff $\tilde{\nu}$ has $c$ boundary nodes. By construction, $v'\in TS^{c}\left(\mm_{\sfr}\right)$,
which shows that $\For_{\sfr}$ is a submersion. $\For_{\sfr}$ is
proper since $\mm_{\sfr}$ is compact. 
\end{proof}

\subsubsection{\label{subsec:Fixed-point-profiles.}Fixed-point profiles.}

When discussing fixed points, it is convenient to restrict the possible
labeling of edges of trees, as follows.
\begin{defn}
Let $\ts_{\basic}^{r,r'}\subset\ts_{\basic}^{r+r'}$ be the subset
of trees such that:
\begin{itemize}
\item for $1\leq e\leq r$, $\beta_{\text{tail}\left(e\right)}=1\mod2$,
\item for $r+1\leq e\leq r+r'$, $\beta_{\text{tail}\left(e\right)}=0\mod2$
and
\item for $1\leq e\leq r+r'$ $\beta_{\text{head}\left(e\right)}=0\mod2$.
\end{itemize}
A labeled tree $\tc\in\ts_{\basic}$ is called \emph{separated }if
$\tc\in\ts_{\basic}^{r,r'}$ for some $r,r'\geq0$. The edges $1\leq e\leq r$
are called \emph{odd-even edges }and the edge $r+1\leq e\leq r+r'$
are called \emph{even-even }edges. Note $\ts_{\basic}^{r,0}$ is the
set of odd-even trees.
\end{defn}
Note if $\check{\mm}_{\tc}^{\tb}\neq\emptyset$, then $\tau.\tc$
is separated for some $\tau\in\Sym\left(r+r'\right)$. Define $\mm_{\basic}^{r,r'}\subset\mm_{\basic}^{r+r'}$
by $\mm_{\basic}^{r,r'}=\coprod_{\tc\in\ts_{\basic}^{r,r'}}\mm_{\tc}$.
Similarly, we define clopen components $\mm_{\basic}^{r,r',S}\subset\mm_{\basic}^{r+r',S}$
for $S\subset\left[r\right]$, $\check{\mm}_{\basic}^{r,r'}\subset\check{\mm}_{\basic}^{r+r'}$
by restricting to separated trees.

Let $\tc\in\ts_{\basic}^{r,r'}$ be a separated tree. A \emph{(fixed-point)
profile} for $\tc$ is some discrete information associated with each
connected component of $\check{\mm}_{\tc}^{\tb}\subset\left(\check{\mm}_{\basic}^{r}\right)^{\tb}$,
describing its behavior near $L\subset X$. We first explain how
to determine the profile of a given fixed point, and then use gluing
to reverse the process. 

Write $\left(\left(\kf_{i},\lf_{i},\beta_{i}\right),\sigma_{i}\right)=\sfr_{\tc}\left(v_{i}\right)$
for $1\leq i\leq r+1$. Let $q\in\check{\mm}_{\tc}^{\tb}$. We can
write 
\[
q=\left(q_{1},...,q_{r+1}\right)\in\prod_{i=1}^{r+1}F_{\left(\kf_{i},\lf_{i},\beta_{i}\right)},
\]
where each $q_{i}$ is represented by a $\tb$-invariant $\left(\kf_{i},\lf_{i},\beta_{i}\right)$-disc
configuration 
\[
\left(\Sigma_{i},\kappa_{i},\lambda_{i},\nu_{i},u_{i}\right).
\]

Elementary arguments yield the following.
\begin{lem}
\label{lem:fp domain decomposition}If $\beta_{i}=1\mod2$ we have
$\kf_{i}=\emptyset$, and there's a unique decomposition 
\begin{equation}
\Sigma_{i}=D_{i}\cup\Sigma_{i}^{1}\label{eq:odd domain decomp-1}
\end{equation}
such that $D_{i}$ is a connected component biholomorphic to a disc,
and $\Sigma_{i}^{1}$ is either the unique fixed point of $D_{i}$
or a closed (possibly disconnected) Riemann surface which does not
intersect $D_{i}$.

If $\beta_{i}=0\mod2$ there's a decomposition
\begin{equation}
\Sigma_{i}=D_{i}\coprod\coprod_{j=1}^{s_{i}}\left(P_{i}^{j}\cup\Sigma_{i}^{j}\right)\label{eq:even domain decomp-2}
\end{equation}
where

\begin{itemize}
\item $D_{i}$ is the maximal clopen component of $\Sigma_{i}$ satisfying
(i) $\partial\Sigma_{i}\subset D_{i}$ (ii) $D_{i}/\nu$ is connected,
and (iii) $u\left(D_{i}\right)=p_{0}$.
\item For each $1\leq j\leq s_{i}$, $P_{i}^{j}\simeq\cc\pp^{1}$ is connected
by a node to $D_{i}$.
\item $\Sigma_{i}^{j}$ is either the unique non-nodal fixed point of $P_{i}^{j}$,
or a clopen component connected to $P_{i}^{j}$ by a node. In the
latter case we require $\Sigma_{i}^{j}/\nu$ to be connected.
\end{itemize}
The decomposition (\ref{eq:even domain decomp-2}) is unique up to
renumbering the $P_{i}^{j}\cup\Sigma_{i}^{j}$ components. 
\end{lem}

Consider some $1\leq i\leq r+1$ with $\beta_{i}$ even. We set $\lf_{i}^{0}:=\lambda^{-1}\left(D_{i}\right)$,
and for $1\leq j\leq s_{i}$; we define a positive integer by 
\[
d_{i}^{j}=u_{*}\left[P_{i}^{j}\right]\in H_{2}\left(X\right);
\]
we specify $a_{i}^{j}\in\left\{ 1,...,2m\right\} $ by requiring that
$p_{0},p_{a_{i}^{j}}$ are the fixed points of $u\left(P_{i}^{j}\right)$;
we set $\tilde{d}_{i}^{j}=u_{*}\left[\Sigma_{i}^{j}\right]\in H_{2}\left(X\right)$
and $\lf_{i}^{j}=\lambda^{-1}\left(\Sigma_{i}^{j}\right)$. We call
$\phi_{i}=\left(\kf_{i},\lf_{i}^{0},\left(d_{i}^{j},a_{i}^{j},\tilde{d}_{i}{}^{j},\lf_{i}^{j}\right)_{j=1}^{s_{i}},\sigma_{i}\right)$
the\emph{ profile of $q_{i}$. }Note $\Sym\left(s_{i}\right)$ acts
on $\left(d_{i}^{j},a_{i}^{j},\tilde{d}_{i}{}^{j},\lf_{i}^{j}\right)_{j=1}^{s_{i}}$,
yielding equivalent profiles for $p_{i}$. 

For $1\leq i\leq r+1$ with $\beta_{i}$ odd, the profile of $q_{i}$
is defined to be $\phi_{i}=\left(\beta_{i}^{0},a_{i},\tilde{d}_{i}^{1},\sigma_{i}\right)$
with $\beta_{i}^{0}=u_{*}\left(\left[D_{i},\partial D_{i}\right]\right)$,
${\tilde{d}_{i}^{1}=\frac{\beta_{i}-\beta_{i}^{0}}{2}=u_{*}\left[\Sigma_{i}^{1}\right]}$
and $p_{a_{i}}\in u_{i}\left(D_{i}\right)$ as in $\S$\ref{subsec:fp for generalized forms}. 

We call the data

\[
\vp=\left(\phi_{1},\phi_{2},...,\phi_{r+1}\right)
\]
\emph{the profile of $q$.} The profile of $p$ is unique up to
the obvious ${\prod_{i=1}^{r+1}\Sym\left(s_{i}\right)}$ group action.
We denote by $\tc\left(\vp\right)$ the sturdy tree associated with
$\vp$.
\begin{defn}
(a) For $\sfr=\left(\left(\kf,\lf,\beta\right),\sigma\right)$ a sturdy
moduli specification, the set of (fixed-point) \emph{profiles }for
$\sfr$, \emph{$\mathcal{P}\left(\sfr\right)$,} is defined as follows.

Suppose $\beta=0\mod2$. For $s\geq0$ we define $\mathcal{P}_{s}\left(\sfr\right)$
to be the set of tuples $\left(\lf^{0},\left(d^{j},a^{j},\tilde{d}{}^{j},\lf^{j}\right)_{j=1}^{s},\sigma\right)$
where: (i) $\tilde{d}^{1},...,\tilde{d}^{s}$ are non-negative integers
and $d^{1},...,d^{s}$ are positive integers such that $\beta=2\sum_{j=1}^{s}\left(d^{j}+\tilde{d}^{j}\right)$,
(ii) we have $\lf=\coprod_{j=0}^{s}\lf^{j}$ and $2\left(\left|\lf^{0}\right|+s\right)+\left|\kf\right|\geq3$,
and (iii) $a^{j}\in\left\{ 1,...,2m\right\} $ for $1\leq j\leq s$.
We set $\mathcal{P}\left(\sfr\right)=\coprod_{s\geq0}\mathcal{P}_{s}\left(\sfr\right)$.

Suppose $\beta=1\mod2$. If $\kf\neq\emptyset$ we set $\mathcal{P}\left(\sfr\right)=\emptyset$,
otherwise we let $\mathcal{P}\left(\sfr\right)$ be the set of pairs
$\left(\beta^{0},a,\sigma\right)$ with $\beta^{0}$ an odd integer
and $a\in\left\{ 1,...,2m\right\} $. For notational convenience we
set $\lf^{1}=\lf$ and $\tilde{d}^{1}=\frac{\beta-\beta^{0}}{2}$
in this case. 

(b) For labeled tree $\tc\in\ts_{\basic}^{r}$, the set of\emph{ }profiles\emph{
$\mathcal{P}\left(\tc\right)$ for $\tc$} is 
\[
\mathcal{P}\left(\tc\right)=\prod_{i=1}^{r+1}\mathcal{P}\left(\sfr_{\tc}\left(v_{i}\right)\right).
\]
There's an obvious ${\Sym\left(\vect s\right):=\prod_{i=\no+1}^{r+1}\Sym\left(s_{i}\right)}$
action on $\mathcal{P}_{\vect s}\left(\tc\right)\subset\mathcal{P}\left(\tc\right)$,
where
\[
\mathcal{P}_{\vect s}\left(\tc\right):=\prod_{i=1}^{\no}\mathcal{P}\left(\sfr_{\tc}\left(v_{i}\right)\right)\times\prod_{i=\no+1}^{r+1}\mathcal{P}_{s_{i}}\left(\sfr_{\tc}\left(v_{i}\right)\right).
\]

(c) We set $\pc_{\basic}^{r,r'}=\coprod_{\tc\in\ts_{\basic}^{r,r'}}\pc\left(\tc\right)$,
$\pc_{\basic}^{r,\geq0}=\coprod_{r'\geq0}\pc_{\basic}^{r,r'}$ and
$\pc_{\basic}=\coprod_{r\geq0}\pc_{\basic}^{r,\geq0}$ (so we only
consider profiles with $\tc\left(\vp\right)$ \emph{separated }sturdy
trees).

We will write $\no\left(\vp\right),r\left(\vp\right)$ and $r'\left(\vp\right)$
for the number of odd vertices, even-even edges, and odd-even edges,
respectively. We denote by $\left(\beta_{i}^{0}\left(\vp\right),a_{i}\left(\vp\right),d_{i}\left(\vp\right),\sigma_{i}\left(\vp\right)\right)$,
$1\leq i\leq\no\left(\vp\right)$ the odd-vertex components of the
fixed-point profile, similarly for the even vertices. For notational
convenience, for $1\leq i\leq\no$ we set $s_{i}\left(\vp\right)=1$,
$a_{i}^{1}\left(\vp\right)=a_{i}\left(\vp\right)$, and $d_{i}^{1}\left(\vp\right)=\beta_{i}^{0}\left(\vp\right)/2$
(this is not an integer). If $\vp$ is clear from the context we may
abbreviate $\beta_{i}^{0}=\beta_{i}^{0}\left(\vp\right)$ etc.
\end{defn}
The inclusion $\ts_{\basic}^{r,r'}\subset\ts_{\basic}^{r+r'}$ is
$\Sym\left(r\right)\times\Sym\left(r'\right)<\Sym\left(r+r'\right)$
equivariant, and there's an obvious action of $\Sym\left(r\right)\times\Sym\left(r'\right)$
on $\pc_{\basic}^{r,r'}$ compatible with the action on $\ts_{\basic}^{r,r'}$.

Consider some $\tc\in\ts_{\basic}^{r}$ and $\vp=\left(\phi_{1},...,\phi_{r+1}\right)\in\mathcal{P}\left(\tc\right)$.
We introduce some notation for the associated components.

Assume first $\beta_{i}=0\mod2$. We set

\begin{alignat}{2}
\mm_{D_{i}}\left(\vp\right) & :=\mm_{\kf_{i},\lf_{i}^{0}\coprod\left\{ \star_{i,j}\right\} _{j=1}^{s_{i}},0} & =\overline{\mm}_{0,\kf_{i},\lf_{i}^{0}\coprod\left\{ \star_{i,j}\right\} _{j=1}^{s_{i}}}\left(\pt\right)\times L\label{eq:even discs mod and fp}\\
F_{D_{i}}\left(\vp\right) & :=\mm_{D_{i}}^{\tb}\left(\vp\right) & =\overline{\mm}_{0,\kf_{i},\lf_{i}^{0}\coprod\left\{ \star_{i,j}\right\} _{j=1}^{s_{i}}}\left(\pt\right).\nonumber 
\end{alignat}

Hereafter, we sometimes write $\overline{\mm}_{0,\kf_{i},\lf_{i}^{0}\coprod\left\{ \star_{i,j}\right\} _{i=1}^{s_{i}}}\left(\pt\right)$
for the moduli of discs, in place of $\overline{\mm}_{0,\kf,\lf^{0}\coprod\left\{ \star_{i}\right\} _{i=1}^{s}}$
as in $\S$\ref{subsec:descendent integrals}, to avoid confusion.
We denote $\mm_{P_{i}^{j}}\left(\vect\phi\right)=\mm_{\left\{ \star_{i,j}',\star_{i,j}''\right\} ,d_{i}^{j}}$
and let
\[
F_{P_{i}^{j}}\left(\vect\phi\right)=\underline{\pt}_{\zz/d_{i}^{j}}\subset\mm_{P_{i}^{j}}^{\tb}\left(\vect\phi\right)
\]
be the unique fixed point with smooth domain that corresponds to a
degree $d_{i}^{j}$ branched cover of the line $p_{0}p_{a_{i}^{j}}$.
We also set

\begin{align*}
\mm_{\Sigma_{i}^{j}}\left(\vect\phi\right) & =\mm_{\lf_{i}^{j}\coprod\star_{i,j}''',\tilde{d}_{i}^{j}}\\
F_{\Sigma_{i}^{j}}\left(\vect\phi\right) & =\ev_{\star_{i,j}'''}^{-1}\left(p_{a_{i}^{j}}\right)^{\tb}\subset\mm{}_{\Sigma_{i}^{j}}^{\tb}.
\end{align*}
Hereafter we adopt the ad hoc convention that in case $3\tilde{d}_{i}^{j}+\left|\lf_{i}^{j}\right|+1<3$,
\begin{equation}
\mm_{\lf_{i}^{j}\coprod\left\{ \star_{i,j}'''\right\} ,\tilde{d}_{i}^{j}}=X.\label{eq:unstable convention-1}
\end{equation}
with the evaluation maps given by the identity map to $X$. This corresponds
to the cases in Lemma \ref{lem:fp domain decomposition} where $\Sigma_{i}^{j}=\pt$,
and allows us to write formulas that treat these case on equal footing
with the other cases. We set 

\begin{equation}
\check{F}_{v_{i}}\left(\vect\phi\right)=\left(F_{D_{i}}\left(\vect\phi\right)\times\prod_{j=1}^{s_{i}}\left(F_{P_{i}^{j}}\left(\vect\phi\right)\times F_{\Sigma_{i}^{j}}\left(\vect\phi\right)\right)\right)\label{eq:even fp component-2}
\end{equation}
Now consider $1\leq i\leq r+1$ such that $\beta_{i}=1\mod2$. We
set $\mm_{D_{i}}\left(\vect\phi\right)=\mm_{\emptyset,\star_{i,1},\beta_{i}^{0}}$
and let 
\[
{F_{D_{i}}\left(\vect\phi\right)=\underline{\mbox{pt}}_{\zz/\beta_{i}^{0}}\subset\mm_{D_{i}}\left(\vect\phi\right)^{\tb}}
\]
be the unique fixed point with smooth domain passing through $p_{a_{i}^{1}}$.
We set $\mm_{\Sigma_{i}^{1}}\left(\vect\phi\right)=\mm_{\lf_{i}^{1}\coprod\star'''_{i,1},\tilde{d}_{i}^{1}}$
using the same unstable convention (\ref{eq:unstable convention-1}),
and $F_{\Sigma_{i}^{1}}\left(\vect\phi\right)=\ev_{\star'''_{i,1}}^{-1}\left(p_{a_{i}^{1}}\right)^{\tb}\subset\mm_{\lf_{i}^{1}\coprod\star'''_{i,1},\tilde{d}_{i}^{1}}^{\tb}$.
We set

\begin{equation}
\check{F}_{v_{i}}\left(\vect\phi\right)=F_{D_{i}}\left(\vect\phi\right)\times F_{\Sigma_{i}^{1}}\left(\vect\phi\right).\label{eq:odd fp component-2}
\end{equation}

Finally, we set 
\[
\check{F}_{\vect\phi}=\prod_{i=1}^{r+1}\check{F}_{v_{i}}\left(\vect\phi\right).
\]

\begin{lem}
\label{lem:fps by profiles}The open gluing maps of $\S$\ref{subsec:closed and open fps}
induce an isomorphism 
\[
\left(\check{\mm}_{\tc}^{\tb}\right)\simeq\coprod_{\vect s}\left(\coprod_{\vect\phi\in\mathcal{P}_{\vect s}\left(\tc\right)}\check{F}_{\vect\phi}\right)_{\Sym\left(\vect s\right)}.
\]
\end{lem}
\begin{proof}
This follows from $\S$\ref{subsec:closed and open fps} and the well-known
description of $\mm_{\lf^{\cc},\beta}^{\tb}$ by gluing components
(see \cite[Chapter 27]{clay-MS}, for instance).
\end{proof}
We fix\footnote{The condition $\left(g_{\basic}^{r+r'+1}\right)^{-1}\left(F_{\vp_{+}}\right)\subset\partial F_{\vp}$
specifies $\vp$ up to the $\Sym\left(\vect s\right)$ action, and
we pick a representative.} a map $\cnt:\pc_{\basic}^{r,r'+1}\to\pc_{\basic}^{r,r'}$ such that
$\left(g_{\basic}^{r+r'+1}\right)^{-1}\left(F_{\vp_{+}}\right)\subset\partial F_{\cnt\vp_{+}}$.
In this case, we denote 
\[
\partial^{\vp_{+}}F_{\vp}:=\left(g_{\basic}^{r+r'+1}\right)^{-1}\left(F_{\vp_{+}}\right),
\]
We have $\cnt\tc\left(\vp_{+}\right)=\tc\left(\cnt\vp_{+}\right)$,
and we define more generally 
\[
\partial^{\vp_{+}}F_{\vp}^{S}=\left(g_{\tc\left(\vp_{+}\right)}^{S}\right)^{-1}\left(F_{\vp_{+}}^{S\cup\left\{ r+r'+1\right\} }\right).
\]

\subsection{\label{subsec:fp for generalized forms}Fixed-point formula for extended
Forms}

For $\vp\in\pc_{\basic}^{r,r'}$, we set $F_{\vp}=\left(\grave{\For}_{\basic}^{r+r',\left\{ r+1,...,r+r'\right\} }\right)^{-1}\left(\check{F}_{\vp}\right)$
(so we remember only the tails of the even-even edges) and $\widehat{F}_{\vp}=\left(\For_{\basic}^{r+r'}\right)^{-1}\left(\check{F}_{\vp}\right)$.
We let $F_{\vp}\xrightarrow{\grave{\For}_{\vp}}\check{F}_{\vp}$ be
the pullback of $\grave{\For}_{\basic}^{r+r',\left\{ r+1,...,r+r'\right\} }$.
It follows from Proposition \ref{prop:fp inverse image} that $\widehat{F}_{\vp}$
and $F_{\vp}$ are suborbifolds with corners of $\mm_{\basic}^{r,r'}$
and of $\mm_{\basic}^{r,r',\left\{ r+1,...,r+r'\right\} }$, respectively,
and that the pullback
\[
\widehat{F}_{\vp}\xrightarrow{\acute{\For}_{\vp}}F_{\vp}
\]
of $\acute{\For}_{\basic}^{r+r',\left\{ r+1,...,r+r'\right\} }$ is
a proper submersion. We use $\ff_{\basic}^{r}$ to orient the fibers
of $\acute{\For}_{\vp}$, and define 
\[
\xi_{\vp}=\left(-1\right)^{r\left(\vp\right)\cdot r'\left(\vp\right)}\left(\acute{\For}_{\vp}\right)_{*}\left(\ed_{1}^{\basic,r+r'}\times\cdots\times\ed_{r}^{\basic,r+r'}\right)|_{\widehat{F}_{\vp}}^{*}\left(\Lambda|_{L\times L\backslash\Delta}\right)^{\boxtimes r}.
\]
Since $\Lambda$ has odd degree, (\ref{eq:ff equivariance}) implies
that 
\begin{equation}
\xi_{\tau.\vp}=\sgn\left(\tau\right)\cdot\xi_{\vp}.\label{eq:xi Sym(r) sign}
\end{equation}
A fixed point profile $\vp$ is called \emph{odd-even} if $r'\left(\vp\right)=0$.
An odd-even fixed-point profile is called \emph{sorted }if $\tc\left(\vp\right)$
is sorted.
\begin{lem}
\label{lem:xi value}For $\vp\in\pc_{\basic}^{r,0}$ a sorted odd-even
fixed point profile we have
\[
\xi_{\vp}=\prod_{j=1}^{\no}\left(\frac{\beta_{j}^{0}\left(\vp\right)}{2}\cdot\left(-\frac{\lambda_{1}\cdots\lambda_{m}}{\alpha_{a_{j}}\left(\vp\right)}\right)\right)^{\deg\left(v_{j}\right)}\cdot o_{0}^{\boxtimes r},
\]
where $o_{0}$ is a local section of $\Or\left(TL\right)$ near $p_{0}$
chosen so $\rho_{0}|_{p_{0}}=\lambda_{1}\cdots\lambda_{m}\cdot o_{0}$,
and $\deg\left(v_{j}\right)$ is the number of edges incident to $v_{j}$
in $\tc=\tc\left(\vp\right)$. In particular, $\xi_{\vp}$ is independent
of the choice of $\Lambda$. 

Moreover, for any $r'\geq0$ and $\vp\in\pc_{\basic}^{r,r'}$ we have
$\xi_{\vp}=\xi_{\cnt^{r'}\phi}$, so if $\cnt^{r'}\vp$ is sorted\emph{,}
$\xi_{\vp}$ is given by the same formula.
\end{lem}
\begin{proof}
For $1\leq j\leq\no=\no\left(\vp\right)$, Let $b_{j}=\min\left(a_{j},2m+1-a_{j}\right)$,
so $p_{a_{j}}\in\pp\left(W_{b_{j}}\otimes\cc\right)$. If $\sstar'_{i}\in\sigma_{\tc}\left(v_{j}\right)$
the map 
\[
\ed_{i}^{\basic,r}|_{\widehat{F}}:E\to L\times L
\]
factors through the inclusion 
\[
M_{j}:=\pp\left(W_{b_{j}}\right)\times\pp\left(W_{0}\right)\to L\times L,
\]
so the Cartan-Weil differential satisfies $D|_{M_{j}}=-\lambda_{b_{j}}\iota_{\xi_{b_{j}}}$
(here $\iota_{\xi_{b_{i}}}$ denotes contraction with the infinitesimal
generator of the $\tb$ action on $L\times L$). We have 
\[
D\Lambda|_{M_{j}}=-\pr_{2}^{*}\rho_{0}|_{M_{j}}=-\lambda_{1}\cdots\lambda_{m}\pr_{2}^{*}o_{0}
\]
and so 
\[
\Lambda|_{M_{j}}=\frac{\lambda_{1}\cdots\lambda_{m}}{\lambda_{b_{j}}}\cdot\frac{d\theta_{a_{j}}}{4\pi}\cdot\pr_{2}^{*}o_{0}
\]
where $d\theta_{a_{j}}$ is an angular form for $\pp_{\rr}\left(W_{b_{j}}\right)$,
oriented as the boundary of the unit disc in $V_{2m+1-a_{j}}\subset\pp_{\cc}\left(W_{b_{j}}\otimes\cc\right)$.
Simple degree considerations show that $\xi_{\vp}$ has de Rham degree
zero, and is a locally constant form given by integration on the top
dimensional strata of the fiber of $\hat{F}\xrightarrow{\acute{\For}_{\vp}}F_{\vp}$,
which has a dense open subset of the form $\prod_{j=1}^{\no}\left(\partial\Sigma_{j}^{0}\right)^{\deg\left(v_{j}\right)}$.
Pulling back $\prod\Lambda|_{M_{j}}^{\boxtimes\deg\left(v_{j}\right)}$
and integrating (cf. Proposition \ref{prop:sorted is good}) we obtain
the result.

We check that for $\vp\in\pc_{\basic}^{r,r'}$, $\xi_{\vp}=\xi_{\cnt^{r'}\vp}$.
There's an obvious identification of the fibers of $\acute{\For}_{\vp}$
and $\acute{\For}_{\cnt^{r'}\vp}$ which respects the integrand, so
it suffices to check that the orientations are the same. This follows
from coherence of $\ff_{\basic}^{r}$, see \cite[Proposition 21]{equiv-OGW-invts}.
\end{proof}

For $\vp\in\pc_{\basic}^{r,r'}$ and $S\subset\left\{ r+1,...,r+r'\right\} $,
let $F_{\vp}^{S}:=\left(\grave{\For}_{\basic}^{r+r,S}\right)^{-1}\check{F}_{\vp}$
(cf. Proposition \ref{prop:fp inverse image}) and let $N_{\vp}^{S}\xrightarrow{\pi_{\vp}^{S}}F_{\vp}^{S}$
be the normal l-bundle associated with the closed embedding $F_{\vp}^{S}\to\mm_{\tc\left(\vp\right)}^{S}$.
We denote $N_{\vp}=N_{\vp}^{\left\{ r+1,...,r+r'\right\} }$ and $\check{N}_{\vp}=N_{\vp}^{\emptyset}$.
Consider the short exact sequence associated with smoothing the $\left(\star_{i,j},\star'_{i,j}\right)$
nodes (cf. (\ref{eq:normal to glued}))
\begin{equation}
0\to M_{\vp}^{S}\xrightarrow{\mu_{\vp}^{S}}N_{\vp}^{S}\xrightarrow{\nu_{\vp}^{S}}\check{S}_{\vp}^{S}\to0,\label{eq:M-N-S ses}
\end{equation}
where 
\begin{equation}
S_{\vp}=\bigoplus_{i=\no+1}^{r+1}\bigoplus_{j=1}^{s_{i}}\mathbb{L}_{\star_{i,j}}^{\vee}\boxtimes\mathbb{L}_{\star'{}_{i,j}}^{\vee}=\bigoplus_{i,j}\mathbb{L}_{\star_{i,j}}^{\vee}\left(\frac{-\alpha_{a_{i}^{j}}}{d_{i}^{j}}\right).\label{eq:S_Gamma direct sum decomposition}
\end{equation}

Let $\gamma_{\vp_{+}}^{S}:N_{\vp}^{S}\to N_{\vp_{+}}^{S\coprod\left\{ r+r'+1\right\} }$
and $\grave{f}_{\vp}^{S}:N_{\vp}^{S}\to\check{N}_{\vp}$ denote the
linearizations of $g_{\tc\left(\vp_{+}\right)}^{S}$ and $\grave{\For}_{\vp_{+}}^{S}$
respectively. The sequences (\ref{eq:M-N-S ses}) and the decomposition
(\ref{eq:S_Gamma direct sum decomposition}) are natural with respect
to the forgetful maps and attaching maps. In particular, for every
$\vp_{+}\in\pc_{\basic}^{r,r'+1}$ with $\cnt\left(\vp_{+}\right)=\vp\in\pc_{\basic}^{r,r'}$
we have maps of short exact sequences
\begin{equation}
\xymatrix{0\ar[r] & \partial^{\vp_{+}}\check{M}_{\vp}\ar[r]\ar[d]_{\gamma_{M_{\vp_{+}}}^{\emptyset}} & \partial^{\vp_{+}}\check{N}_{\vp}\ar[r]\ar[d]^{\gamma_{\vp_{+}}^{\emptyset}} & \partial^{\vp_{+}}\check{S}_{\vp}\ar[r]\ar[d]^{\gamma_{S_{\vp_{+}}}^{\emptyset}} & 0\\
0\ar[r] & M_{\vp_{+}}^{\left\{ R\right\} }\ar[d]_{\grave{f}_{M_{\vp_{+}}}^{\left\{ R\right\} }}\ar[r]^{\mu_{\vp_{+}}^{\left\{ R\right\} }} & N_{\vp_{+}}^{\left\{ R\right\} }\ar[d]^{\grave{f}_{\vp_{+}}^{\left\{ R\right\} }}\ar[r]^{\nu_{\vp_{+}}^{\left\{ R\right\} }} & S_{\vp_{+}}^{\left\{ R\right\} }\ar[d]^{\grave{f}_{S_{\vp_{+}}}^{\left\{ R\right\} }}\ar[r] & 0\\
0\ar[r] & \check{M}_{\vp_{+}}\ar[r] & \check{N}_{\vp_{+}}\ar[r] & \check{S}_{\vp_{+}}\ar[r] & 0
}
\label{eq:M-N-S check gamma}
\end{equation}
where $R=r+r'+1$. We set $\check{\gamma}_{\vp_{+}}:=\grave{f}_{\vp_{+}}^{\left\{ R\right\} }\gamma_{\vp_{+}}^{\emptyset},\,\check{\gamma}_{M_{\vp_{+}}}=\grave{f}_{M_{\vp_{+}}}^{\left\{ R\right\} }\gamma_{M_{\vp_{+}}}^{\emptyset}$
and $\check{\gamma}_{S_{\vp_{+}}}=\grave{f}_{S_{\vp_{+}}}^{\left\{ R\right\} }\gamma_{S_{\vp_{+}}}^{\emptyset}$. 

Let
\begin{equation}
N_{0}:=\left(N_{\Delta}\right)_{\left(p_{0},p_{0}\right)}\simeq T_{p_{0}}L,\label{eq:N_0 def}
\end{equation}
be the fiber of the normal bundle to the diagonal $L\xrightarrow{\Delta}L\times L$,
and for $r+i\in S$ let
\[
\delta_{\vp}^{i,S}:N_{\vp}\to N_{0}\times F_{\vp}^{S}
\]
be the linearization of $\ed_{i}^{\basic,r+r'}\circ\acute{\For}_{\basic}^{r,S}$.
We set $\delta_{M_{\vp}}^{i,S}=\delta_{\vp}^{i,S}\,\mu_{\vp}^{i,S}$.
Since $M_{\vp}$ corresponds to nodal configurations, for which the
even degree discs map to a point, there exists a map 
\[
{\check{\delta}_{M_{\vp}}^{i}=\delta_{M_{\vp}}^{i,\emptyset}:M_{\vp}\to N_{0}\times\check{F}_{\vp}}
\]
such that $\delta_{M_{\vp}}^{i,S}$ is the pullback of $\check{\delta}_{M_{\vp}}^{i}$.
To be precise, ${\left(\id\times\grave{\For}_{\vp}^{i,S}\right)\delta_{M_{\vp}}^{i,S}=\check{\delta}_{M_{\vp}}^{i}\grave{f}_{M_{\vp}}^{S}}$.
It follows from that there's a cartesian map 
\[
h_{M_{\vp_{+}}}:\partial^{\vp_{+}}\check{M}_{\vp}\to\ker\left(\check{\delta}_{M_{\vp_{+}}}^{r'+1}\right)
\]
such that $i^{\ker\check{\delta}_{M_{\vp_{+}}}^{r'+1}}h_{M_{\vp_{+}}}=\grave{f}_{\vp_{+}}^{\left\{ r'+1\right\} }\gamma_{M_{\vp_{+}}}$.
We emphasize that for general $\vp$, $\delta_{\vp}^{i}$ is \emph{not
}the pullback of any map $\check{N}_{\vp}\to N_{0}\times\check{F}_{\vp}$.

Recall that if $E\xrightarrow{}B$ is a $\tb$-equivariant bundle,
with $S\left(E\right)\xrightarrow{\pi_{S}}B$ the associated sphere
bundle, a form
\begin{equation}
\theta\in\Omega\left(S\left(E\right),\pi_{S}^{*}\Or\left(E\right)\otimes\rr\left[\vec{\lambda}\right]\right)^{\tb}\label{eq:N's angular form}
\end{equation}
is an \emph{equivariant angular form }if $\pi_{S*}\theta=1$ and $D\theta=-\pi_{S}^{*}e$
for some form\emph{ }$e$ on $B$ which we call the \emph{equivariant
Euler form} associated with $\theta$ (see \cite{twA8}; the discussion
extends to the case where the base is an orbifold; here we also allow
the rank of the bundle $E$, and the degrees of the corresponding
forms, to vary). If the action on $B$ is trivial and the action
on $E\backslash B$ is fixed-point free, then $e$ is invertible.

Consider some odd-even $\vp\in\pc_{\basic}^{r,0}$. We say an equivariant
Euler $e\left(M_{\vp}\right)$ for $M_{\vp}=\check{M}_{\vp}$ is \emph{canonical
}if it is associated with an equivariant angular form $\theta_{M}$
on $S\left(M_{\vp}\right)$ such that for all $\vp_{+}$ with $\cnt\vp_{+}=\vp$,
we have 
\begin{equation}
\left(i_{S\left(M_{\vp}\right)}^{\partial^{\vp_{+}}}\right)^{*}\theta_{M}\in\im S\left(h_{M_{\vp_{+}}}\right){}^{*}.\label{eq:can theta_M}
\end{equation}
Here $S\left(h_{M_{\vp_{+}}}\right):\partial^{\vp_{+}}S\left(M_{\vp}\right)\to S\left(\ker\check{\delta}_{M_{\vp_{+}}}^{r'+1}\right)$
is the map of sphere bundles induced from the cartesian map $h_{M_{\vp_{+}}}$.
We say an equivariant Euler form $e\left(S_{\vp}\right)$ for $S_{\vp}$
is canonical if it is associated with an angular form $\theta_{S}$
on $S\left(S_{\vp}\right)$ such that for all $\vp_{+}$ with $\cnt\vp_{+}=\vp$
we have 
\begin{equation}
\left(i_{S\left(\check{S}_{\vp}\right)}^{\partial^{\vp_{+}}}\right)^{*}\theta_{S}\in\im S\left(\check{\gamma}_{S_{\vp_{+}}}\right).\label{eq:can theta_S}
\end{equation}

\begin{thm}
\label{thm:general fixed point formula}Let $\basic$ be a basic moduli
specification and let $\omega=\left\{ \check{\omega}_{r}\right\} \in\Omega_{\basic}$
be an extended form with $D\omega=0$. For every $r\geq0$ and $\vp\in\pc_{\basic}^{r,0}$,
canonical Euler forms $e\left(M_{\vp}\right)$ and $e\left(S_{\vp}\right)$
exist, and for any choice of such forms we have

\begin{equation}
\int_{\basic}\omega=\sum_{r\geq0}\sum_{\vp\in\pc_{\basic}^{r,0}}\frac{\xi_{\vp}}{r!\boldsymbol{s}\left(\vp\right)!}\int_{\check{F}_{\vp}}e\left(M_{\vp}\right)^{-1}e\left(S_{\vp}\right)^{-1}\check{\omega}_{r}.\label{eq:general fp formula}
\end{equation}
where $\boldsymbol{s}\left(\vp\right)!:=\prod s_{i}\left(\vp\right)!$
\end{thm}
\begin{cor}
\label{cor:Lambda independence}The integral $\int_{\basic}\omega$
is independent of the choice of equivariant homotopy kernel $\Lambda$.
\end{cor}

\begin{rem}
\label{rem:sum over orbits}We can rewrite (\ref{eq:general fp formula})
as a sum over isomorphism types of fixed-point profiles. More precisely,
choose a subset $\wc\subset\pc_{\basic}$ of an even-odd profiles
such that, for every $r\geq0$, $\tc\in\ts_{\basic}^{r,0}$ with $\no$
odd vertices, and $\boldsymbol{s}=\left(s_{\no+1},...,s_{r+1}\right)\in\zz_{\geq0}^{r-\no+1}$,
$\wc$ intersects every orbit of the $\Sym\left(r\right)\times\Sym\left(\boldsymbol{s}\right)$
action on $\pc_{\boldsymbol{s}}\left(\tc\right)$ in precisely one
element. By (\ref{eq:check J equivariance}) and (\ref{eq:xi Sym(r) sign})
we can rewrite (\ref{eq:general fp formula}) as
\[
\int_{\basic}\omega=\sum_{\vp\in\wc}\frac{\xi_{\vp}}{\Aut\vp}\int_{\check{F}_{\vp}}e\left(M_{\vp}\right)^{-1}e\left(S_{\vp}\right)^{-1}\check{\omega}_{r}.
\]
where $\Aut\vp<\Sym\left(r\right)\times\prod\Sym\left(s_{i}\left(\vp\right)\right)$
is the stabilizer subgroup of $\vp$. This has obvious computational
advantages.
\end{rem}

\subsection{\label{subsec:Equivariant-open-Gromov-Witten}Computing equivariant
open Gromov-Witten invariants}

In this section we apply Theorem \ref{thm:general fixed point formula}
to compute the equivariant open Gromov-Witten invariants. We write
an explicit formula for the fixed-point contributions in Proposition
\ref{prop:ogw contribution by fp profile}, use this to compute a
few examples, and finally derive the generating function expression
presented in Theorem \ref{thm:open localization}.

\subsubsection{Fixed-point contributions}

For $\omega$ given by (\ref{eq:ogw omega}) the contributions to
(\ref{eq:general fp formula}) can be computed in terms of the closed
correlators and descendent integrals, as follows.

If $f$ is a power series in $x,y,...$ we use the notation $f\left[x^{a}y^{b}\cdots\right]=\left(\partial_{x}^{a}\partial_{y}^{b}f\right)_{x=0,y=0,...}$.
This gives the coefficient of $\frac{x^{a}}{a!}\frac{y^{b}}{b!}...$.
Let $\mathcal{P}_{\basic}^{+}\subset\pc_{\basic}$ denote the set
of odd-even\emph{ }fixed point profiles such that $\forall j\in\coprod_{i=\no+1}^{r+1}\lf_{i}^{0}$
we have $d_{j}=0$. The contribution of $\check{F}_{\vect\phi}$ vanishes
unless $\vect\phi\in\mathcal{P}_{\basic}^{+}$, since $\eta|_{p_{0}}=\rho_{\tb}\left(\alpha_{0}\right)=0$.
Set $k_{i}=\left|\kf_{\Gamma}\left(v_{i}\right)\right|,l_{i}=\left|\lf_{i}^{0}\right|+s_{i}$. 
\begin{prop}
\label{prop:ogw contribution by fp profile}If $\vect\phi\in\mathcal{P}_{\basic}^{+}$
is sorted, we have

\begin{multline}
\xi_{\vp}\int_{\check{F}_{\vp}}e\left(M_{\vp}\right)^{-1}e\left(S_{\vp}\right)^{-1}\check{\omega}_{r}=\\
\left(-1\right)^{\binom{\no}{2}\left(1+m\right)}\prod_{i=\no+1}^{r+1}\left(\frac{\left(\lambda_{1}\cdots\lambda_{m}\right)^{\hat{\kf}_{i}}}{\lambda_{1}\cdots\lambda_{m}}\sum_{\vect b\in\zz_{\geq0}^{s_{i}}}\prod_{j=1}^{s_{i}}\left(-\frac{d_{i}^{j}}{\alpha_{a_{i}^{j}}}\right)^{b_{j}+1}F_{0}^{o}\left[t_{b_{1}}\cdots t_{b_{s_{i}}}s_{0}^{\left|\kf_{i}\right|}\right]\right)\times\\
\times\left(\prod_{i=\no+1}^{r+1}\prod_{j=1}^{s_{i}}\frac{\left(-1\right)^{d_{i}^{j}}\left(d_{i}^{j}\right)^{2d_{i}^{j}-1}}{\left(d_{i}^{j}!\right)^{2}\left(\alpha_{a_{i}^{j}}\right)^{2d_{i}^{j}}}\cdot\frac{\prod_{a=1}^{2m}\alpha_{a}}{\prod_{\begin{array}[t]{c}
a'\neq a_{i}^{j},0\\
0\leq c\leq d_{i}^{j}
\end{array}}\left(\frac{c}{d_{i}^{j}}\alpha_{a_{i}^{j}}-\alpha_{a'}\right)}\right)\times\\
\times\left(\prod_{i=1}^{\no}\frac{\left(\beta_{i}^{0}\right)^{\beta_{i}^{0}-1}}{\beta_{i}^{0}!\left(2\alpha_{a_{i}}\right)^{\beta_{i}^{0}}\prod_{\begin{array}[t]{c}
a'\neq a_{i},2m+1-a_{i}\\
0\leq c\leq\frac{\beta_{i}^{0}-1}{2}
\end{array}}\left(\frac{2c-\beta_{i}^{0}}{\beta_{i}^{0}}\alpha_{a_{i}}-\alpha_{a'}\right)}\cdot\left(-\frac{\beta_{i}^{0}}{2}\frac{\lambda_{1}\cdots\lambda_{m}}{\alpha_{a_{i}}}\right)^{\deg\left(v_{i}\right)}\right)\times\\
\times\prod_{\begin{array}[t]{c}
1\leq i\leq r+1\\
1\leq j\leq s_{i}
\end{array}}\left(2\tilde{d}_{i}^{j}\right)!^{-1}Z_{a_{i}^{j}}\left(\hbar=\alpha_{a_{i}^{j}}/d_{i}^{j}\right)\left[q^{2\tilde{d}_{i}^{j}}\prod_{x\in\lf_{i}^{j}}\eta_{e_{x}}\right]\label{eq:F_phi^Gamma contribution-2}
\end{multline}
Here $\deg\left(v_{i}\right)$ is the number of edges of $\tc\left(\vp\right)$
incident to $v_{i}$. $\hat{\kf}_{i}:=\kf_{i}\backslash\left\{ \sstar''_{\bullet}\right\} $,
so $\kf=\coprod_{i=1}^{r_{0}}\hat{\kf}_{i}$. 
\end{prop}
\begin{proof}
We discuss orientations. The sign $\left(-1\right)^{\binom{\no}{2}\left(1+m\right)}$
comes from Proposition \ref{prop:sorted is good}, and allows us to
treat each moduli factor $\check{\mm}_{\sfr}$ separately. Consider
some $\sfr=\left(\left(\kf,\lf,\beta\right),\sigma\right)$. If $\beta=1\mod2$
then $\check{\jc}_{\sfr}$ determines an orientation for $\Or\left(T\check{\mm}_{\sfr}\right)$.
If $\beta=0\mod2$ we use the local section $o_{0}$ of $\Or\left(TL\right)|_{p_{0}}$
to orient $\bigotimes_{x\in\kf}\ev_{x}^{-1}\Or\left(TL\right)|_{\check{F}_{\sfr}}$
and then $\check{\jc}_{\sfr}$ determines an orientation for $\Or\left(T\check{\mm}_{\sfr}\right)$.
The fixed-point components $F_{P_{i}^{j}},F_{\Sigma_{i}^{j}}$ and
$F_{D_{i}}$ for $1\leq i\leq\no$ are oriented by the complex structure
or as zero-dimensional orbifolds. The even degree disc fixed-point
components, $F_{D_{i}}$ for $\no+1\leq i\leq r+1$, are oriented
by \cite[Lemma 2.22]{JRR}. (\ref{eq:S_Gamma direct sum decomposition})
is naturally oriented by the complex structure. These conventions
determine an orientation for $N_{\vp}$, and $M_{\vp}$ (see (\ref{eq:M-N-S ses})).
Thus we consider $e\left(M_{\vp}\right),e\left(S_{\vp}\right)$ as
taking values in the trivial local system. Since (\ref{eq:gluing ls map})
commutes with $\check{\jc}_{\bullet}$, we can use $\check{\jc}_{\bullet}$
to orient just the normal bundles $N_{D_{i}}\left(\vect\phi\right)$
associated with $F_{D_{i}}\to\mm_{D_{i}}$ for $1\leq i\leq r+1$,
using the complex orientation for everything else. For $\no+1\leq i\leq r+1$
we have $N_{D_{i}}\left(\vect\phi\right)\simeq T_{p_{0}}L$ and the
orientation $\check{\jc}_{\bullet}$ agrees with orientation $o_{0}$.

We compute $e\left(S_{\vp}\right)$. Consider (\ref{eq:S_Gamma direct sum decomposition}).
Note that if we forget the $\tb$-action, $\mathbb{L}_{\star_{i,j}}^{\vee}\left(\frac{-\alpha_{a_{i}^{j}}}{d_{i}^{j}}\right)=\mathbb{L}_{\star_{i,j}}^{\vee}$.
Consider some $\vp_{+}\in\pc_{\basic}^{r,1}$ with $\cnt\vp_{+}=\vp$.
Write $\tc_{+}=\tc\left(\vp_{+}\right)$. Let $v_{i},v_{j}\in\left(\tc_{+}\right)_{0}$
denote the endpoints of the $\left(r+1\right)$'st edge $e$. We have
$\beta_{\tc_{+}}\left(v_{i}\right)=\beta_{\tc_{+}}\left(v_{j}\right)=0\mod2$,
so the orientation of $e$ is determined by the parity of $\left|\kf_{\tc_{+}}\left(v_{i}\right)\right|$
and so the boundary condition $\check{\gamma}_{S_{\vp_{+}}}$ agrees
with the map used to define the descendent integrals of discs (cf.
$\S$\ref{subsec:descendent integrals}). Since $\check{\gamma}_{S_{\vp_{+}}}$
respects the direct sum decomposition (\ref{eq:S_Gamma direct sum decomposition})
we can construct a canonical angular form $\theta_{S}$ as the join\footnote{Let $E_{0},E_{1}$ be vector bundles over $B$ of even rank. We have
a map

\[
S\left(E_{0}\right)\times_{B}S\left(E_{1}\right)\times\left[0,1\right]_{t}\xrightarrow{q}S\left(E_{0}\oplus E_{1}\right),
\]

and we define $\theta_{0}\star\theta_{1}$ by
\[
q^{*}\left(\theta_{0}\star\theta_{1}\right)=\left(1-t\right)\,d\theta_{0}\,\theta_{1}+t\,\theta_{0}\,d\theta_{1}+\theta_{0}\,\theta_{1}\,dt.
\]
The case of $n$ direct summands can be defined recursively (a more
symmetric definition can also be given along these lines).} 
\[
\theta_{S}=\bigstar_{i,j}\theta_{i,j}
\]
of $\tb$-invariant canonical angular 1-forms $\theta_{i,j}$ for
$\mathbb{L}_{\star_{i,j}}^{\vee}$. It follows that 
\[
e\left(S_{\vp}\right)=\prod_{i=\no+1}^{r+1}\prod_{j=1}^{s_{i}}\left(-\frac{\alpha_{a_{i}^{j}}}{d_{i}^{j}}-\psi_{i,j}\right).
\]

We turn our attention to $e\left(M_{\vp}\right)$. We denote
\[
\mathcal{R}_{\vp}:=\prod_{i=1}^{\no\left(\vp\right)}\mm_{\basic_{\tc\left(\vp\right)}\left(v_{i}\right)}\times\prod_{i=\no+1}^{r+1}\mm_{\lf_{i}\coprod\left\{ \star'_{i,j}\right\} _{j=1}^{s_{i}},\beta_{i}/2}
\]
(this parameterizes everything but the energy zero disc components).
Let $N_{\mathcal{R}_{\vp}}$ be the normal bundle to the fixed points
$\rc_{\vp}^{\tb}\to\rc_{\vp}$. By (\ref{eq:even discs mod and fp})
the normal bundle to $F_{D_{i}}\left(\vect\phi\right)\to\mm_{D_{i}}\left(\vect\phi\right)$
for $\no+1\leq i\leq r+1$ is $TL$, and we obtain a short exact sequence
\[
0\to\check{M}_{\vp}\to T_{p_{0}}L^{\oplus n_{\text{even}}}\boxplus N_{\mathcal{R}_{\vp}}\to\left(T_{p_{0}}X\right)^{N}\to0.
\]
of bundles over $\check{F}_{\Gamma}\left(\vect\phi\right)$ where
$N=\sum_{i=\no+1}^{r+1}s_{i}$. For every $\vp_{+}$ with $\cnt\vp_{+}=\vp$,
we have a map of short exact sequences
\begin{equation}
\xymatrix{0\ar[r] & \partial^{\vp_{+}}\check{M}_{\vp}\ar[r]\ar[d]_{h_{M_{\vp_{+}}}} & T_{p_{0}}L^{n_{\text{even}}}\boxplus N_{\mathcal{R}_{\vp}}\ar[r]^{\alpha_{\vp}}\ar[d] & \ar@{=}[d]\left(T_{p_{0}}X\right)^{N}\ar[r] & 0\\
0\ar[r] & \ker\check{\delta}_{M_{\vp_{+}}}^{r+1}\ar[r] & T_{p_{0}}L^{\left(n_{\text{even}}+1\right)}\boxplus N_{\mathcal{R}_{\vp_{+}}}\ar[r]_{\alpha_{\vp_{+}}} & \left(T_{p_{0}}X\right)^{N}\ar[r] & 0
}
,\label{eq:detach even discs coherence}
\end{equation}
where the middle vertical map is the direct sum of (i) the diagonal
inclusion $T_{p_{0}}L^{n_{\text{even}}}\to T_{p_{0}}L^{n_{\text{even}}+1}$
corresponding to the edge $r+1$ and (ii) an isomorphism
\begin{equation}
N_{\mathcal{R}_{\vp}}\simeq N_{\mathcal{R}_{\vp_{+}}}\label{eq:N^R identification}
\end{equation}
induced by the obvious identification $\mathcal{R}_{\vp}\simeq\mathcal{R}_{\vp_{+}}$.
Use the $GL\left(2m+1,\cc\right)^{N}$ action on $\mathcal{R}_{\vp}$
to fix a map $\bigoplus_{i,j}\ev_{\star'_{i,j}}^{*}T_{p_{0}}X\xrightarrow{\sigma}N_{\rc_{\vp}}$
so $\alpha_{\vp}\circ\left(\sigma,0\right)=\id$. By (\ref{eq:N^R identification})
$\left(\sigma,0\right)$ defines a section of $\alpha_{\vp_{+}}$
compatible with (\ref{eq:detach even discs coherence}). It follows
that there exists a canonical equivariant angular form $\theta_{M}$
for $M_{\vp}$ with
\[
e\left(M_{\vp}\right)=\frac{e_{\mathcal{R}}\cdot\left(\prod_{j=1}^{m}\lambda_{j}\right)^{n_{\text{even}}}}{\left(\prod_{j=1}^{2m}\alpha_{j}\right)^{N}}
\]
where $e_{\mathcal{R}}$ is an equivariant Euler form for $N_{\rc_{\vp}}$,
pulled back from the orbifold without boundary $\rc_{\vp}^{\tb}$. 

The computation of $e_{\mathcal{R}}$ is standard. See, for example,
\cite[Chapter 27]{clay-MS}. We use a computation similar to \cite[Lemma 6]{disc-enumeration}
to fix the weights for the normal bundle to the odd disc moduli (this
determines the sign of the term on the fourth line of (\ref{eq:F_phi^Gamma contribution-2})).
We stop computing after we express $e_{\mathcal{R}}$ in terms of
the equivariant Euler forms to the normal bundle $F_{\Sigma_{i}^{j}\left(\vect\phi\right)}\to\mm_{\Sigma_{i}^{j}\left(\vect\phi\right)}$.

We integrate. The term $\sum_{\vect b\in\zz_{\geq0}^{s_{i}}}\prod_{j=1}^{s_{i}}\left(-\frac{d_{i}^{j}}{\alpha_{a_{i}^{j}}}\right)^{b_{j}+1}F_{0}^{o}\left[t_{b_{1}}\cdots t_{b_{s_{i}}}s_{0}^{\left|\kf_{i}\right|}\right]$
comes from expanding $e\left(S_{\vp}\right)^{-1}=\prod_{j=1}^{s_{i}}\frac{d_{i}^{j}}{-\alpha_{a_{i}^{j}}-\psi_{i,j}}$
as a power series in $\psi_{i,j}$ and integrating on $F_{D_{i}}\left(\vect\phi\right)$.
The closed correlator $Z_{a_{i}^{j}}\left(\hbar=\alpha_{a_{i}^{j}}/d_{i}^{j}\right)\left[q^{2\tilde{d}_{i}^{j}}\prod_{x\in\lf_{i}^{j}}\eta_{e_{x}}\right]$
comes from summing over the contributions of $F_{\Sigma_{i}^{j}}\left(\vect\phi\right)\subset\mm_{\Sigma_{i}^{j}}\left(\vect\phi\right)$
(keeping the convention (\ref{eq:unstable convention-1}) in mind).The
factor $\prod_{i=1}^{\no}\left(-\frac{\beta_{i}^{0}}{2}\frac{\lambda_{1}\cdots\lambda_{m}}{\alpha_{a_{i}}}\right)^{\deg\left(v_{i}\right)}$
comes from Lemma \ref{lem:xi value}.
\end{proof}

\begin{proof}
[Proof of theorem \ref{thm:open localization}]Recall that ${E\in\widehat{\rc}:=\mathcal{R}\left[\left[u^{-1},u,\nu_{a,d},q,\eta_{0},...,\eta_{2m},s,s_{0},t_{0},t_{1},...\right]\right]}$
was given by
\begin{multline*}
E=\underbrace{\exp\left(\sum_{a\neq0}\sum_{d>0}Z_{a}\left(\frac{\alpha_{a}}{d}\right)\overbrace{\partial_{\nu_{a,d}}}^{\mbox{C}}\right)}_{\mbox{type IV}}\underbrace{\exp\left(\left(\sqrt{-1}\right)^{\left(1+m\right)}\sum_{a\neq0}\sum_{d\mbox{ odd}}uq^{d}\nu_{a,d/2}\frac{d^{d-1}}{d!\left(2\alpha_{a}\right)^{d}}\frac{\exp\left(\frac{d}{2}\cdot\frac{\lambda_{1}\cdots\lambda_{m}}{u\alpha_{a}}\cdot\overbrace{\partial_{s^{o}}}^{\mbox{B}}\right)}{\prod_{0\leq c\leq\frac{d-1}{2},b\neq a,\overline{a}}\left(\frac{2c-d}{d}\alpha_{a}-\alpha_{b}\right)}\right)}_{\mbox{type III}}\\
\underbrace{\exp\left(\sum_{a\neq0}\sum_{d>0}q^{2d}\nu_{a,d}\frac{\left(-1\right)^{d}d^{2d-1}\cdot\left(\prod_{a\in\left\{ 1,...2m\right\} }\alpha_{a}\right)\cdot\left(\sum_{i\geq0}\left(-\frac{d}{\alpha_{a}}\right)^{i+1}\overbrace{\partial_{t_{i}}}^{\mbox{A}}\right)}{\left(d!\right)^{2}\alpha_{a}^{2d}\prod_{0\leq c\leq d,b\neq a,0}\left(\frac{c}{d}\alpha_{a}-\alpha_{b}\right)}\right)}_{\mbox{type II}}\\
\underbrace{\exp\left(\frac{u}{\lambda_{1}\cdots\lambda_{m}}\exp\left(s\cdot\lambda_{1}\cdots\lambda_{m}\cdot\partial_{s^{o}}\right)\exp\left(\eta_{0}\partial_{t_{0}}\right)F_{0}^{o}\right)}_{\mbox{type I}}
\end{multline*}
Expanding the exponents and the generating functions, we can express
$E$ as a sum over diagrams with vertices of type I,II,III or IV and
edges of type A,B or C, in accordance with the labeling above. The
vertices and edges of a given type are numbered, and assigned additional
data as necessary to specify their contribution to $E$.

$E-1$ belongs to the ideal $I\triangleleft\widehat{\mathcal{R}}$
generated by $q,\nu_{a,d},\eta_{0},...,\eta_{2m},s,s^{o},t_{0},t_{1},...$,
and if we let $E_{0}\in I$ denote the sum over nonempty \emph{connected}
diagrams, then $\left(E-1\right)|_{\nu_{a,d}=0}=\sum_{a\geq1}\frac{1}{a!}\left(E_{0}|_{\nu_{a,d}=0}\right)^{a}$.
It follows that $\partial_{u}|_{u=0}\log\left(E|_{\nu_{a,d}=0,t_{i}=0,s^{o}=0}\right)$
is given by a sum over diagrams whose underlying graph is a tree.
The contribution of each tree diagram is easily identified with the
contribution of some (non-unique) fixed-point profile $\vp\in\pc_{\basic}$
to (\ref{eq:F_phi^Gamma contribution-2}). This many-to-many relation
between fixed-point profiles and tree diagrams can be turned into
a bijection of sets by adding labelings to both the profiles and to
the tree diagrams. In other words, we have a correspondence 
\[
\pc_{\basic}\leftarrow\left\{ \text{common denominator diagrams}\right\} \rightarrow\left\{ \text{tree diagrams}\right\} 
\]
and equation (\ref{eq:ogw loc formula}) follows by computing the
size of the fiber of the two maps in this correspondence.
\end{proof}

\subsection{\label{subsec:Examples-and-applications.}Examples and applications}

Let $m=1$ and consider invariants $I_{1}\left(k,\left(0,l\right),\beta\right)$
with $\deg I_{1}\left(k,\left(0,l\right),\beta\right)=0$. Let $\basic=\left(\kf,\lf,\beta\right)$
be a basic moduli specification with $\left|\kf\right|=k,\left|\lf\right|=l$.
As discussed in \cite{equiv-OGW-invts} (see Remark 3 and the discussion
at the end of $\S1$) we have 
\[
I_{1}\left(k,\left(0,0,l\right),\beta\right)=\int_{\mm_{\basic}}\omega_{0},
\]
so in this case, the contributions of $\mm_{\basic}^{r}$ with $r\geq1$
in (\ref{eq:extended integral}) vanish. Moreover, by introducing
non-equivariant deformations of $\omega_{0}$ this integral can be
interpreted as a (signed) count of discs or twice Welschinger's count
of rational real curves \cite{welschinger}, with suitable constraints.

\begin{example}
\label{exa:lines thru 2pts}Consider 
\[
I=I_{1}\left(2,\left(0,0,0\right),1\right)=2,
\]
twice the number of lines through a pair of points in the plane. Let
$\basic=\left(\left\{ 1,2\right\} ,\emptyset,1\right)$ be the corresponding
basic moduli specification. Note that $\mm_{\basic}$ has no fixed
points, but $\partial\mm_{\basic}\neq\emptyset$, see Remark \ref{rem:boundary contributes}.
The only odd-even tree for $\basic$ is 
\[
\Gamma=\underbrace{\left(1\right)}_{v_{1}}\to\underbrace{\left(0\right)..}_{v_{2}}\text{ with }
\]
\[
\sfr_{\Gamma}\left(v_{1}\right)=\left(\left(\emptyset,\emptyset,1\right),\sstar'_{1}\right),\;\sfr_{\Gamma}\left(v_{2}\right)=\left(\left(\left\{ \sstar''_{1},1,2\right\} ,\emptyset,0\right),\emptyset\right).
\]
\[
\mathcal{P}\left(\Gamma\right)=\left\{ \vect\phi_{1},\vect\phi_{2}\right\} 
\]
where $\vp_{i}$ is specified by the map of the $v_{1}$ disc, $a_{1}^{1}\left(\vp_{i}\right)=i$.
$F_{\Gamma}\left(\vect\phi_{i}\right)$ is an isolated fixed point,
and 
\begin{multline*}
\int_{F_{\Gamma}\left(\vect\phi_{i}\right)}\sigma_{\Gamma}\cdot\check{\omega}_{1}\cdot\xi_{\Gamma}^{\Lambda}=\frac{\lambda_{1}^{2}}{\lambda_{1}}\underbrace{F_{0}^{o}\left[s_{0}^{3}\right]}_{=2}\times\frac{1}{2\alpha_{i}\cdot\left(-\alpha_{i}\right)}\cdot\left(-\frac{1}{2}\frac{\lambda_{1}}{\alpha_{i}}\right)^{1}\times\underbrace{Z_{a_{i}}\left(\hbar=2\alpha_{a_{i}}\right)\left[1\right]}_{=2\alpha_{a_{i}}}=1
\end{multline*}
so 
\[
I=\int_{\mm_{\basic}}\omega_{0}=2,
\]
as expected.
\end{example}
\begin{example}
\label{exa:relation}Consider ${I_{m}\left(1,\emptyset,2\right)\in\rr\left[\lambda_{1},...,\lambda_{m}\right]}$.
Since 
\[
\deg\left(\check{\omega}_{r}\cdot\Lambda^{\boxtimes r}\right)-\dim\check{\mm}_{\basic}^{r}=2m-\left[\left(2m+1\right)\cdot2+\left(2m-3\right)+1\right]=-4m<0
\]
we find that $\int_{\basic}\omega=0$. On the other hand, this can
be expressed as the sum of contributions from the following two $\Sym\left(r\right)$
orbits (cf. Remark \ref{rem:sum over orbits}) of odd-even diagrams:
\begin{eqnarray*}
\Gamma_{1}: & \left(2\right). & \text{}\\
\Gamma_{2}: & \left(1\right)-\left(0\right).-\left(1\right)
\end{eqnarray*}
Note $\Gamma_{1}$ has $\left(-1\right)^{\left(m+1\right)\binom{o}{2}}=+1$,
whereas $\Gamma_{2}$ has $\left(-1\right)^{\left(m+1\right)\binom{o}{2}}=\left(-1\right)^{m+1}$
and $\left|\Aut\Gamma_{2}\right|=2$. In sum we have
\begin{multline*}
0=\sum_{a_{1}^{1}\in\left\{ 1,...,2m\right\} }\frac{\prod_{a=1}^{2m}\alpha_{a}}{\left(\alpha_{a_{1}^{1}}\right)^{2}\prod_{a'\neq0,a_{1}^{1}}\left[\left(\alpha_{a_{1}^{1}}-\alpha_{a'}\right)\left(-\alpha_{a'}\right)\right]}+\\
\frac{1}{2}\cdot\left(-1\right)^{m+1}\cdot2\times\left(\sum_{a_{2}\in\left\{ 1,...,2m\right\} }\left(\frac{1}{\prod_{a'\not\in\left\{ a_{2},2m+1-a_{2}\right\} }\left(-\alpha_{a'}-\alpha_{a_{2}}\right)}\cdot\left(-\frac{\lambda_{1}\cdots\lambda_{m}}{2\alpha_{a_{2}}}\right)\right)\right)^{2}.
\end{multline*}
We used a computer to verify that this relation holds for $1\le m\leq8$.
\end{example}

\section{\label{sec:steepest descent}Fixed-point Localization by Steepest
Descent}

In this section we use steepest descent to compute the limit of certain
integrals on an orbifold $\xx$ in terms of the asymptotic behavior
near an l-orbifold $\zc\subset\xx$. The limit is expressed using
the \emph{equivariant Segre form}, whose relation to the equivariant
Euler of the normal bundle to $\zc\subset\xx$ and other properties
will be discussed in $\S$\ref{subsec:Equivariant-Segre-forms.}.
The classical localization formula for $\partial\xx=\emptyset$ is
then derived as an easy corollary $\S$\ref{subsec:closed fp formula}.

Let $\xx$ be a compact $\tb$-orbifold with corners, $\zc=\coprod_{n\geq0}\zc_{n}$
a local $\tb$-orbifold with corners, and $i:\mathcal{Z}\hookrightarrow\xx$
an equivariant closed embedding. Let $\pi:N\to\zc$ denote the associated
normal l-bundle $N=i^{*}T\xx/T\zc$, with $\mu_{a}:N\to N$ scalar
multiplication by $a\in\rr_{\geq0}$. Let 
\[
N\overset{j}{\supset}\mathcal{V}\xrightarrow{\gamma}\xx
\]
be a $\tb$-equivariant tubular neighborhood for $\zc$ (cf. Definition
\ref{def:tubular ngbhd}).

If $\eta$ is a differential form on $\xx$, the \emph{$k$-homogeneous
component} $\eta^{\left(k\right)}$ \emph{of $\eta$ around $\mathcal{Z}$}
is a form on $N=N_{\zc}^{\xx}$ defined recursively as follows. Set
$\eta^{\left(0\right)}:=\lim_{s\to\infty}\mu_{s^{-1}}^{*}\gamma^{*}\eta=\left(i\circ\pi\right)^{*}\eta$
and 
\[
\eta^{\left(k\right)}:=\lim_{s\to\infty}s^{k}\mu_{s^{-1}}^{*}\left(\gamma^{*}\eta-\sum_{0\leq j<k}\eta^{\left(j\right)}\right).
\]
(as $s\to\infty$, $\mu_{s^{-1}}\left(\mathcal{V}\right)$ exhausts
$N$, and the limit exists uniformly on compact subsets with all derivatives).

We have $\left(d\eta\right)^{\left(k\right)}=d\left(\eta^{\left(k\right)}\right)$,
$\left(\lambda_{i}\iota_{\xi_{i}}\eta\right)^{\left(k\right)}=\lambda_{i}\iota_{\xi_{i}}\left(\eta^{\left(k\right)}\right)$
for $1\leq i\leq m$, and $\left(D\eta\right){}^{\left(k\right)}=D\left(\eta^{\left(k\right)}\right)$.
In general, $\eta^{\left(k\right)}$ depends on the linear structure
on the tubular neighbourhood $N$, except for the leading term: if
$k$ satisfies
\begin{equation}
\eta^{\left(j\right)}=0\mbox{ for }0\leq j<k,\label{eq:leading term}
\end{equation}
then $\eta^{\left(k\right)}$ is independent of the choice of tubular
neighbourhood.  

\begin{defn}
\label{def:localizing form}Let $\zc\hookrightarrow\xx$ be a $\tb$-equivariant
closed embedding of compact $\tb$-orbifolds with corners. A form
${\eta\in\Omega\left(\xx;\mathcal{R}\right)^{\tb}=\Omega\left(\xx;S^{-1}\rr\left[\boldsymbol{\lambda}\right]\right)^{\tb}}$
will be called \emph{localizing to $\mathcal{Z}$ }if

(a) $\eta=\sum_{i=1}^{N}\frac{\eta_{i}}{\left\langle \zeta_{i},\vec{\lambda}\right\rangle }$
for invariant 1-forms $\eta_{i}\in\Omega^{1}\left(\xx;\rr\right)^{\tb}$
and nonzero tuples $\zeta_{i}=\left(\zeta_{i}^{1},...,\zeta_{i}^{m}\right)\in\qq^{m}\backslash\left\{ 0\right\} $,
where $\left\langle \zeta_{i},\vec{\lambda}\right\rangle :=\sum_{j=1}^{m}\zeta_{i}^{j}\lambda_{j}$.

(b) $f:=\sum_{j=1}^{m}\lambda_{j}\iota_{\xi_{j}}\eta$ is a real-valued,
non-negative function with $f^{-1}\left(0\right)=\mathcal{Z}$.

(c) $f^{\left(0\right)},f^{\left(1\right)}$ vanish and $f^{\left(2\right)}:=\lim_{s\to\infty}\mu_{s^{-1}}^{*}\left(s^{2}f\right)|_{N}$
is a positive definite quadratic form on each fiber of $N$.
\end{defn}

\begin{defn}
We say a $\tb$-orbifold with corners $\xx$ \emph{has regular fixed
points} if the map $\xx^{\tb}\to\xx$ is a closed embedding (cf. Definition
\ref{def:closed embedding}).
\end{defn}
All the orbifolds we'll encounter have regular fixed points, see
$\S$\ref{subsec:closed and open fps}. If $\xx$ has regular fixed
points a form $\eta$ localizing to $\xx^{\tb}$ is called \emph{a}
\emph{fixed-point localizing form}.

We let $i_{\xx}^{\partial}:\partial\xx\to\xx$ and $i_{\partial\xx}^{\partial}:\partial^{2}\xx\to\partial\xx$
denote the structure maps. Note that $\partial^{2}\xx$ is equipped
with a $\zz/2$ action.
\begin{prop}
\label{prop:localizing eta exists}Suppose $\xx$ has regular fixed
points. Given a fixed-point localizing form $\eta_{\partial}\in\Omega\left(\partial\xx;\mathcal{R}\right)^{\tb}$
such that $\left(i_{\partial\xx}^{\partial}\right)^{*}\eta_{\partial}$
is $\zz/2$-invariant, there exists a fixed-point localizing form
$\eta\in\Omega\left(\xx;\mathcal{R}\right)^{\tb}$ with $\left(i_{\xx}^{\partial}\right)^{*}\eta=\eta_{\partial}$.
\end{prop}
\begin{proof}
Fix an equivariant tubular neighborhood 
\begin{equation}
N\supset\mathcal{V}\xrightarrow{\gamma}\xx\label{eq:tube for eta}
\end{equation}
 for $\xx^{\tb}\to\xx$ (cf. Lemma \ref{lem:tubular exists}). We
construct an atlas $\pi:\underline{U}\to\xx$ for $\xx$, with $U=\coprod U_{p}$
where $p$ ranges over a set of representatives for the points $\pt\hookrightarrow\xx$
of $\xx$, as well as a provisional form 
\[
\eta''=\coprod\eta''_{p}\in\Omega\left(\coprod U_{p};\mathcal{R}\right).
\]
We divide the construction into cases.

\uline{Case 1:} $p$ is an interior point which is not a fixed
point. There exists a basis $e_{1}',...,e_{m}'$ for $Lie\left(\tb\right)\simeq\rr^{m}$
with the property that the corresponding vector fields $\xi_{1}',...,\xi_{m}'$
on $\xx$ satisfy the following condition at $T_{p}\xx$, for some
$1\leq r\leq m$:
\[
\xi_{1}'|_{p},...,\xi_{r}'|_{p}\mbox{ are linearly independent, and for }j\in\left\{ r+1,...,m\right\} ,\;\xi_{j}'|_{p}=0.
\]
Choose 1-forms 
\[
\gamma_{1}',...,\gamma_{r}'\in\Omega^{1}\left(\xx;\rr\right)
\]
that satisfy $\gamma_{i}'\left(\xi_{j}'\right)|_{p}=\delta_{ij}$
for $1\leq i,j\leq r$. Let $\lambda_{1}',...,\lambda_{m}'$ be a
basis for $H_{\tb}^{2}$ dual to $e_{1}',...,e_{m}'$. We see that
$f:=\sum_{j=1}^{m}\lambda_{j}\iota_{\xi_{j}}\left(\sum_{j=1}^{r}\frac{\gamma_{j}'}{\lambda_{j}'}\right)$
is a real-valued function on $\xx$ with $f\left(p\right)=r$, so
it remains positive in an open neighborhood $U_{p}$ of $p$ and we
set $\eta_{p}''=\left(\sum_{j=1}^{r}\frac{\gamma'_{j}}{\lambda_{j}'}\right)|_{U_{p}}$.

\uline{Case 2:} $p$ is an interior fixed point of the $\tb$-action.
Choose a trivialization
\[
N|_{U}\simeq U\times\bigoplus_{i=1}^{b}\cc_{\boldsymbol{\zeta}_{i}}
\]
over an open $p\in U\subset\xx^{\tb}$. (\ref{eq:tube for eta}) becomes
\[
U\times\bigoplus_{i=1}^{b}\cc_{\boldsymbol{\zeta}_{i}}\supset U_{p}\subset\xx
\]
Here $\cc_{\boldsymbol{\zeta}_{i}}$ denotes the rank one complex
$\tb$-representation with fractional weight $\boldsymbol{\zeta}_{i}$.
Using polar coordinates $z_{i}=r_{i}\cdot e^{\sqrt{-1}\theta_{i}}$
for $\cc_{\boldsymbol{\zeta}_{i}}$ we can set 
\[
\eta_{p}''=\sum_{i=1}^{b}\frac{r_{i}^{2}\frac{d\theta_{i}}{2\pi}}{\left\langle \boldsymbol{\zeta}_{i},\boldsymbol{\lambda}\right\rangle }
\]
on $U_{p}$. Clearly $f=\sum_{j=1}^{m}\lambda_{j}\iota_{\xi_{j}}\eta_{p}''=\sum_{i=1}^{b}r_{i}^{2}$
vanishes only on $U\times0=\xx^{\tb}\times_{\xx}U_{p}$, and induces
a positive definite quadratic form on the normal bundle. 

Before we deal with points of positive depth, we need the following.
Write $\eta_{\partial}=\sum_{i=1}^{N}\frac{\eta_{\partial,i}}{\left\langle \boldsymbol{\zeta}_{i},\boldsymbol{\lambda}\right\rangle }$
as in Definition \ref{def:localizing form}, fix 1-forms $\tilde{\eta}_{i}\in\Omega^{1}\left(\xx;\rr\right)^{\tb}$
with $i^{*}\tilde{\eta}_{i}=\eta_{\partial,i}$ (cf. Lemma \ref{lem:extending forms inward})
and set $\tilde{\eta}=\sum_{i=1}^{N}\frac{\tilde{\eta}_{i}}{\left\langle \boldsymbol{\zeta}_{i},\boldsymbol{\lambda}\right\rangle }\in\Omega\left(\xx;\mathcal{R}\right)^{\tb}$
and $\tilde{f}:=\sum\lambda_{j}\iota_{\xi_{j}}\tilde{\eta}$.

\uline{Case 3:} $p$ has depth $c\geq1$, $p$ is not a fixed point.
In this case, there's an open neighborhood $p\in U_{p}\subset\xx$
with $\tilde{f}|_{U_{p}}>0$, and we take $\eta''_{p}=\tilde{\eta}|_{U_{p}}$. 

\uline{Case 4:} $p$ has depth $c\geq1$, $p$ is a fixed point.
$p$ has depth $c$ in $\xx^{\tb}$ also, and the normal bundle to
$\left(\partial\xx\right)^{\tb}\to\partial\xx$ is identified with
$\left(i_{\xx}^{\partial}\right)^{*}N$. $\left(i_{\xx}^{\partial}\right)^{*}\tilde{\eta}=\eta_{\partial}$
has positive definite quadratic form at $p$, the same is true for
an open neighborhood $p\in U\subset\xx^{\tb}$ and we take $U_{p}=N|_{U}\cap\vc$
and $\eta''_{p}=\tilde{\eta}|_{U_{p}}$.

We construct a groupoid $X_{1}\overset{s,t}{\rightrightarrows}X_{0}$
representing $\xx$ with $X_{0}=\coprod U_{p}$.  Since $\xx$ is
compact there exists a compactly supported function $\rho:X_{0}\to\left[0,1\right]$
which satisfies $t_{*}s^{*}\rho\equiv1$ (such a function is called
\emph{a partition of unity }for $X_{\bullet}$). We define $\eta'\in\Omega\left(\xx;\mathcal{R}\right)$
by $\eta'=t_{*}s^{*}\left(\rho\eta''\right)$, and $\eta\in\Omega\left(\xx;\mathcal{R}\right)^{\tb}$
by averaging out the $\tb$-action on $\eta'$ using a Haar measure
on $\tb$. This $\eta$ is seen to be a fixed-point localizing form
with $\left(i_{\xx}^{\partial}\right)^{*}\eta=\eta|_{\partial}$.
\end{proof}

We now describe a setup for steepest descent analysis. 
\begin{defn}
\label{def:localizable map}Let $\xx$ be a compact $\tb$-orbifold,
$\zc=\coprod\zc_{n}$ a compact local $\tb$-orbifold, and $\zc\to\xx$
a $\tb$-equivariant closed embedding with associated normal bundle
$N\to\zc$. A \emph{$\zc$-localizable morphism }consists of

\begin{itemize}
\item An equivariant tubular neighborhood $N\overset{j}{\longleftarrow}\mathcal{V}\overset{\gamma}{\longrightarrow}\xx$
for $\zc$,
\item a $\tb$-equivariant diagram of $\tb$-orbifolds with corners
\begin{equation}
\xymatrix{\widetilde{N}\ar[d]_{\bu_{N}} & \widetilde{\mathcal{V}}\ar[l]\ar[r]^{\tilde{\gamma}}\ar[d] & \widetilde{\xx}\ar[d]^{\bu}\\
N & \ar[l]^{j}\mathcal{V}\ar[r]_{\gamma} & \xx
}
\label{eq:localizable diagram}
\end{equation}
with both squares cartesian, $\widetilde{\xx}$ compact, and $\bu,\bu_{N}$
relatively oriented, and
\item A smooth map $\widetilde{N}\times\rr_{\geq0}\xrightarrow{\tilde{\mu}=\left\{ \tilde{\mu}_{a}\right\} _{a\geq0}}\widetilde{N}$
such that $\bu\circ\tilde{\mu}=\mu\circ\left(\bu\times\id\right)$
where $\mu:N\times\rr_{\geq0}\to N$ denotes scalar multiplication.
\end{itemize}
We'll abbreviate and write ``$\bu:\widetilde{\xx}\to\xx$ is a $\zc$-localizable
morphism'', keeping the additional data implicit. A simple diagram
chase shows that $\widetilde{\mu}_{0}^{-1}\left(\widetilde{\vc}\right)\xrightarrow{\widetilde{\mu}_{0}^{*}\tilde{j}}\widetilde{N}$
is a diffeomorphism and 
\[
\widetilde{\rho}_{0}:=\tilde{j}^{*}\widetilde{\mu}_{0}\left(\widetilde{\mu}_{0}^{*}\tilde{j}\right)^{-1}:\widetilde{N}\to\widetilde{\mathcal{V}}
\]
lies over $\rho_{0}$, the retraction $N\to\zc\to\mathcal{V}$.
\end{defn}

\begin{prop}
\label{prop:steepest descent}Let $\bu:\widetilde{\xx}\to\xx$ be
a $\zc$-localizable morphism, with the additional data denoted as
in Definition \ref{def:localizable map}. Let $\eta\in\Omega\left(\xx;\mathcal{R}\right)^{\tb}$
be $\zc$-localizing, and denote $\tilde{\eta}=\bu^{*}\eta$.  Then
for any $\tilde{\omega}\in\Omega\left(\widetilde{\xx};Or\left(T\widetilde{\xx}\right)^{\vee}\otimes_{\zz}\mathcal{R}\right)^{\tb}$,
we have 
\begin{equation}
\lim_{s\to\infty}\int_{\widetilde{\xx}}e^{s^{2}D\tilde{\eta}}\tilde{\omega}=\int_{\widetilde{N}}\bu_{N}^{*}\left(e^{D\eta^{\left(2\right)}}\right)\left(\tilde{\gamma}\tilde{\rho}_{0}\right)^{*}\tilde{\omega}.\label{eq:steepest descent}
\end{equation}
\end{prop}
Let us explain what we mean by (\ref{eq:steepest descent}). Writing
$D\tilde{\eta}=-\tilde{f}+d\tilde{\eta}$ for $\tilde{f}=\sum_{j=1}^{m}\lambda_{j}\iota_{\xi_{j}}\tilde{\eta}$
a real-valued function, we define 
\[
e^{s^{2}D\tilde{\eta}}:=e^{-s^{2}\tilde{f}}\sum_{i=0}^{\infty}\frac{\left(s^{2}d\tilde{\eta}\right)^{i}}{i!},
\]
where the sum has only finitely many nonzero terms since $d\tilde{\eta}$
is nilpotent, similarly for $e^{D\eta^{\left(2\right)}}$. It follows
that there are finitely many linearly independent rational functions
$q_{j}\in\mathcal{R}$, $1\leq j\leq M$ such that $\int_{\widetilde{N}}\beta_{N}^{*}\left(e^{D\eta^{\left(2\right)}}\right)\tilde{\mu}_{0}^{*}\left(\tilde{\omega}|_{\widetilde{\mathcal{V}}}\right)$
and $\int_{\widetilde{\xx}}e^{s^{2}D\tilde{\eta}}\tilde{\omega}$
(for all $s$) are elements of the finite dimensional vector space
$\rr\left\langle q_{1},...,q_{M}\right\rangle $. The limit in (\ref{eq:steepest descent})
is interpreted with respect to the usual topology on $\rr^{M}$. Note
that there are $\vec{a}_{1},...,\vec{a}_{M}\in\rr^{m}$ belonging
to the common domain of definition of $\left\{ q_{j}\right\} $ such
that $\left[q_{j_{1}}\left(\vec{a}_{j_{2}}\right)\right]_{j_{1},j_{2}}$
is an invertible matrix,  so it is enough to check that 
\begin{equation}
\lim_{s\to\infty}\int_{\widetilde{\xx}}\left(e^{s^{2}D\tilde{\eta}}\tilde{\omega}\right)|_{\vec{a}_{j}}^{top}=\int_{\widetilde{N}}\beta_{N}^{*}\left(e^{D\eta^{\left(2\right)}}\right)\tilde{\rho}_{0}^{*}\left(\tilde{\omega}|_{\widetilde{\mathcal{V}}}\right)|_{\vec{a}_{j}}^{top}\label{eq:limit explicitly}
\end{equation}
for all $1\leq j\leq M$, were $|_{\vec{a}_{j}}^{top}$ denotes the
top form component of the form obtained by substituting $\vec{\alpha}=\vec{a}_{j}$.
This is what we'll now do.
\begin{proof}
[Proof (of Proposition \ref{prop:steepest descent}).]Fix a smooth
$\tb$-invariant inner product on $N$, with ${\nu:N\to\rr_{\geq0}}$
the associated norm. Let ${B\left(r\right)=\gamma\left(\nu^{-1}\left([0,r)\right)\right)}$,
which is defined for sufficiently small $r$. Let $f=\sum_{j=1}^{m}\lambda_{j}\iota_{\xi_{j}}\eta$,
and write $f|_{\mathcal{V}}=f^{\left(2\right)}+R$ for some real valued
function $R:\mathcal{V}\to\rr$. By the non-degeneracy assumption,
$f^{\left(2\right)}\geq2c\nu^{2}$ for some $c>0$. Because $\zc$
is compact we can take $\delta>0$ sufficiently small such that $B\left(\delta\right)\subset\mathcal{V}$
and $\left|R|_{B\left(\delta\right)}\right|<c\nu^{2}$, so $f|_{B\left(\delta\right)}\geq c\nu^{2}$.
Since $f$ vanishes only on $\mathcal{Z}$ and $\xx\backslash B\left(\delta\right)$
is compact, there exists some $C>0$ such that $f\geq C$ on $\xx\backslash B\left(\delta\right)$.

We denote $\tilde{B}\left(s\right):=\bu_{N}^{-1}B\left(s\right)$.
Consider some $1\leq j\leq M$, and set $\underline{a}=\underline{a}_{j}$.
We will prove (\ref{eq:limit explicitly}) by showing that

\begin{equation}
\lim_{s\to\infty}\int_{\widetilde{X}}\left(I_{\widetilde{\xx}\backslash\widetilde{B}\left(s^{-1/2}\right)}e^{s^{2}D\tilde{\eta}}\tilde{\omega}\right)|_{\underline{a}}^{top}=0\label{eq:limit outside}
\end{equation}
and
\begin{equation}
\lim_{s\to\infty}\int_{\widetilde{N}}\left(I_{\widetilde{B}\left(s^{-1/2}\right)}e^{s^{2}D\tilde{\eta}}\tilde{\omega}\right)|_{\underline{a}}^{top}=\int_{\widetilde{N}}\left(e^{D\eta^{\left(2\right)}}\tilde{\rho}_{0}^{*}\left(\tilde{\omega}|_{\widetilde{\mathcal{V}}}\right)\right)|_{\underline{a}}^{top}.\label{eq:limit inside}
\end{equation}
Here $I_{W}$ denotes indicator function for $W$.

We begin with (\ref{eq:limit outside}). There exist smooth top forms
$\left\{ \tilde{\omega}_{i}\in\Omega^{\dim\widetilde{\xx}}\left(\widetilde{\xx};Or\left(T\xx\right)\right)^{\tb}\right\} _{i=0}^{n}$
(which depend on $\tilde{\omega}$, $d\eta$ and $\underline{a}$)
such that 
\begin{equation}
\left(e^{s^{2}D\tilde{\eta}}\tilde{\omega}\right)|_{\underline{a}}^{top}=\sum_{i=0}^{n}e^{-s^{2}\tilde{f}}s^{2i}\tilde{\omega}_{i}.\label{eq:exponent times poly expansion}
\end{equation}
For $s$ sufficiently large, we have $\widetilde{X}\backslash\widetilde{B}\left(s^{-1/2}\right)\subset b^{-1}\left(\xx\backslash B\left(\delta\right)\right)$,
which means 
\[
I_{\widetilde{\xx}\backslash\widetilde{B}\left(s^{-1/2}\right)}e^{-s^{2}\tilde{f}}s^{2i}\leq e^{-s^{2}C}s^{2i}.
\]
The right-hand side tends to zero, and the limit (\ref{eq:limit outside})
is obtained.

To establish (\ref{eq:limit inside}) we first change variables by
$\left(\tilde{\gamma}\circ\tilde{j}^{-1}\circ\tilde{\mu}_{s^{-1}}\right)$,
which we'll denote simply by $\tilde{\mu}_{s^{-1}}$, to obtain: 
\[
\int_{\widetilde{N}}\left(I_{\widetilde{B}\left(s^{-1/2}\right)}e^{s^{2}D\tilde{\eta}}\tilde{\omega}\right)|_{\underline{a}}^{top}=\int_{\widetilde{N}}\left(I_{\widetilde{B}\left(s^{+1/2}\right)}e^{s^{2}\tilde{\mu}_{s^{-1}}^{*}D\tilde{\eta}}\tilde{\mu}_{s^{-1}}^{*}\tilde{\omega}\right)|_{\underline{a}}^{top}
\]
We have $\left(I_{\widetilde{B}\left(s^{+1/2}\right)}e^{s^{2}\tilde{\mu}_{s^{-1}}^{*}D\tilde{\eta}}\tilde{\mu}_{s^{-1}}^{*}\tilde{\omega}\right)|_{\underline{a}}^{top}\longrightarrow e^{D\eta^{\left(2\right)}}\tilde{\mu}_{0}^{*}\left(\tilde{\omega}|_{\widetilde{\mathcal{V}}}\right)|_{\underline{a}}^{top}$
pointwise in $\widetilde{N}$. We use Eq (\ref{eq:exponent times poly expansion})
to write 
\[
\left(I_{\widetilde{B}\left(s^{+1/2}\right)}e^{s^{2}\tilde{\mu}_{s^{-1}}^{*}D\tilde{\eta}}\tilde{\mu}_{s^{-1}}^{*}\tilde{\omega}\right)|_{\underline{a}}^{top}=\sum_{i=0}^{n}I_{\widetilde{B}\left(s^{+1/2}\right)}e^{-s^{2}\tilde{\mu}_{s^{-1}}^{*}f}s^{2i}\tilde{\mu}_{s^{-1}}^{*}\tilde{\omega}_{i}|_{\underline{a}}^{top}.
\]
If we fix some reference $\tb$-invariant volume form $vol$ on $\widetilde{N}$
we can write 
\[
\tilde{\mu}_{s^{-1}}^{*}\tilde{\omega}_{i}|_{\underline{a}}^{top}=g_{i,s}\cdot vol,\quad e^{-s^{2}\tilde{\mu}_{s^{-1}}^{*}\tilde{f}}s^{2i}\tilde{\mu}_{s^{-1}}^{*}\tilde{\omega}_{i}|_{\underline{a}}^{top}=h_{i,s}\cdot vol
\]
for some $\rr$-valued functions $g_{i,s}$ and $h_{i,s}=e^{-s^{2}\mu_{s^{-1}}^{*}\tilde{f}}s^{2i}g_{i,s}$.
It is not hard to see that there's some constant $C_{1}$ such that
$\left|g_{i,s}\right|\leq C_{1}$ for all $0\leq i\leq n$ and $s\geq1$.
We have $\mu_{s}^{*}f\geq\mu_{s}^{*}\left(c\nu^{2}\right)=cs^{2}\nu^{2}$
so 

\[
\left|h_{i,s}\right|\leq e^{-cs^{2}\nu^{2}}s^{2i}C_{1}
\]
which shows that the forms are uniformly bounded by an integrable
multiple of $vol$, and we can apply Lebesgue's dominant convergence
theorem. This completes the proof (\ref{eq:limit inside}) and of
the proposition.
\end{proof}

\subsection{\label{subsec:Equivariant-Segre-forms.}Equivariant Segre forms.}
\begin{lem}
\label{lem:primitive of e^Dq-1}Let $\xx$ be a local $\tb$-orbifold
and let $\eta\in\Omega\left(\xx;\mathcal{R}\right)^{\tb}$ be any
form such that $f:=\sum_{j=1}^{m}\lambda_{j}\iota_{\xi_{j}}\eta$
is a real valued function, $f\in\Omega\left(\xx;\rr\right)^{\tb}$.
There's a well-defined form $P\left(\eta\right)\in\Omega\left(\xx;\mathcal{R}\right)^{\tb}$
given by 
\begin{equation}
P\left(\eta\right)=\frac{\eta}{D\eta}\left(e^{D\eta}-1\right):=\eta\sum_{i=0}^{\dim\xx}h^{\left(i\right)}\left(-f\right)\frac{\left(d\eta\right)^{i}}{i!}\label{eq:primitive of e^Dq-1}
\end{equation}
where $h^{\left(i\right)}$ denotes the $i$'th derivative of the
smooth function $h\left(a\right):=\frac{e^{a}-1}{a}$.The form $P\left(\eta\right)$
satisfies 
\[
DP\left(\eta\right)=e^{D\eta}-1.
\]
Moreover, for any $f:\yy\to\xx$ we have $P\left(f^{*}\eta\right)=f^{*}P\left(\eta\right)$.
\end{lem}
\begin{proof}
Straightforward.
\end{proof}
If $E=\coprod_{n}E_{n}$ is any $\tb$-equivariant l-vector bundle
over a local $\tb$-orbifold $\zc=\coprod_{n\geq0}\zc_{n}$, we say
a form $q\in\Omega\left(E;\mathcal{R}\right)^{\tb}$ is \emph{quadratic
}if $q^{\left(2\right)}=q$. A quadratic form $q$ will be called
\emph{localizing }if it is localizing to the zero section of $E$.
There are bundles $E$ which admit no forms localizing to the zero
section (quadratic or otherwise).
\begin{defn}
Suppose $q$ is a quadratic localizing form for $E$. The \emph{equivariant
Segre form }
\[
\sigma\in\Omega\left(\zc;\Or\left(E\right)\otimes\rr\left[\vec{\lambda}\right]\right)^{\tb}
\]
 \emph{of $E$ associated with $q$ }is 
\[
\sigma=\pi_{*}\left(e^{Dq}\right),
\]
where $\pi:E\to\zc$ is the structure map and $\pi_{*}$ denotes integration
on the fiber.

A \emph{boundary condition for $E$} consists of a vector bundle $E'\to\zc'$
sitting in a cartesian square 
\begin{equation}
\xymatrix{\partial E\ar[d]_{\partial\pi}\ar[r]^{h} & E'\ar[d]^{\pi'}\\
\partial\zc\ar[r]_{g} & \zc'
}
\label{eq:boundary condition}
\end{equation}
where the total dimension of $E'$ is strictly less than the total
dimension of $\partial E$ (so $g$ maps $\partial\zc_{n}$ to $\zc'_{\leq n-2}$)
We say $q$ satisfies the boundary condition if $q|_{\partial E}$
is pulled back along $h$.
\end{defn}
Note the degree of $\sigma$ and the rank of $E$ may only be locally
constant, with $\deg\sigma=-\operatorname{rk}E.$
\begin{lem}
Let $\zc=\coprod\zc_{n}$ be a compact local $\tb$-orbifold, and
let $E\to\zc$ be a vector bundle over $\zc$. Let $q,q'\in\Omega\left(E;\mathcal{R}\right)^{\tb}$
be quadratic localizing forms for the vector bundle $E$, with $\sigma=\pi_{*}\left(e^{Dq}\right),\sigma'=\pi_{*}\left(e^{Dq'}\right)$
the associated equivariant Segre forms. Write $Dq=-f+dq$ and $Dq'=-f'+dq'$
for real-valued functions $f,f'$ which are positive definite quadratic
forms on each fiber. The form 
\[
P_{2}\left(q,q'\right)=\frac{q-q'}{Dq-Dq'}\left(e^{Dq}-e^{Dq'}\right):=\left(q-q'\right)\sum_{i,j\geq0}h_{2}^{\left(i,j\right)}\left(-f,-f'\right)\cdot\frac{\left(dq\right)^{i}\,\left(dq'\right)^{j}}{i!\,j!},
\]
where $h_{2}^{\left(i,j\right)}\left(a,b\right):=\partial_{x}^{i}\partial_{y}^{j}|_{\left(x,y\right)=\left(a,b\right)}\frac{e^{x}-e^{y}}{x-y}$,
satisfies
\[
DP_{2}\left(q,q'\right)=e^{Dq}-e^{Dq'}.
\]
Moreover, $P_{2}\left(q,q'\right)$ is rapidly decaying at infinity,
so 
\[
\epsilon=\pi_{*}\left(P_{2}\left(q,q'\right)\right)
\]
is well-defined, and we have 
\[
\sigma-\sigma'=D\epsilon.
\]

If $q,q'$ satisfy the boundary condition (\ref{eq:boundary condition}),
then $\epsilon|_{\partial\zc}$ is pulled back along $g:\partial\zc\to\zc'$.
\end{lem}
\begin{proof}
Establishing the properties of $P_{2}\left(q,q'\right)$ is straightforward;
we then compute 
\[
\sigma-\sigma'=\pi_{*}\left(e^{Dq}-e^{Dq'}\right)=\pi_{*}\left(DP_{2}\left(q,q'\right)\right)=D\epsilon,
\]
where in the last equality we use the fact $P_{2}\left(q,q'\right)$
is rapidly decaying at infinity. Clearly, if $q|_{\partial E},q'|_{\partial E}$
are pulled back along $\tilde{f}$ so is $P_{2}\left(q,q'\right)$
and the last claim follows from the identity $\left(\partial\pi\right){}_{*}h^{*}=g^{*}\pi'_{*}$.
\end{proof}
\begin{rem}
we also have $e^{Dq}-e^{Dq'}=D\left(P\left(q\right)-P\left(q'\right)\right)$,
but $P\left(q\right)-P\left(q'\right)$ may not be rapidly decaying.
\end{rem}
We now show the equivariant Segre form is, up to an exact form, the
inverse to the equivariant Euler form. Let $S\left(E\right)\xrightarrow{\pi_{S}}\zc$
be the sphere bundle associated with $E\xrightarrow{\pi}\zc$.

\begin{lem}
\label{lem:euler*segre =00003D 1}Suppose $E$ admits a localizing
quadratic form $q$, with $\sigma$ the corresponding Segre form.
Let $\phi\in\Omega\left(S\left(E\right),\pi_{S}^{*}\Or\left(E\right)^{\vee}\otimes_{\zz}\rr\left[\vec{\lambda}\right]\right)^{\tb}$
be an equivariant angular form for $E$ with $e$ the corresponding
Euler form, $D\phi=-\pi_{S}^{*}e$ (note the degrees of $e$ and $\phi$
vary with the local rank of $E$). There exists a form $\epsilon\in\Omega\left(\zc;\mathcal{R}\right)^{\tb}$
such that 
\[
\sigma\cdot e=1+D\epsilon
\]

If $q$ satisfies the boundary condition (\ref{eq:boundary condition})
and $\phi|_{\partial S\left(E\right)}$ is pulled back along $S\left(\partial E\right)\xrightarrow{S\left(h\right)}S\left(E'\right)$,
then we can choose $\epsilon$ pulled back along $g$.
\end{lem}
\begin{proof}
We blow up the zero section of $E$, to obtain the following setup:
\[
\xymatrix{S\left(E\right)\times\left\{ 0\right\} \ar[dr]_{\pi_{S}}\ar[r]^{i} & S\left(E\right)\times[0,\infty)\ar[d]^{\widetilde{\pi}}\aru{r}{\beta} & E\ar[dl]^{\pi}\\
 & \zc
}
\]
Let $p_{S}$ denote the projections of $S\left(E\right)\times[0,\infty)$
on $S\left(E\right)$. Since $\beta$ has degree one, we have 
\begin{multline*}
e\cdot\pi_{*}\left(\exp\left(Dq\right)\right)=\pi_{*}\left(\pi^{*}e\cdot\exp\left(Dq\right)\right)=\\
=\tilde{\pi}_{*}\left(p_{S}^{*}\pi_{S}^{*}e\cdot\beta^{*}\exp\left(Dq\right)\right)=\tilde{\pi}_{*}\left(D\left(p_{S}^{*}\phi\cdot\beta^{*}\exp\left(Dq\right)\right)\right)
\end{multline*}
Since $\exp\left(Dq\right)$ is rapidly decaying at infinity we have
\begin{multline*}
\tilde{\pi}_{*}\left(D\left(p_{S}^{*}\phi\cdot\beta^{*}\exp\left(Dq\right)\right)\right)=\pi_{S*}\left(i^{*}\left(p_{S}^{*}\phi\cdot\beta^{*}\exp\left(Dq\right)\right)\right)+D\tilde{\pi}_{*}\left(p_{S}^{*}\phi\cdot\beta^{*}\exp\left(Dq\right)\right)=\\
=\pi_{S*}\left(\phi\cdot\exp\left(0\right)\right)+D\epsilon=1+D\epsilon
\end{multline*}
as claimed.

If $\phi|_{|_{\partial S\left(E\right)}}=S\left(h\right)^{*}\phi'$
and $q|_{\partial E}=h^{*}q'$ then $p_{S'}^{*}\phi\cdot\beta'^{*}\exp\left(Dq\right)|_{\partial E}$
is pulled back along $\tilde{h}:\partial S\left(E\right)\times[0,\infty)\to S\left(E'\right)\times[0,\infty)$
and the last claim follows from $\tilde{\pi}_{*}\tilde{h}^{*}=g^{*}\tilde{\pi}_{*}$.
\end{proof}

\subsection{\label{subsec:closed fp formula}The fixed-point formula for closed
orbifolds}

We give a proof the fixed-point formula for closed $\tb$-orbifolds
(see \cite{AB}). Though the result is well-known, this proof helps
motivate the setup above and brings into focus the difficulty in applying
the fixed-point formula to spaces with boundary.
\begin{cor}
\label{cor:fixed point formula}Let $\xx$ be a compact $\tb$-orbifold
without boundary. Let $F\hookrightarrow\xx$ denote the fixed-point
suborbifold. For any $\omega\in\Omega\left(\xx;\mbox{Or}\left(T\xx\right)^{\vee}\otimes_{\zz}\rr\left[\lambda_{1},....,\lambda_{m}\right]\right)^{\tb}$
with $D\omega=0$ we have
\begin{equation}
\int_{\xx}\omega=\int_{F}\frac{\omega|_{F}}{e\left(N\right)}\label{eq:closed localization}
\end{equation}
where $e\left(N\right)$ denotes the equivariant Euler form of the
normal bundle $N$ to $F\hookrightarrow\xx$. 
\end{cor}
\begin{proof}
We have
\begin{equation}
\lim_{s\to\infty}\int_{\xx}e^{s^{2}D\eta}\omega=\int_{F}\pi_{*}\left(e^{D\eta^{\left(2\right)}}\right)\omega|_{F}=\int_{F}\frac{\omega|_{F}}{e\left(N\right)}.\label{eq:cfp eq1}
\end{equation}
where the first equality uses Proposition \ref{prop:steepest descent},
with $\beta=\id_{\xx}$ the identity $F$-localizable morphism, and
the projection formula; the second equality uses Lemma \ref{lem:euler*segre =00003D 1}. 

On the other hand, by Lemma \ref{lem:primitive of e^Dq-1} we have
$\left(e^{s^{2}D\eta}-1\right)\omega=D\left(P\left(s^{2}\eta\right)\right)$,
so by Stokes' theorem 
\begin{equation}
\int_{\xx}e^{s^{2}D\eta}\omega=\int_{\xx}\omega\label{eq:cfp eq2}
\end{equation}
 for all $s$.  Combining (\ref{eq:cfp eq1}) and (\ref{eq:cfp eq2})
gives (\ref{eq:closed localization}).
\end{proof}
\begin{rem}
\label{rem:boundary contributes}If $\partial\xx\neq\emptyset$, the
two sides of (\ref{eq:cfp eq2}) differ by $\int_{\partial\xx}P\left(s^{2}\eta\right)\omega$.
Example \ref{exa:lines thru 2pts} shows that this can be nonzero.
\end{rem}

\section{\label{sec:singular loc}Singular Fixed-point Formula}

\subsection{\label{subsec:sing loc}A singular fixed-point formula by steepest
descent}

In this subsection we apply the results of Section \ref{sec:steepest descent}
to obtain a singular fixed point formula for integrals of extended
forms (\ref{eq:integration}).

Fix some basic moduli specification $\basic$. The space $\check{\mm}_{\basic}^{r}$
has regular fixed points for every $r\geq0$. We denote $\check{F}_{\basic}^{r}:=\left(\check{\mm}_{\basic}^{r}\right)^{\tb}$
and let $i_{\basic}^{r}:\check{F}_{\basic}^{r}\hookrightarrow\check{\mm}_{\basic}^{r}$
denote the associated closed embedding, with $\check{N}_{\basic}^{r}=N_{i_{\basic}^{r}}$
the associated normal l-bundle.
\begin{prop}
\label{prop:regular fps and extended localizing exists}There exists
an \emph{extended localizing form} $\eta=\left\{ \check{\eta}_{r}\right\} \in\Omega_{\basic}^{\mathcal{R}}:=\mathcal{R}\otimes_{\rr\left[\vec{\lambda}\right]}\Omega_{\basic}$,
that is, an extended form such that $\check{\eta}_{r}$ is fixed-point
localizing for every $r\geq0$.
\end{prop}
The proof of this proposition appears in $\S$\ref{subsec:localizing eta}.
We fix an extended localizing form $\eta\in\Omega_{\basic}^{\mathcal{R}}$.
By Proposition \ref{prop:fp inverse image}, $F_{\basic}^{r}:=\left(\For_{\basic}^{r}\right)^{-1}\left(\check{F}_{\basic}^{r}\right)$
is an l-orbifold. Since $\For_{\basic}^{r}$ is a b-submersion and
$i_{\basic}^{r}$ is a closed embedding, the induced map $N_{\basic}^{r}\to\check{N}_{\basic}^{r}$
is cartesian. In particular, $\eta_{r}:=\left(\For_{\basic}^{r}\right)^{*}\check{\eta}_{r}$
is $F_{\basic}^{r}$-localizing for all $r\geq0$.

Recall $\int_{\basic}\omega$ was defined using a blow up $\widetilde{\mm}_{\basic}^{r}\xrightarrow{\bu_{\basic}^{r}}\mm_{\basic}^{r}$
(cf. $\S$\ref{subsec:Extended forms and integ}). We will see in
$\S$\ref{subsec:Blowup localizable} that the map $\bu_{\basic}^{r}:\widetilde{\mm}_{\basic}^{r}\to\mm_{\basic}^{r}$
is $F_{\basic}^{r}$-localizable. More precisely:
\begin{prop}
\label{prop:blowup is localizable}There exists a $\tb$-equivariant
commutative diagram
\begin{equation}
\xymatrix{\widetilde{N}_{\basic}^{r}\ar[d]_{\bu_{N_{\basic}^{r}}} & \widetilde{\mathcal{V}}_{\basic}^{r}\ar[l]_{\tilde{j}_{\basic}^{r}}\ar[r]^{\tilde{\gamma}_{\basic}^{r}}\ar[d] & \widetilde{\mm}_{\basic}^{r}\ar[d]^{\bu_{\basic}^{r}}\\
N_{\basic}^{r} & \ar[l]^{j_{\basic}^{r}}\mathcal{V}_{\basic}^{r}\ar[r]_{\gamma_{\basic}^{r}} & \mm_{\basic}^{r}
}
\label{eq:bu_b^r is localizable}
\end{equation}
and a smooth family of $\tb$-equivariant maps $\left\{ \tilde{\mu}_{a}:\widetilde{N}_{\basic}^{r}\to\widetilde{N}_{\basic}^{r}\right\} _{a\in\rr_{\geq0}}$
satisfying the conditions of Definition \ref{def:localizable map}.
\end{prop}

\begin{thm}
\label{thm:singular localization formula}Let $\basic$ be a basic
moduli specification, and let $\omega=\left\{ \check{\omega}_{r}\right\} \in\Omega_{\basic}$
be an extended form with $D\omega=0$. Let $\eta=\left\{ \check{\eta}_{r}\right\} _{r\geq0}\in\Omega_{\basic}^{\mathcal{R}}$
be an extended localizing form. Then, for $\eta_{r}=\left(\For{}_{\basic}^{r}\right)^{*}\check{\eta}_{r}$
and $\omega_{r}=\left(\For{}_{\basic}^{r}\right)^{*}\check{\omega}_{r}$
we have 

\begin{equation}
\int_{\basic}\omega=\sum_{r\geq0}\frac{1}{r!}\int_{\widetilde{N}_{\basic}^{r}}\bu_{N_{\basic}^{r}}^{*}e^{D\eta_{r}^{\left(2\right)}}\cdot\left(\tilde{\gamma}_{\basic}^{r}\tilde{\mu}_{0}\right)^{*}\left(\left(\bu_{\basic}^{r}\right)^{*}\omega_{r}\cdot\left(\widetilde{\ed}_{\basic}^{r}\right)^{*}\Lambda^{\boxtimes r}\right).\label{eq:sing loc formula}
\end{equation}
\end{thm}
\begin{proof}
Consider
\[
I\left(s\right):=\int_{\basic}e^{s^{2}D\eta}\omega=\sum_{r\geq0}\frac{1}{r!}\int_{\widetilde{\mm}_{\basic}^{r}}{\bu_{\basic}^{r}}^{*}\left(e^{s^{2}D\eta_{r}}\omega_{r}\right)\cdot{\widetilde{\ed}{}_{\basic}^{r}}^{*}\Lambda^{\boxtimes r}.
\]
First we claim that 
\begin{equation}
I\left(s\right)=I\left(0\right)\label{eq:time independence}
\end{equation}
 for all $s\geq0$. To see this, consider the extended form $\epsilon_{s}\omega=\left\{ \check{\epsilon}_{s,r}\check{\omega}_{r}\right\} \in\Omega_{\basic}^{\mathcal{R}}$
given by 
\[
\check{\epsilon}_{s,r}\check{\omega}_{r}=\frac{\check{\eta}_{r}}{D\check{\eta}_{r}}\left(e^{s^{2}D\check{\eta}_{r}}-1\right)\check{\omega}_{r},
\]
where the right hand side is interperted as in (\ref{eq:primitive of e^Dq-1}).
The collection $\left\{ \check{\epsilon}_{s,r}\check{\omega}_{r}\right\} $
is coherent and $Sym\left(r\right)$-invariant since $\left\{ \check{\eta}_{r}\right\} $
and $\left\{ \check{\omega}_{r}\right\} $ are. We have $D\left(\epsilon_{s}\omega\right)=\left(e^{s^{2}D\eta}-1\right)\omega$,
so (\ref{eq:time independence}) follows from Stokes' theorem, $\int_{\basic}D\left(\epsilon_{s}\omega\right)=0$.

Now take the limit as $s\to\infty$, and apply Proposition \ref{prop:steepest descent}.
\end{proof}

\subsection{\label{subsec:localizing eta}The group $E_{\kf}^{r}$ and localizing
forms}

For the constructions in this section and the next, we need to consider
a group $E_{\kf}^{r}$ which acts on $\check{\mm}_{\basic}^{r}$,
extending the action of $\Sym\left(r\right)$. Because the self-diffeomorphisms
(``autoequivalences'') of a stack do not form a group, we first
define these groups as groups of symmetries of a discrete object,
and only then let them act on $\check{\mm}_{\basic}^{r}$.

Let $T_{\kf}^{r}$ denote the set of trees on vertices $\left\{ v_{1},...,v_{r+1}\right\} $
with oriented edges $\left\{ e_{1},...,e_{r}\right\} $ as well as
inward-pointing half-edges labeled $\left\{ h_{x}\right\} _{x\in\kf}$.
A tree $\tc\in T_{\kf}^{r}$ can be represented by a pair of partitions
$\left(\left(\kf_{\tc}\left(v_{i}\right)\right)_{i=1}^{r+1},\left(\sigma_{\tc}\left(v_{i}\right)\right)_{i=1}^{r+1}\right)$,
so that $\sstar'_{\left[r\right]}=\coprod_{i=1}^{r+1}\sigma_{\tc}\left(v_{i}\right)$
and $\sstar''_{\left[r\right]}\coprod\left\{ 1,...,k\right\} =\coprod_{i=1}^{r+1}\kf_{\tc}\left(v_{i}\right)$.

Consider the subgroup 
\[
\hat{\rc}_{\kf}^{r}<\Sym\left(T_{\kf}^{r}\times\left(\sstar''_{\left[r\right]}\coprod\kf\right)\right)
\]
consisting of $\rho\in\Sym\left(T_{\kf}^{r}\times\left(\sstar''_{\left[r\right]}\coprod\kf\right)\right)$
such that\\
(a) there exists an element ${\alpha^{\rho}\in\Sym\left(T_{\kf}^{r}\right)}$
and a collection ${\left\{ F_{\tc}^{\rho}\in\Sym\left(\sstar''_{\left[r\right]}\coprod\kf\right)\right\} _{\tc\in T_{\kf}^{r}}}$
such that 
\[
{\rho\left(\tc,x\right)=\left(\alpha^{\rho}\left(\tc\right),F_{\tc}^{\rho}\left(x\right)\right),\text{ and }}
\]
(b) we have
\[
F_{\tc}^{\rho}\left(\kf_{\tc}\left(v_{i}\right)\right)=\kf_{\alpha^{\rho}\left(\tc\right)}\left(v_{i}\right)
\]
for all $\tc\in T_{\kf}^{r}$ $1\leq i\leq r+1$. 

Note that $\Sym\left(v_{1},...,v_{r+1}\right)$ acts on $T_{\kf}^{r}$
by relabeling the vertices, and thus forms another subgroup
\[
\Sym\left(v_{1},...,v_{r+1}\right)\times\id<\Sym\left(T_{\kf}^{r}\right)\times\Sym\left(\sstar''_{\left[r\right]}\coprod\kf\right)<\Sym\left(T_{\kf}^{r}\times\left(\sstar''_{\left[r\right]}\coprod\kf\right)\right).
\]

The conjugation action of $\Sym\left(v_{1},...,v_{r+1}\right)$ preserves
$\hat{\rc}_{\kf}^{r}$, and we take $\rc_{\kf}^{r}<\hat{\rc}_{\kf}^{r}$
to be the fixed-points of this action. 

\begin{lem}
\label{lem:R_k^r action}Let $\basic=\left(\kf,\lf,\beta\right)$
be any basic moduli specification. Then the finite group $\rc_{\kf}^{r}$
acts on $\check{\mm}_{\basic}^{r}$.
\end{lem}
\begin{proof}
Let $\rho\in\rc_{\kf}^{r}$. The action of $\rho$ takes two steps. 

\uline{step 1}: there's an obvious forgetful map that sends $\tc\in\ts_{\basic}^{r}$
to $\underline{\tc}\in T_{\kf}^{r}$, and we map

\begin{equation}
\check{\mm}_{\tc}=\prod_{i=1}^{r+1}\check{\mm}_{\sfr_{\tc}\left(v_{i}\right)}\to\prod_{i=1}^{r+1}\check{\mm}_{\left(\left(\kf_{\alpha^{\rho}\underline{\tc}'}\left(v_{i}\right),\lf_{\tc}\left(v_{i}\right),\beta_{\tc}\left(v_{i}\right)\right),\sigma_{\alpha^{\rho}\underline{\tc}}\left(v_{i}\right)\right)}\label{eq:action step 1}
\end{equation}
by using $F_{\underline{\tc}}|_{\kf_{\underline{\tc}}\left(v_{i}\right)}$
to relabel the boundary markings. 

\uline{step 2:} There's a unique $w=w_{\rho,\underline{\tc}}\in\Sym\left(v_{1},...,v_{r+1}\right)$
which orders the vertices so that condition (c) in Definition \ref{def:labeled tree}
is satisfied, so that
\begin{equation}
\prod_{i=1}^{r+1}\check{\mm}_{\left(\left(\kf_{\alpha^{\rho}\underline{\tc}}\left(w\left(v_{i}\right)\right),\lf_{\tc}\left(w\left(v_{i}\right)\right),\beta_{\tc}\left(w\left(v_{i}\right)\right)\right),\sigma_{\alpha^{\rho}\underline{\tc}}\left(\xi\left(v_{i}\right)\right)\right)}=\check{\mm}_{\tc'}\label{eq:action step 2}
\end{equation}
for some $\tc'\in\ts_{\basic}^{r}$. 

The composition of (\ref{eq:action step 1}) and (\ref{eq:action step 2})
defines an automorphism of $\check{\mm}_{\basic}^{r}$. Associativity
follows from the assumption that $\Sym\left(v_{1},...,v_{r+1}\right)$
commutes with $\rc_{\kf}^{r}$.
\end{proof}
Now we'll now define a special subgroup of $\rc_{\kf}^{r}$. First,
consider the following elements of $\rc_{\kf}^{r}$.
\begin{itemize}
\item Any $\tau\in\Sym\left(e_{1},...,e_{r}\right)$ defines an element
of $\rc_{\kf}^{r}$. We set $\alpha^{\tau}\in\Sym\left(T_{\kf}^{r}\right)$
to be the map which relabels the edges according to $\tau$ and $F_{\tc}^{\tau}\in\Sym\left(\sstar''_{\left[r\right]}\coprod\kf\right)$,
which is independent of $\tc$ in this case, permutes $\sstar''_{\left[r\right]}$
and fixes $\kf$. 
\item $\mv_{x,e_{b}}\in\rc_{\kf}^{r}$, for $1\leq b\leq r$, $x\in\sstar'_{\left\{ 1,...,\hat{b},...,r\right\} }$,
which acts by ``moving the tail of $e_{a}$ along $e_{b}$''. More
precisely, consider some $\tc\in T_{\kf}^{r}$, and suppose $e_{b}$
is incident to $\left\{ v_{i},v_{j}\right\} $. If $x\in\sigma_{\tc}\left(v_{i}\right)$,
we define $\alpha^{\mv_{x,e_{b}}}\tc$ by removing $x$ from $\sigma_{\tc}\left(v_{i}\right)$
and adding it to $\sigma_{\tc}\left(v_{j}\right)$. If $x\not\in\sigma_{\tc}\left(v_{i}\right)\cup\sigma_{\tc}\left(v_{j}\right)$,
we set $\alpha^{\mv_{x,e_{b}}}\tc=\tc$. In either case we set $F_{\tc}^{\mv_{x,e_{b}}}=\id_{\sstar''_{\left[r\right]}\coprod\kf}$. 
\item $\mv_{x,e_{b}}\in\rc_{\kf}^{r}$ for $1\leq b\leq r$, $x\in\kf\coprod\sstar''_{\left\{ 1,...,\hat{b},...,r\right\} }$,
which acts by ``moving $x$ along $e_{b}$ and flipping $e_{b}$''.
More precisely, consider some $\tc\in T_{\kf}^{r}$, and suppose $e_{b}$
is incident to $\left(v_{i},v_{j}\right)$ \emph{where $v_{i}$ is
the tail} of $e_{b}$. If $x\in\kf_{\tc}\left(v_{i}\right)$, we define
$\tc'=\alpha^{\mv_{x,e_{b}}}\tc$ so that 
\[
\sigma_{\tc'}\left(v_{i}\right)=\sigma_{\tc}\left(v_{i}\right)\backslash\left\{ \sstar'_{b}\right\} 
\]
\[
\kf_{\tc'}\left(v_{i}\right)=\kf_{\tc}\left(v_{i}\right)\backslash\left\{ x\right\} \cup\left\{ \sstar''_{b}\right\} 
\]
\[
\sigma_{\tc'}\left(v_{j}\right)=\sigma_{\tc}\left(v_{j}\right)\cup\left\{ \sstar'_{b}\right\} 
\]
\[
\kf_{\tc'}\left(v_{j}\right)=\kf_{\tc}\left(v_{j}\right)\backslash\left\{ \sstar''_{b}\right\} \cup\left\{ x\right\} 
\]
for $k\not\in\left\{ i,j\right\} $ we set $\sigma_{\tc'}\left(v_{k}\right)=\sigma_{\tc}\left(v_{k}\right)$
and $\kf_{\tc'}\left(v_{k}\right)=\kf_{\tc}\left(v_{k}\right)$, and
$F_{\tc}^{\mv_{x,e_{b}}}$ the permutation the swaps $\left(\sstar''_{b},x\right)$.
If $x\not\in\kf_{\tc}\left(v_{i}\right)$, we set $\tc'=\tc$ and
$F_{\tc}^{\mv_{x,e_{b}}}=\id$.
\end{itemize}
We define $E_{\kf}^{r}<\rc_{\kf}^{r}$ to be the subgroup generated
by all of these elements.

In the next section we'll also need the subgroup ${E_{\kf}^{r,r'}<\rc_{\kf}^{r+r'}}$
obtained by ``restricting the action to the even-even edges $\left\{ r+1,...,r+r'\right\} $''.
More precisely, we have the following.
\begin{defn}
\label{def:E_k^r,r'}For $r,r'\geq0$, we let $E_{\kf}^{r,r'}$ be
the group generated by $\Sym\left(e_{r+1},...,e_{r+r'}\right)<\Sym\left(e_{1},...,e_{r+r'}\right)$
and the elements $\mv_{x,e_{b}}$ for $r+1\leq b\leq r+r'$ and ${x\in\sstar'_{\left\{ 1,...,\hat{b},...,r\right\} }\coprod\kf\coprod\sstar''_{\left\{ 1,....,\hat{b},...,r\right\} }}$.
\end{defn}
It is not hard to check that the action of this subgroup preserves
$\check{\mm}_{\basic}^{r,r'}$.

\begin{proof}
[Proof of Proposition \ref{prop:regular fps and extended localizing exists}]We
construct a sequence of forms $\left\{ \check{\eta}_{r}\in\Omega\left(\check{\mm}_{\basic}^{r};\rr\left[\vec{\alpha}\right]\right)^{\tb}\right\} _{r\geq0}$
satisfying the following properties, for every $r\geq0$.\\
($a_{r}$) for every $\tc_{+}\in\ts_{\basic}^{r+1}$ with $\cnt\tc_{+}=\tc$,
$\left(i_{\check{\mm}_{\tc}}^{\partial^{\tc_{+}}}\right)^{*}\check{\eta}_{r}=\left(\check{g}_{\tc_{+}}\right)^{*}\check{\eta}_{r+1}$,\\
($b_{r}$) $\check{\eta}_{r}$ is $\check{F}_{\basic}^{r}$-localizing,
and\\
($c_{r}$) $\check{\eta}_{r}$ is $E_{\kf}^{r}$-invariant.

The construction proceeds by backward induction. Suppose for $R\geq0$,
we have constructed equivariant forms $\left\{ \check{\eta}_{r}\right\} _{r\geq R+1}$
satisfying $\left(a_{r}\right)-\left(c_{r}\right)$ for $r\geq R+1$.
Note that this assumption holds vacuously for $R$ so large that $\check{\mm}_{\basic}^{R+1}=\emptyset$.
We will now construct $\check{\eta}_{R}$ satisfying $\left(a_{R}\right)-\left(c_{R}\right)$.

For $\tc_{+}\in\ts_{\basic}^{R+1}$ let $\check{\eta}_{\partial^{\tc_{+}}}=\left(\check{g}_{\tc_{+}}\right)^{*}\check{\eta}_{r+1}.$
Consider the covering map 
\[
\coprod_{\tc_{+}\in\ts_{\basic}^{R+1}}\partial^{\tc_{+}}\check{\mm}_{\cnt\tc_{+}}\to\partial\check{\mm}_{\basic}^{R}.
\]
Because $\check{\eta}_{R+1}$ is invariant under $\mv_{x,e_{R+1}}$
for $x\in\sstar'_{\left[R\right]}$, $\coprod_{\tc_{+}}\check{\eta}_{\partial^{\tc_{+}}}$
is the pullback of some form $\check{\eta}_{\partial}$ on $\partial\check{\mm}_{\basic}^{R}$. 

We show $\check{\eta}_{\partial}$ is $E_{\kf}^{R}$-invariant. Let
$\tc_{+}\in\ts_{\basic}^{R+1}$ and consider a generator ${\mv_{x,e_{b}}\in E_{\kf}^{R}}$.
We have
\[
e.\,\check{g}_{\tc_{+}}=\check{g}_{\tc_{+}}\,\partial^{\tc_{+}}\left(\mv_{x,e_{b}}.\right)
\]
for $e\in\left\{ \mv_{x,e_{b}},\mv_{x,e_{b}}\circ\mv_{\sstar'_{b},e_{R+1}},\mv_{x,e_{b}}\circ\mv_{x,e_{R+1}}\right\} \subset E_{\kf}^{R+1}$.
Since $\check{g}_{\tc_{+}}$ is also $\Sym\left(e_{1},...,e_{R}\right)\to\Sym\left(e_{1},...,e_{R+1}\right)$-equivariant,
$E_{\kf}^{R}$-invariance of $\check{\eta}_{\partial}$ follows from
$E_{\kf}^{R+1}$-invariance of of $\check{\eta}_{R+1}$. 

Next we want to argue that $\check{\eta}_{\partial}$ can be extended
inward. For this it suffices to check that $\left(i_{\partial\check{\mm}_{\basic}^{R}}^{\partial}\right)^{*}\check{\eta}_{\partial}$
is $\zz/2$-invariant with respect to the action that switches the
local boundary components at a corner. For $\tc_{1}\in\ts_{\basic}^{R+1}$
and $\tc_{0}=\cnt\tc_{1}$, the decomposition into horizontal and
vertical boundaries with respect to $\check{g}_{\tc_{1}}$,
\begin{multline*}
\partial\partial^{\tc_{1}}\check{\mm}_{\tc_{0}}=\partial_{-}^{\check{g}_{\tc_{1}}}\partial^{\tc_{1}}\check{\mm}_{\tc_{0}}\coprod\partial_{+}^{\check{g}_{\tc_{1}}}\partial^{\tc_{1}}\check{\mm}_{\tc_{0}}=\\
=\left(\acute{g}_{\tc_{1}}^{\emptyset}\right)_{-}^{-1}\left(\partial_{-}^{\grave{\For}_{\basic}^{R+1,\left\{ R+1\right\} }}\check{\mm}_{\tc_{1}}\coprod\partial_{+}^{\grave{\For}_{\basic}^{R+1,\left\{ R+1\right\} }}\check{\mm}_{\tc_{1}}\right),
\end{multline*}
induces a decomposition $\partial^{2}\check{\mm}_{\basic}^{R}=\partial_{-}^{2}\check{\mm}_{\basic}^{R}\coprod\partial_{+}^{2}\check{\mm}_{\basic}^{R}$.
We deal with these two components separately. 

First, we have 
\[
\partial_{-}^{2}\check{\mm}_{\basic}^{R}\xleftarrow{c_{-}}\coprod_{\tc_{2}\in\ts_{\basic}^{R}}\partial^{\tc_{2}}\partial^{\tc_{1}}\check{\mm}_{\tc_{0}}\xrightarrow{h_{-}}\coprod_{\tc_{2}\in\ts_{\basic}^{R+2}}\check{\mm}_{\tc_{2}},
\]
where $\tc_{i-1}=\cnt\tc_{i}$,  $\partial^{\tc_{2}}\partial^{\tc_{1}}\check{\mm}_{\tc_{0}}:=\left(\check{g}_{\tc_{1}}\right)_{-}^{-1}\left(\partial^{\tc_{2}}\check{\mm}_{\tc_{1}}\right)$,
$c_{-}$ is the covering map, and $h_{-}=\check{g}_{\tc_{2}}\circ\left(\check{g}_{\tc_{1}}\right)_{-}$.
The diagram is $\zz/2$-equivariant where $\zz/2$ acts on $\tc_{2}$
via $\Sym\left(e_{R+1},e_{R+2}\right)<E_{\kf}^{R+2}$, so $\zz/2$-invariance
of $\check{\eta}_{\partial}|_{\partial_{-}^{2}}$ follows from $\Sym\left(R+2\right)$
invariance of $\check{\eta}_{R+2}$.

Consider next 
\[
\partial_{+}^{2}\check{\mm}_{\basic}^{R}\xleftarrow{c_{+}}\coprod_{\tc_{1}\in\ts_{\basic}^{R+1}}\partial_{+}\partial^{\tc_{1}}\check{\mm}_{\tc_{0}}\xrightarrow{h_{+}}\coprod_{\tc_{1}\in\ts_{\basic}^{R+1}}\check{\mm}_{\tc_{1}},
\]
where $c_{+}$ is the covering map and $h_{+}:=\left(\check{g}_{\tc_{1}}\right)_{+}$.
Since $\im h_{+}$ is exhausted by clopen components where $a\in\kf\coprod\sstar''_{\left[R\right]}$
bubbles off on a ghost disc together with $\sstar'_{R+1}$, we find
$c_{+},h_{+}$ are $\zz/2$-equivariant, where we act on the domain
and codomain of $h_{+}$ by $\zz/2\simeq\left\langle \mv_{a,e_{R+1}}\right\rangle <E_{\kf}^{R+1}$.
In other words, $\zz/2$-invariance of $\check{\eta}_{\partial}|_{\partial_{+}^{2}}$
follows from $\left\langle \mv_{a,e_{R+1}}\right\rangle $ invariance
of $\check{\eta}_{R+1}$. 

By Proposition \ref{prop:localizing eta exists}, there exists a form
$\check{\eta}_{R}$ which satisfies conditions $\left(a_{R}\right)$
and $\left(b_{R}\right)$. By averaging, we ensure that $\check{\eta}_{R}$
also satisfies $\left(c_{R}\right)$. This completes the proof of
the inductive step. Clearly, $\check{\eta}=\left\{ \check{\eta}_{r}\right\} \in\Omega_{\basic}^{\rc}$
is an extended localizing form.
\end{proof}

\subsection{\label{subsec:Blowup localizable}Blow up map is localizable}

In this subsection we prove Proposition \ref{prop:blowup is localizable}
and derive a slightly more explicit form of (\ref{eq:sing loc formula}).
Focus on some $\tc\in\ts_{\basic}^{r,r'}$. Let 
\[
{\mbox{ed}{}_{\tc}^{\leftarrow}:=\prod_{i=1}^{r}\mbox{ed}{}_{i}^{\basic,r}:\mm_{\tc}\to\left(L\times L\right)^{r}}
\]
and
\[
{\mbox{ed}{}_{\tc}:=\prod_{i=r+1}^{r+r'}\mbox{ed}{}_{i}^{\basic,r}:\mm_{\tc}\to\left(L\times L\right)^{r'}},
\]
so 
\[
\mbox{ed}{}_{\basic}^{r}|_{\mm_{\tc}}=\mbox{ed}{}_{\tc}^{\leftarrow}\times\mbox{ed}{}_{\tc}.
\]
Let $i_{\tc}:F_{\tc}\to\mm_{\tc}$ be the pullback of $\check{F}_{\basic}^{r}\to\check{\mm}_{\basic}^{r}$.
It follows from the description of the fixed points in $\S$\ref{subsec:Fixed-point-profiles.}
that $\ed_{\tc}\circ i_{\tc}$ factors through $\left(p_{0},p_{0}\right)\hookrightarrow L\times L$
whereas $\ed^{\leftarrow}\circ i_{\tc}$ factors through the open
inclusion $L\times L\backslash\Delta\to L\times L$. In particular,
we have $F_{\tc}\subset\Delta_{\tc}:=\ed{}_{\tc}^{-1}\left(\Delta^{r'}\right)$.
Recall the blow up $\widetilde{L\times L}\overset{\bu_{\Delta}}{\longrightarrow}L\times L$
was defined using a tubular neighbourhood $N_{\Delta}\xrightarrow{\gamma_{\Delta}}L\times L$
(\ref{eq:diagonal tube}). Fix an open neighbourhood $p_{0}\in U\subset L$
and a trivialization, $V_{0}:=N_{\Delta}|_{U}\simeq U\times N_{0}=U\times N_{\Delta}|_{p_{0}}$.
We set $\gamma_{0}:=\gamma_{\Delta}|_{V_{0}}$ and $q_{0}:U\times N_{0}\to N_{0}$
the projection. 

The proof of the following Lemma appears in the next subsection.
\begin{lem}
\label{lem:N_T->N_0 compatible flows}For $\tc\in\ts_{\basic}^{r,r'}$,
there exists a $\tb$-equivariant tubular neighborhood $N_{\tc}\overset{j_{\tc}}{\longleftarrow}\mathcal{V}_{\tc}\overset{\gamma_{\tc}}{\longrightarrow}\mm_{\tc}$
making the diagram 
\begin{equation}
\xymatrix{N_{\tc}\ar[d]_{\psi_{\tc}} & \mathcal{V}_{\tc}\ar[d]\ar[l]_{j_{\tc}}\ar[r]^{\gamma_{\tc}} & \mm_{\tc}\ar[d]^{\mbox{ed}_{\tc}}\\
N_{0}^{r'} & V_{0}^{r'}\ar[l]^{q_{0}^{r'}}\ar[r]_{\gamma_{\Delta}^{r'}} & \left(L\times L\right)^{r'}
}
\label{eq:bottom faces}
\end{equation}
commute, where $\psi_{\tc}$ comes from the linearization of $\ed_{\tc}$.
\end{lem}
\begin{proof}
[Proof of Proposition \ref{prop:blowup is localizable}]Consider $\tc\in\ts_{\basic}^{r,r'}$
and let $\nu:N_{0}\overset{}{\longrightarrow}L\times L$ be the composition
$N_{0}\to N_{\Delta}\overset{\gamma_{\Delta}}{\longrightarrow}L\times L$.
We define $\bu_{N_{0}}$ and $\tilde{\nu}$ by the following cartesian
square 
\[
\xymatrix{\widetilde{N}_{0}=S\left(N_{0}\right)\times[0,\infty)\ar[r]^{\tilde{\nu}}\ar[d]_{\bu_{0}} & \widetilde{L\times L}\ar[d]^{\bu_{\Delta}}\\
N_{0}\ar[r]_{\nu} & L\times L
}
,
\]
We construct a diagram of cartesian squares
\begin{equation}
\xymatrix{\widetilde{N}_{0}^{r'}\ar[d]_{\bu_{0}^{r'}} & \widetilde{V}_{0}^{r'}=\left(U\times\widetilde{N}_{0}\right)^{r'}\ar[d]\ar[l]\ar[r] & \left(\widetilde{L\times L}\right)^{r'}\ar[d]^{\bu_{\Delta}^{r'}}\\
N_{0}^{r'} & V_{0}^{r'}\ar[l]^{q_{0}^{r'}}\ar[r]_{\gamma_{\Delta}^{r'}} & \left(L\times L\right)^{r'}
}
.\label{eq:back faces}
\end{equation}
Build a diagram modeled on a row of two cubes, by gluing (\ref{eq:bottom faces}\ref{eq:back faces})
along their bottom row. The remaining vertices of the diagram are
the pullbacks
\[
\widetilde{N}_{\tc}=N_{\tc}\times_{N_{0}^{r'}}\widetilde{N}_{0}^{r'},\;\widetilde{\mathcal{V}}_{\tc}=\mathcal{V}_{\tc}\times_{V_{0}^{r'}}\widetilde{V}_{0}^{r'}\text{ and }\widetilde{\mm}_{\tc}=\mm_{\tc}\times_{\left(L\times L\right)^{r+r'}}\left(\widetilde{L\times L}\right)^{r+r'}
\]
and we obtain maps $\widetilde{N}_{\tc}\leftarrow\widetilde{\mathcal{V}}_{\tc}\to\widetilde{\mm}_{\tc}$
as the pullback of $N_{\tc}\leftarrow V_{\tc}\to\mm_{\tc}$. One face
of the row of cubes is the diagram
\[
\xymatrix{\widetilde{N}_{\tc}\ar[d] & \widetilde{\mathcal{V}}_{\tc}\ar[l]\ar[d]\ar[r] & \widetilde{\mm}_{\tc}\ar[d]\\
N_{\tc} & \mathcal{V}_{\tc}\ar[r]\ar[l] & \mm_{\tc}
}
,
\]
which is cartesian by a simple diagram chase. The maps $\widetilde{\mu}_{a}:\widetilde{N}_{\tc}\to\widetilde{N}_{\tc}$
for $a\geq0$ are defined by the linear structure on $N_{\tc}$ and
the obvious rescaling of ${\widetilde{N}_{0}^{r'}=\left(S\left(N_{0}\right)\times[0,\infty)\right)^{r'}}$.
Smoothness of $\tilde{\mu}_{\bullet}:\widetilde{N}_{\tc}\times\rr_{\geq0}\to\widetilde{N}_{\tc}$
follows by elementary arguments from the fact the map 
\[
[0,\infty)\times\rr_{\geq0}\to[0,\infty):\left(x,y\right)\mapsto x\cdot y
\]
is smooth, see \cite[Example 2.3, (vii)]{joyce-generalized}.

For $\tc\in\ts_{\basic}^{r+r'}\backslash\ts_{\basic}^{r,r'}$ with
$F_{\tc}\neq\emptyset$, we use an element of $\Sym\left(r+r'\right)$
to transfer the construction from some $\tc'\in\ts_{\basic}^{r,r'}$.
Taking a disjoint union over all such $\tc$ with $F_{\tc}\neq\emptyset$
we obtain the data specified in the statement of the proposition;
the conditions of Definition \ref{def:localizable map} are readily
verified.
\end{proof}
Let us now take a closer look at the singular localization formula.
For $\vp\in\pc_{\basic}$ a fixed-point profile, let $\widetilde{N}_{\vect\phi}\xrightarrow{\bu_{\vect\phi}}N_{\vect\phi}$
be defined by the cartesian square 
\[
\xymatrix{\widetilde{N}_{\vect\phi}\ar[rr]^{\widetilde{\delta}_{\vect\phi}=\prod_{i=1}^{r'}\widetilde{\delta}_{\vp}^{i}}\ar[d]_{\bu_{\vect\phi}} &  & \widetilde{N}_{0}^{r'}\ar[d]^{\bu_{0}^{r'}}\\
N_{\vect\phi}\ar[rr]_{\delta_{\vect\phi}=\prod\delta_{\vp}^{i}} &  & N_{0}^{r'}
}
.
\]

\begin{cor}
\label{cor:sing formula by profiles}We have
\begin{equation}
\int_{\basic}\omega=\sum_{r,r'\geq0}\frac{1}{r!\cdot r'!}\sum_{\vp\in\pc_{\basic}^{r,r'}}\frac{\left(-1\right)^{r\cdot r'}\xi_{\vp}}{\boldsymbol{s}\left(\vp\right)!}\int_{\widetilde{N}_{\vp}}\bu_{N_{\vp}}^{*}\left(e^{D\eta_{\vp}}\pi_{\vp}^{*}\left(\omega_{\vp}\right)\right)\cdot\widetilde{\delta}_{\vp}^{*}\left(-\theta_{0}\right)^{\boxtimes r'}\label{eq:sing fp formula by fp profile}
\end{equation}
Hereafter, $\eta_{\vect\phi}=f_{\vect\phi}^{*}\check{\eta}_{\vect\phi}$
with $\check{\eta}_{\vect\phi}=\check{\eta}_{r+r'}^{\left(2\right)}|_{\check{N}_{\vect\phi}}$,
$\omega_{\vp}=\grave{\For}_{\vp}^{*}\check{\omega}_{r+r'}|_{\check{F}_{\vp}}$,
and $-\theta_{0}=-\theta_{\Delta}|_{p_{0}}$, which we consider as
a form on $\widetilde{N}_{0}=S\left(N_{0}\right)\times[0,\infty)$.
\end{cor}
\begin{proof}
We have $\mm_{\basic}^{R}=\coprod_{\tc\in\ts_{\basic}^{R}}\mm_{\tc}$.
The $\Sym\left(R\right)$ orbits of $\coprod_{r+r'=R}\ts_{\basic}^{r,r'}$
cover $\ts_{\basic}^{R}$, and we use Lemma \ref{lem:fps by profiles}
to decompose further into into fixed-point profiles:
\[
\frac{1}{R!}\int_{\widetilde{N}_{\basic}^{r}}\ii=\sum_{r+r'=R}\frac{1}{r!}\frac{1}{r'!}\sum_{\tc\in\ts_{\basic}^{r,r'}}\int_{\widetilde{N}_{\tc}}\ii=\sum_{r+r'=R}\frac{1}{r!}\frac{1}{r'!}\sum_{\vp\in\pc_{\basic}^{r,r'}}\frac{1}{\boldsymbol{s}\left(\vp\right)!}\int_{\widehat{N}_{\vp}^{\sim}}\ii.
\]
where $\widehat{N}_{\vp}\xrightarrow{\pi_{\widehat{F}_{\vp}}}\widehat{F}_{\vp}$
is the normal bundle associated with $\widehat{F}_{\vp}\to\mm_{\tc\left(\vp\right)}$
(cf. $\S$\ref{subsec:fp for generalized forms}), and $\widehat{N}_{\vp}^{\sim}\xrightarrow{\bu_{\widehat{N}_{\vp}}}\widehat{N}_{\vp}$
is the pullback of $\bu_{N_{\basic}^{r}}$.

We have

\[
\bu_{\basic}^{r}\tilde{\gamma}_{\basic}^{r}\tilde{\rho}_{0}=\gamma_{\basic}^{r}\rho_{0}\bu_{N_{\basic}^{r}}
\]
and
\[
\widetilde{\ed}_{\basic}^{r}\left(\tilde{\gamma}_{\basic}^{r}\tilde{\rho}_{0}\right)|_{\widehat{N}_{\vp}^{\sim}}=\ed_{\tc\left(\vp\right)}^{\leftarrow}|_{\widehat{F}_{\vp}}\pi_{\widehat{F}_{\vp}}\bu_{\widehat{N}_{\vp}^{\sim}}\times E^{r'}\widetilde{\delta}_{\vp}\acute{f}_{\vp}^{\sim}
\]
where $E$ is the retraction $\widetilde{N}_{0}\to S\left(N_{0}\right)\to\widetilde{N}_{0}$,
$\acute{f}_{\vp}^{\sim}$ is defined by the cartesian square $\acute{f}_{\vp}\bu_{\widehat{N}_{\vp}}=\bu_{\vp}\acute{f}_{\vp}^{\sim}$,
and $\left(\tilde{\gamma}_{\basic}^{r}\tilde{\rho}_{0}\right)|_{\widehat{N}_{\vp}^{\sim}}$
denotes the pullback of $\tilde{\gamma}_{\basic}^{r}\tilde{\rho}_{0}$
along the étale map $\widehat{N}_{\vp}^{\sim}\to\widetilde{N}_{\tc}\to\widetilde{N}_{\basic}^{r}$.

Now we compute
\begin{multline*}
\int_{\widehat{N}_{\vp}^{\sim}}\left[\bu_{N_{\basic}^{r}}^{*}e^{D\eta_{r}^{\left(2\right)}}\cdot\left(\tilde{\gamma}_{\basic}^{r}\tilde{\rho}_{0}\right)^{*}\left(\left(\bu_{\basic}^{r}\right)^{*}\omega_{r}\cdot\left(\widetilde{\ed}_{\basic}^{r}\right)^{*}\Lambda^{\boxtimes r}\right)\right]|_{\widehat{N}_{\vp}^{\sim}}=\left(-1\right)^{r\cdot r'}\int_{\widehat{N}_{\vp}^{\sim}}\bigg\{\\
\left(\acute{f}_{\vp}^{\sim}\right)^{*}\bu_{N_{\vp}}^{*}\left(\grave{f}_{\vp}^{*}e^{D\check{\eta}_{r}^{\left(2\right)}}\left(\pi_{\basic}^{r}\right)^{*}\check{\omega}_{r}|_{\check{F}_{\basic}^{r}}\left(E^{r'}\widetilde{\delta}_{\vp}\right)^{*}\Lambda^{\boxtimes r'}\right)\left(\ed_{\tc\left(\vp\right)}^{\leftarrow}\pi_{\widehat{F}_{\vp}}\bu_{\widehat{N}_{\vp}^{\sim}}\right)^{*}\Lambda^{\boxtimes r}\bigg\}=\\
=\xi_{\vp}\int_{\widetilde{N}_{\vp}}\bu_{N_{\vp}}^{*}\left(\grave{f}_{\vp}^{*}\left(e^{D\check{\eta}_{r}^{\left(2\right)}}\left(\pi_{\basic}^{r}\right)^{*}\check{\omega}_{r}|_{\check{F}_{\basic}^{r}}\right)\widetilde{\delta}_{\vp}^{*}\left(-\theta_{0}\right)^{\boxtimes r'}\right),
\end{multline*}
where in the last equality we applied the projection formula (see
\cite[Lemma 63]{twA8} for our sign conventions) for $\acute{f}_{\vp}^{\sim}$,
using the push-pull property for the cartesian commutative square
$\acute{\For}_{\vp}\left(\pi_{\widehat{F}_{\vp}}\bu_{\widehat{N}_{\vp}}\right)=\left(\pi_{\vp}\bu_{\vp}\right)\acute{f}_{\vp}$.
Eq (\ref{eq:sing fp formula by fp profile}) follows.
\end{proof}
In the Section \ref{sec:Resummation} we'll explain how to regularize
(\ref{eq:sing fp formula by fp profile}) to derive the general fixed-point
formula (\ref{eq:general fp formula}).

\subsection{Tubular neighbourhood construction}

\begin{defn}
Let $\mm$ be an orbifold with corners and let $V$ be a vector space.

(a) A vector bundle map $\delta:T\mm\to V\times\mm$ is \emph{strongly
surjective} if $\delta\circ di_{\mm}^{\partial^{c}}:T\partial^{c}\mm\to V\times\partial^{c}\mm$
is fiberwise surjective for every $c\geq0$. 

(b) Let $\mm$ be an orbifold with corners. A \emph{vector field}
on $\mm$ is a section $v:\mm\to T\mm$ of the map $T\mm\to\mm$.
A vector field $v$ will be called \emph{b-tangent }if $\left(i_{\mm}^{\partial^{c}}\right)^{-1}\left(v\right):\partial^{c}\mm\to\left(i_{\mm}^{\partial^{c}}\right)^{*}T\mm$
factors through $di_{\mm}^{\partial^{c}}:T\partial^{c}\mm\to\left(i_{\mm}^{\partial^{c}}\right)^{*}T\mm$
for all $c\geq0$. More generally if $v:\mathcal{E}\to\mm$ is a vector
bundle over $\mm$ we say a map $\mathcal{E}\to T\mm$ is b-tangent
if $v\circ s$ is b-tangent for every section $s:\mm\to\mathcal{E}$.
\end{defn}
\begin{lem}
\label{lem:b-tangent section}Let $\mm$ be a $\tb$-orbifold with
corners and $V$ a $\tb$-representation of finite rank. Let $\delta:T\mm\to V\times\mm$
a strongly surjective $\tb$-equivariant map. Then there exists a
$\tb$-equivariant b-tangent map $\sigma:V\times\mm\to T\mm$ with
$\delta\sigma=\id$.
\end{lem}
\begin{proof}
Consider first the case $\tb$ acts trivially on $V=\rr$ and $\mm=\rr_{c}^{n}$.
In terms of local coordinates $x=\left(x_{1},...,x_{n}\right)$ for
$\rr_{c}^{n}$ with $x_{i}\geq0$ for $1\leq i\leq c$, we can write
$\delta=a_{1}\left(x\right)dx_{1}+\cdots+a_{n}\left(x\right)dx_{n}$
for some functions $a_{i}:\rr_{c}^{n}\to\rr$ such that $\left(a_{1},...,a_{n}\right)\neq\boldsymbol{0}$
everywhere and $\left(a_{1}\left(x\right),...,\widehat{a_{i}}\left(x\right),...,a_{n}\left(x\right)\right)\neq\boldsymbol{0}$
for all $x$ with $x_{i}=0$. We take $\sigma=\sum_{i=1}^{n}\sigma_{i}\frac{\partial}{\partial x_{i}}$
for 
\[
\sigma_{j}=\frac{1}{\sum_{i=1}^{c}x_{i}a_{i}^{2}+\sum_{i=c+1}^{n}a_{i}^{2}}\cdot\begin{cases}
x_{j}a_{j} & \mbox{if }1\leq j\leq c\\
a_{j} & \mbox{if }c+1\leq j\leq n
\end{cases}.
\]
The general case follows, using a partition of unity and averaging
with respect to a Haar measure on $\tb$.
\end{proof}
\begin{proof}
[Proof of Lemma \ref{lem:N_T->N_0 compatible flows}]Use Lemma \ref{lem:tubular exists}
to fix some $\tb$-equivariant tubular neighbourhood  for 
\[
{F_{\tc}\subset\Delta_{\tc}:=\left(\ed_{\tc}\right)^{-1}\left(\Delta^{r'}\right)}.
\]
We thus reduce to proving the following.
\begin{claim*}
There exists an open $\tb$-invariant open neighbourhood $U_{1}$,
\[
p_{0}^{\times r'}\subset U_{1}\subset U_{0}^{r'}\subset L^{\times r'}\hookrightarrow\left(L\times L\right)^{r'},
\]
so that the closed immersion
\[
{U_{\tc}:=\ed_{\tc}^{-1}\left(U_{1}\right)\xrightarrow{f}\mm_{\tc}}
\]
given by the composition ${U_{\tc}\to\Delta_{\tc}\to\mm_{\tc}}$
admits a $\tb$-equivariant tubular neighbourhood
\[
N_{f}=N_{0}^{r'}\times U_{\tc}\xleftarrow{j_{f}}\mathcal{V}_{f}\xrightarrow{\gamma_{f}}\mm_{\tc}
\]
making the following diagram commute
\[
\xymatrix{N_{0}^{r'}\times U_{\tc}\ar[d]_{\pr_{1}} & \mathcal{V}_{f}\ard{l}{j_{f}}\ar[r]^{\gamma_{f}}\ar[d] & \mm_{\tc}\ar[d]^{\mbox{ed}_{\tc}}\\
N_{0}^{r'} & V_{0}^{r'}\ar[l]^{q_{0}^{r'}}\ard{r}{\gamma_{\Delta}^{r'}} & \left(L\times L\right)^{r'}
}
.
\]
\end{claim*}
We prove the claim. Set $W_{0}=\ed_{\tc}^{-1}\left(\gamma_{\Delta}^{r'}\left(V_{0}^{r'}\right)\right)\subset\mm_{\tc}$,
so that 
\[
Q:=\left(q_{0}\gamma_{\Delta}^{-1}\right)^{\times r'}\ed_{\tc}:W_{0}\to N_{0}^{r'}
\]
is well-defined, and let 
\[
\delta_{1}:=P\circ dQ:TW_{0}\to N_{0}^{r'}\times W_{0}
\]
be the linearization of $Q$ composed with the tautological ``parallel
transport to $0\in N_{0}$'' map $P:TN_{0}^{r'}\to T_{0}N_{0}^{r'}=N_{0}^{r'}$. 

For every $c\geq0$ the map ${\mbox{ed}{}_{\tc}\circ i_{\mm_{\tc}}^{\partial^{c}}}$
is transverse to the diagonal $\Delta^{r'}\to\left(L\times L\right)^{r'}$.
It follows that there exists a possibly smaller open neighbourhood
$\ed_{\tc}^{-1}\left(\left(p_{0},p_{0}\right)\right)\subset W_{1}\subset W_{0}$
such that the restriction $\delta_{1}:=TW_{1}\to N_{0}^{r'}\times W_{1}$
of $\delta_{0}$ is strongly surjective. 

Apply Lemma \ref{lem:b-tangent section} to construct a b-tangent
map $\sigma_{1}:N_{0}^{r'}\times W_{1}\to TW_{1}$ with $\delta_{1}\sigma_{1}=\id$.
Exponential flow along $\sigma_{1}$, defined for for some sufficiently
small suborbifold $\mathcal{V}_{f}$ of $N_{0}^{r'}\times U_{\tc}$
containing $0\times U_{\tc}$, gives the desired $\gamma_{f}$. This
completes the proof of Lemma \ref{lem:N_T->N_0 compatible flows}.
\end{proof}
\begin{rem}
If we want the flow to be defined uniformly, so that $\mathcal{V}_{f}=U_{\tc}\times V_{1}$
for some open neighbourhood $0\in V_{1}\subset N_{0}^{r'}$, we can
restrict $U_{\tc}$ further to a precompact open subset of $U_{\tc}$.
\end{rem}

\section{\label{sec:Resummation}Resummation}

For $\underline{\vp}\in\pc_{\basic}^{r,0}$, we denote $\pc_{\underline{\vp}}^{r'}:=\left\{ \vp\in\pc_{\basic}^{r,r'}|\cnt^{r'}\left(\vp\right)=\underline{\vp}\right\} $,
and set
\begin{equation}
\cont_{\underline{\vp}}=\sum_{r'\geq0}\frac{1}{r'!}\sum_{\vp\in\pc_{\underline{\vp}}^{r'}}\frac{\xi_{\vp}}{\boldsymbol{s}\left(\vp\right)!}\int_{\widetilde{N}_{\vp}}\bu_{N_{\vp}}^{*}\left(e^{D\eta_{\vp}}\pi_{\vp}^{*}\left(\omega_{\vp}\right)\right)\cdot\widetilde{\delta}_{\vp}^{*}\left(-\theta_{0}\right)^{\boxtimes r'}\label{eq:cont_vp}
\end{equation}
so that, by Corollary \ref{cor:sing formula by profiles}, 
\[
\int_{\basic}\omega=\sum_{r\geq0}\frac{1}{r!}\sum_{\underline{\vp}\in\pc_{\basic}^{r,0}}\text{Cont}_{\underline{\vp}}.
\]

The main result of this section is the following.
\begin{prop}
\label{thm:resummation}For $\underline{\vp}\in\pc_{\basic}^{r,0}$
we have
\[
\cont_{\underline{\vp}}=\frac{\xi_{\underline{\vp}}}{\boldsymbol{s}\left(\underline{\vp}\right)!}\int_{F_{\underline{\vp}}}e_{M}^{-1}\cdot e_{S}^{-1}\cdot\check{\omega}_{\underline{\vp}}
\]
where $e_{M},e_{S}$ are canonical equivariant Euler forms for $M_{\underline{\vp}},S_{\underline{\vp}}$,
respectively (cf. (\ref{eq:can theta_M}) and (\ref{eq:can theta_S})).
\end{prop}
Theorem \ref{thm:general fixed point formula} immediately follows
from this.

\subsection{\label{subsec:problems with singular}The difficulty with the singular
formula}

This subsection tries to explain the difficulty in computing $\text{Cont}_{\underline{\vp}}$,
and gives an indication of how we overcome this problem. The subsections
that follow are logically independent of this discussion.

\subsubsection{$r'\left(\vp\right)>0$ integrals don't vanish}

One might expect the terms in (\ref{eq:cont_vp}) with $r'\left(\vect\phi\right)>0$
to vanish. Indeed, $\eta_{\vect\phi}=f_{\vect\phi}^{*}\check{\eta}_{\vect\phi}$
is pulled back from a space of strictly lower dimension \textendash{}
but $\widetilde{\delta}_{\vect\phi}$ does not factor through $f_{\vect\phi}$.
The next example shows that these contributions can be nonzero.
\begin{example}
Consider the basic moduli specification $\basic=\left(\left\{ 1,2,3\right\} ,\left\{ 1\right\} ,2\right)$
for $\left(\cc\pp^{2},\rr\pp^{2}\right)$, and let $\vect\phi_{1}\in\pc_{\basic}^{0,1}$
be a fixed-point profile with $\tc=\tc\left(\vp_{1}\right)$ given
by $\tc_{0}=\left\{ v_{1},v_{2}\right\} $, $\sfr_{\tc}\left(v_{1}\right)=\left(\left(\left\{ 1\right\} ,\left\{ 1\right\} ,2\right),\sstar'_{1}\right)$
and $\sfr_{\tc}\left(v_{2}\right)=\left(\left(\left\{ \sstar''_{1},2,3\right\} ,\emptyset,0\right),\emptyset\right)$.
We have
\[
F_{\vect\phi_{1}}=F_{D_{1}}\left(\vect\phi_{1}\right)\times F_{\Sigma_{1}^{1}}\left(\vect\phi_{1}\right)\times F_{D_{2}}\left(\vect\phi_{1}\right)=\left[0,2\pi\right]\times\pt\times\pt
\]
where $\alpha\in\left[0,2\pi\right]=F_{D_{1}}\left(\vect\phi_{1}\right)$
is represented by the unit disc with the markings $\star_{1,1},1,\sstar'_{1}$
at $0,1,e^{\sqrt{-1}\alpha}\in\cc$, respectively. $\check{F}_{\vect\phi_{1}}=\pt$
and we have an equality of vector spaces
\[
\check{N}_{\vect\phi_{1}}=E\oplus\cc_{s}\oplus\left(T_{p_{0}}L\right)^{2}.
\]
where we've used the evaluation maps at $\sstar'_{1},\sstar''_{1}$
together with the translation action of the lie algebra of $O\left(3\right)\times O\left(3\right)$
to split off $\left(T_{p_{0}}L\right)^{2}\simeq\cc_{x,y}^{2}\simeq\rr_{x_{1},x_{2},y_{1},y_{2}}^{4}$.
We also split off $\cc_{s}$ using the map $\check{N}_{\vect\phi_{1}}\to\check{S}_{\vect\phi_{1}}=\mathbb{L}_{\star_{1,1}}^{\vee}\otimes\mathbb{L}_{\star'_{1,1}}^{\vee}\simeq\cc$
and its natural section which keeps the $v_{2}$ disc fixed and smoothes
the $\star'_{1,1}$-node so the boundary of the $v_{1}$ disc is a
circle centered at $p_{0}\in\rr\pp^{2}$.

We have 
\[
\delta_{\vect\phi_{1}}:N_{\vect\phi_{1}}=E\times\cc\times\left(T_{p_{0}}L\right)^{2}\times[0,2\pi]\to N_{0}=T_{p_{0}}L
\]
\[
\left(e,s,x,y,\alpha\right)\mapsto y-x-s\cdot e^{\sqrt{-1}\alpha}.
\]
We can take
\[
\eta_{\vect\phi_{1}}=\eta_{E}+\eta'=\eta_{E}+\left|s\right|^{2}d\arg s+x_{1}dx_{2}-x_{2}dx_{1}+y_{1}dy_{2}-y_{2}dy_{1}
\]
and then 
\begin{multline*}
\int_{N_{\vect\phi_{1}}}e^{D\eta_{\vect\phi_{1}}}\pi_{\vect\phi_{1}}^{*}\omega_{1}\tilde{\delta}_{\vect\phi_{1}}^{*}\left(-\theta_{0}\right)=\check{\omega}_{1}\left(\check{F}_{\vect\phi_{1}}\right)\cdot e_{E}^{-1}\times\bigg[\\
\int_{\cc\times\left(T_{p_{0}}L\right)^{2}\times[0,2\pi]}e^{-\left|s\right|^{2}-\left|x\right|^{2}-\left|y\right|^{2}+\lambda_{1}^{-1}d\eta'}d\arg\left(y-x-s\cdot e^{\sqrt{-1}\alpha}\right)\bigg]=\\
=\frac{\check{\omega}_{1}\left(\check{F}_{\vect\phi_{1}}\right)\cdot e_{E}^{-1}}{\lambda_{1}^{3}}\int e^{-x_{1}^{2}-x_{2}^{2}-y_{1}^{2}-y_{2}^{2}-\left(x_{1}-y_{1}\right)^{2}-\left(x_{2}-y_{2}\right)^{2}}\,dx_{1}\,dx_{2}\,dy_{1}\,dy_{2}=\\
=\frac{\check{\omega}_{1}\left(\check{F}_{\vect\phi_{1}}\right)\cdot e_{E}^{-1}}{\lambda_{1}^{3}}\cdot\frac{1}{3}\neq0.
\end{multline*}
\end{example}

\subsubsection{$r'\left(\vp\right)=0$ integral is ill-defined}

A closely related problem is that the contribution of the unique
$r'=0$ summand to (\ref{eq:cont_vp}),
\[
\frac{\xi_{\underline{\vp}}}{\boldsymbol{s}\left(\underline{\vp}\right)!}\int_{N_{\underline{\vp}}}e^{D\eta_{\underline{\vp}}}\pi_{\underline{\vp}}^{*}\omega_{\underline{\vp}}=\frac{\xi_{\underline{\vp}}}{\boldsymbol{s}\left(\underline{\vp}\right)!}\int_{F_{\underline{\vp}}}\sigma\,\omega_{\underline{\vp}},
\]
seems to depend on the choice of $\eta_{\underline{\vp}}$. Since
this integral is non-singular (i.e. it contains no $-\theta_{0}$
factors), we could express it as an integral of a Segre form $\sigma$
on $F_{\vect\phi}=\check{F}_{\vect\phi}$. But $\sigma$ does not
satisfy any boundary condition. More precisely, although   
\[
\eta_{\underline{\vp}}|_{\partial^{\vp_{+}}N_{\underline{\vp}}}=\left(f_{\vp_{+}}\circ\gamma_{\vp_{+}}\right)^{*}\check{\eta}_{\vp_{+}}
\]
for every $\vp_{+}\in\pc_{\underline{\vp}}^{1}$, this does not constitute
a boundary condition since 
\[
\dim\check{N}_{\vp_{+}}>\dim\partial^{\vp_{+}}N_{\underline{\vp}}.
\]
This problem is also related to the fact $\delta_{\vp_{+}}$ does
not factor through $f_{\vp_{+}}$; indeed, otherwise $\eta_{\underline{\vp}}|_{\partial^{\vp_{+}}N_{\underline{\vp}}}$
would be pulled back from 
\[
K_{\vp_{+}}=\ker\left(\check{N}_{\vp_{+}}\to N_{0}\right)
\]
 and we would have had
\[
\sigma|_{\partial^{\vp_{+}}F_{\underline{\vp}}}=\For_{\vp_{+}}^{*}\left(K_{\vp_{+}}\to\check{F}_{\vp_{+}}\right)_{*}e^{D\check{\eta}_{\vp_{+}}}.
\]

\subsubsection{\label{subsubsec:no pulled back splitting}(\ref{eq:M-N-S check gamma})
doesn't split}

Consider the diagram (\ref{eq:M-N-S check gamma}). The maps $\acute{\delta}_{M_{\vp_{+}}}^{r'+1}:=\delta_{M_{\vp_{+}}}^{r'+1,\left\{ r'+1\right\} }$
and $\acute{\delta}_{\vp_{+}}^{r'+1}:=\delta_{\vp_{+}}^{r'+1,\left\{ r'+1\right\} }$
are the cokernels of the fiberwise-injective maps $\acute{\gamma}_{M_{\vp_{+}}}$
and $\acute{\gamma}_{\vp_{+}}$, respectively (the map $\acute{\gamma}_{S_{\vp_{+}}}$
is cartesian). Since $\acute{\delta}_{\vp_{+}}^{r'+1}$ does not factor
through $\grave{f}_{\vp_{+}}^{\left\{ r'+1\right\} }$ it is not possible
to have compatible splittings for the three short exact sequences
in the diagram (\ref{eq:M-N-S check gamma}) (that is, splittings
so that (\ref{eq:M-N-S check gamma}) with the horizontal arrows reversed
by the splitting, also commutes). This may be seen as a third problem
with the integrals appearing in the singular localization formula
(though of course, it is closely related to the previous two).

\subsubsection{Resolution}

In $\S$\ref{subsec:Splitting-sess} we will construct sections $\check{N}_{\vp}\xrightarrow{\check{\zeta}_{\vp}}\check{M}_{\vp}$,
whose pullback we denote by $\zeta_{\vp}:N_{\vp}\to M_{\vp}$. The
l-bundle map 
\[
\hat{\delta}_{\vp}=\delta_{M_{\vp}}\circ\zeta_{\vp}:N_{\vp}\to N_{0}^{r'}\times F_{\vp}
\]
is a kind of correction for $\delta_{\vp}$, with the effect of $S_{\vp}$
pushing on the tail markings $\sstar'_{i}$ switched off. Indeed,
$\hat{\delta}_{\vp}$ (more precisely, $\left(\id\times\For_{\vp}\right)\circ\hat{\delta}_{\vp}$)
factors through the forgetful map $f_{\vp}$. We will see that $\delta_{\vp}$
and $\hat{\delta}_{\vp}$ admit a common section 
\[
s_{\vp}:N_{0}^{r'}\times F_{\vp}\to N_{\vp}
\]
so the isomorphism 
\[
A_{\vp}=\id+s_{\vp}\left(\hat{\delta}_{\vp}-\delta_{\vp}\right)
\]
is isotopic to the identity and satisfies $\delta_{\vp}A_{\vp}=\hat{\delta}_{\vp}$.
Setting $\hat{\gamma}_{\vp}:=A_{\vp}^{-1}\gamma_{\vp}$ we find that
$\hat{f}_{\vp}\hat{\gamma}_{\vp}$ factors through a map $\partial^{\vp}N_{\cnt\vp}\simeq\ker\hat{\delta}_{\vp}\to\ker\delta_{M_{\vp}}$
and $\hat{\gamma}_{\vp}=\gamma_{M_{\vp}}\oplus\gamma_{S_{\vp}}$ with
respect to the splittings induced by $\zeta_{\cnt\vp},\zeta_{\vp}$,
so the corrections $\hat{\gamma}_{\vp},\hat{\delta}_{\vp}$ address
all of the problems mentioned above. In $\S$\ref{subsec:Regularizing-and-resumming}
we will use the isotopies $\id\Rightarrow A_{\vp}$ to resum $\text{Cont}_{\vp}$.

\subsection{\label{subsec:Splitting-sess}Splitting short exact sequences}

For $\vp\in\pc_{\basic}^{r,r'}$, we set $\check{\delta}_{M_{\vp}}=\prod_{i=1}^{r'}\check{\delta}_{M_{\vp}}^{i}$.
We denote $F_{\basic}^{r,r'}=\coprod_{\vp\in\pc_{\basic}^{r,r'}}F_{\vp}$
and similarly for other fixed-point local constructions, e.g. $N_{\basic}^{r,r'}=\coprod_{\vp\in\pc_{\basic}^{r,r'}}N_{\vp}$,
$\delta_{\basic}^{r,r'}:=\coprod_{\vp\in\pc_{\basic}^{r,r'}}\delta_{\vp}$.

There's a natural action of the group $E_{\kf}^{r,r'}$ (see Definition
\ref{def:E_k^r,r'}) on $\pc_{\basic}^{r,r'}$, preserving the node
labels. The action of $E_{\kf}^{r,r'}$ on $\check{\mm}_{\basic}^{r,r'}$
then induces an action on $\check{N}_{\basic}^{r,r'},\check{F}_{\basic}^{r,r'}$
etc., compatible with the forgetful map to $\pc_{\basic}^{r,r'}$.
Similarly, $\Sym\left(r'\right)$ acts on $F_{\basic}^{r,r'}$. 

Considered  as a groupoid, $\pc_{\basic}^{r,r'}$ acts on $\check{F}_{\basic}^{r,r'}$
in a natural way\footnote{Concretely, this means that if we choose a representative $\vp_{\alpha}$
for each equivalence class ${\oo_{\alpha}=\Sym\left(\boldsymbol{s}\left(\vp_{\alpha}\right)\right)\vp_{\alpha}\subset\pc_{\basic}^{r,r'}}$
then we have an action of $\Sym\left(\boldsymbol{s}\left(\vp_{\alpha}\right)\right)$
on $\coprod_{\vp\in\oo_{\alpha}}\check{F}_{\vp}$.}. The map ${\check{F}_{\basic}^{r,r'}\to\left(\check{\mm}_{\basic}^{r,r'}\right)^{\tb}}$
is the quotient map for this action. This action lifts to \linebreak{}
${F_{\basic}^{r,r'},N_{\basic}^{r,r'},M_{\basic}^{r,r'}}$, and so
on, and the maps $\delta_{M_{\basic}^{r,r'}},\mu_{\basic}^{r,r',S},...$
respect it, as does the $E_{\basic}^{r,r'}$ action. We will tacitly
assume all of the constructions below continue to respect this action,
by averaging.

$\Sym\left(r'\right)$ acts on $N_{0}^{r'}\times F_{\basic}^{r,r'}$
and on $N_{0}^{r'}\times\check{F}_{\basic}^{r,r'}$ diagonally. For
$N_{0}^{r'}\times\check{F}_{\basic}^{r,r'}$, this extends to an action
of $E_{\kf}^{r,r'}$as follows. For $1\leq b\leq r'$ and $x\in\sstar'{}_{\left[r+r'\right]\backslash r+b}\coprod\kf\coprod\sstar''_{\left[r+r'\right]\backslash r+b}$,
\[
\mv_{x,e_{b}}\left(\left(y_{1},...,y_{r}\right),f\right)=\left(\left(y_{1}',...,y_{r}'\right),f'\right)
\]
where $f'=\mv_{x,e_{b}}.\,f$ and, letting $v_{i}$ and $v_{j}$ denote
the tail and head of $e_{r+b}$ in $\check{\pi}_{\basic}^{r+r'}\left(f\right)$,
we consider the following cases
\begin{itemize}
\item if $x\in\kf\coprod\sstar''_{\left[r\right]}$ and $x\in\sigma_{\tc}\left(v_{i}\right)$,
we set $y_{b}'=-y_{b}$, $y_{i}'=y_{i}$ for $i\neq b$
\item if $x=\sstar''_{r+a}$ for $a\neq b$, and $x\in\kf_{\tc}\left(v_{i}\right)$,
we set $y_{b}'=-y_{b}$, $y_{a}'=y_{a}+y_{b}$, $y_{i}'=y_{i}$ for
$i\neq a,b$.
\item if $x=\sstar'_{r+a}$ $a\neq b$ and $x\in\sigma_{\tc}\left(v_{i}\right)$,
we set $y_{a}'=y_{a}+y_{b}$ and $y_{i}'=y_{i}$ for $i\neq a$
\item if $x=\sstar'_{r+a}$ $a\neq b$ and $x\in\sigma_{\tc}\left(v_{j}\right)$,
we set $y'_{a}=y_{a}-y_{b}$ and $y'_{i}=y_{i}$ for $i\neq a$.
\item in all other cases we set $y_{i}'=y_{i}$ for all $i$.
\end{itemize}
The maps $\check{\delta}_{M_{\basic}^{r,r'}}=\coprod_{\vp\in\pc_{\basic}^{r,r'}}\check{\delta}_{M_{\vp}}$
and $\coprod_{\left\{ \vect\phi|r'\left(\vect\phi\right)=r'\right\} }\delta_{\vect\phi}$
are naturally $E_{\kf}^{r,r'}$ and $\Sym\left(r'\right)$ equivariant,
respectively. 
\begin{lem}
For every $r,r'\geq0$ there exists a $\tb\times E_{\kf}^{r,r'}$-equivariant
linear map 
\[
N_{0}^{r'}\times\check{F}_{\basic}^{r,r'}\xrightarrow{\check{s}_{M_{\basic}^{r,r'}}}\check{M}_{\basic}^{r,r'}
\]
 such that the following properties hold.

(a) $\check{s}_{M_{\basic}^{r,r'}}$ is a section of $\check{\delta}_{M_{\basic}^{r,r'}}$,
$\check{\delta}_{M_{\basic}^{r,r'}}\circ\check{s}_{M_{\basic}^{r,r'}}=\id$,

(b) if we let $s_{M_{\basic}^{r,r'}}:N_{0}^{r'}\times F_{\basic}^{r,r'}\to M_{\basic}^{r,r'}$
be the pullback defined by 
\[
{f_{M_{\basic}^{r,r'}}s_{M_{\basic}^{r,r'}}=\check{s}_{M_{\basic}^{r,r'}}\left(\id\times\grave{\For}_{\basic}^{r+r',\left\{ r+1,...,r+r'\right\} }\right)}
\]
 then the composition 
\[
N_{0}^{r'}\times F_{\basic}^{r,r'}\xrightarrow{s_{M_{\basic}^{r,r'}}}M_{\basic}^{r,r'}\to N_{\basic}^{r,r'},
\]
which we denote $s_{\basic}^{r,r'}$, is a $\Sym\left(r'\right)$-equivariant
section of $\delta_{\basic}^{r,r'}$.
\end{lem}
\begin{proof}
The $O\left(2m+1\right)^{r+r'+1}$-action on $\check{\mm}_{\basic}^{r,r'}$
induces a canonical map $\left(T_{p_{0}}L\right)^{r+r'+1}\times\check{F}_{\basic}^{r,r'}\xrightarrow{\alpha}\check{M}_{\basic}^{r,r'}$
which is $\tb\times E_{\kf}^{r,r'}$ equivariant, where $E_{\kf}^{r,r'}$
acts on $\left(T_{p_{0}}L\right)^{r+r'+1}$ by reordering the factors
according to $w_{\rho,\underline{\tc}}$ (see the proof of Lemma \ref{lem:R_k^r action}).
The map $\check{\delta}_{M_{\basic}^{r,r'}}\circ\alpha$ is surjective,
and we choose a vector bundle map
\[
N_{0}^{r'}\times\check{F}_{\basic}^{r,r'}\xrightarrow{\beta}\left(T_{p_{0}}L\right)^{r+r'+1}\times\check{F}_{\basic}^{r,r'}
\]
over $\id_{\check{F}_{\basic}^{r,r'}}$ which satisfies $\check{\delta}_{\vect\phi}\circ\alpha\circ\beta=\id$.
By averaging, we may assume without loss of generality that $\beta$
is $\tb\times E_{\kf}^{r,r'}$ equivariant, and we set $\check{s}_{M_{\basic}^{r,r'}}=\alpha\circ\beta$;
clearly, $\check{s}_{M_{\basic}^{r,r'}}$ is a $\tb\times E_{\kf}^{r,r'}$
equivariant section as claimed.

(b) follows from (a) by functoriality of the pullback: $\delta_{\basic}^{r,r'}\circ\mu_{\basic}^{r,r'}$
is the pullback of $\check{\delta}_{M_{\basic}^{r,r'}}$.
\end{proof}
We decompose $\check{s}_{M_{\basic}^{r,r'}}=\bigoplus_{i=1}^{r'}\check{s}_{M_{\basic}^{r,r'}}^{i}$
for $\check{s}_{M_{\basic}^{r,r'}}^{i}:N_{0}\times\check{F}_{\basic}^{r,r'}\to\check{M}_{\basic}^{r,r'}$.

Let $\vp_{+}\in\pc_{\basic}^{r,r'+1}$. To avoid excessive labels,
we use the following notation
\[
\acute{\delta}_{M_{\vp_{+}}}^{r'+1}=\acute{\delta}_{M_{\vp_{+}}}^{r'+1,\left\{ r+r'+1\right\} }:\acute{N}_{\vp_{+}}^{\left\{ r+r'+1\right\} }\to N_{0}\times\acute{F}_{\vp_{+}}^{\left\{ r+r'+1\right\} }
\]
 and 
\[
\acute{s}_{M_{\vp_{+}}}^{r'+1}:N_{0}\times\acute{F}_{\vp_{+}}^{\left\{ r+r'+1\right\} }\to\acute{N}_{\vp_{+}}^{\left\{ r+r'+1\right\} }
\]
for the pullbacks of $\check{\delta}_{M_{\basic}^{r,r'+1}}^{r'+1}$
and $\check{s}_{M_{\basic}^{r,r'+1}}^{r'+1}$.
\begin{prop}
\label{prop:splitting construction}Let $\basic$ be a basic moduli
specification. There exist $\tb\times E_{\kf}^{r,r'}$-equivariant
sections of $\check{\mu}_{\basic}^{r,r'}$
\[
\check{\zeta}_{\basic}^{r,r'}:\check{N}_{\basic}^{r,r'}\to\check{M}_{\basic}^{r,r'}
\]
for $r,r'\geq0$, such that the following \emph{modified-coherence}
property holds.

For $\vp_{+}\in\pc_{\basic}^{r,r'+1}$ let $\acute{\zeta}_{\vp_{+}}:\acute{N}_{\vp_{+}}^{\left\{ r+r'+1\right\} }\to\acute{M}_{\vp_{+}}^{\left\{ r+r'+1\right\} }$
be the pullback, defined by ${\grave{f}_{M_{\vp_{+}}}^{\left\{ r+r'+1\right\} }\acute{\zeta}_{\vp_{+}}=\check{\zeta}_{\basic}^{r,r'+1}|_{\check{F}_{\vp_{+}}}\grave{f}_{\vp_{+}}^{\left\{ r+r'+1\right\} }}$.
Setting $\vect\phi=\cnt\left(\vp_{+}\right)$ we have

\begin{equation}
\acute{\gamma}_{M_{\vp_{+}}}^{\emptyset}\,\check{\zeta}_{\vp}|_{\partial^{\vp_{+}}\check{N}_{\vp}}=\left(\id-\acute{s}_{M_{\vp_{+}}}^{r'+1}\acute{\delta}_{M_{\vp_{+}}}^{r'+1}\right)\,\acute{\zeta}_{\vp{}_{+}}\,\acute{\gamma}_{\vp_{+}}^{\emptyset}.\label{eq:pseudo-coherence}
\end{equation}
\end{prop}
\begin{proof}
Fix $r$. The proof is by reverse induction on $R\geq r$: assume
that we're given $\tb\times E_{\kf}^{r,r'}$-equivariant sections
$\check{\zeta}_{\basic}^{r,r'}$ of $\mu_{\basic}^{r,r'}$ defined
for all $r'\geq R-r+1$ such that (\ref{eq:pseudo-coherence}) holds
for all $\vp_{+}\in\pc_{\basic}^{r,\geq\left(R-r+2\right)}$. Note
that for sufficiently large $R$, this assumption holds vacuously.
We will now prove that there exists a $\tb\times E_{\kf}^{r,R-r}$-equivariant
section $\check{\zeta}_{\basic}^{r,R-r}:\check{N}_{\basic}^{r,R-r}\to\check{M}_{\basic}^{r,R-r}$
so that (\ref{eq:pseudo-coherence}) holds for all $\vp_{+}\in\pc_{\basic}^{r,R-r+1}$.

Set $r'=R-r$. First, for each $\vp_{+}\in\pc_{\basic}^{r,r'+1}$
we define $\check{\zeta}_{\partial^{\vp_{+}}}\in\Hom\left(\partial^{\vp_{+}}\check{N}_{\vp},\partial^{\vp_{+}}\check{M}_{\vp}\right)$
by (\ref{eq:pseudo-coherence}). Since $\acute{\gamma}_{M_{\vp_{+}}}^{\emptyset}$
is injective on fibers, such $\check{\zeta}_{\partial^{\vp_{+}}}$
is unique, and it exists since ${\acute{\delta}_{M_{\vp_{+}}}^{r'+1}\left(\id-\acute{s}_{M_{\vp_{+}}}^{r'+1}\acute{\delta}_{M_{\vp_{+}}}^{r'+1}\right)=0}$
(compare with $\S$\ref{subsubsec:no pulled back splitting}).

Consider the étale map ${\coprod_{\vp_{+}\in\pc_{\basic}^{r,r'+1}}\partial^{\vp_{+}}\check{F}_{\cnt\left(\vp_{+}\right)}\xrightarrow{q}\partial\check{F}_{\basic}^{r,r'}}.$
The invariance of \linebreak{}
${\left(\id-\acute{s}_{M_{\vp_{+}}}^{r'+1}\acute{\delta}_{M_{\vp_{+}}}^{r'+1}\right)\acute{\zeta}_{\basic}^{r}}$
under ${\left\{ \mv_{\sstar'_{a},e_{R+1}}\right\} _{1\leq a\leq R}\subset E_{\kf}^{r,r'+1}}$
(and the $\pc_{\basic}^{r,r'+1}$ groupoid action) implies that $\coprod_{\vp_{+}}\check{\zeta}_{\partial^{\vp_{+}}}$
is the $q$-pullback of a section $\partial\check{N}_{\basic}^{r,r'}\xrightarrow{\check{\zeta}_{\partial}}\partial\check{M}_{\basic}^{r,r'}$.

A straightforward computation shows $\check{\zeta}_{\partial}\partial\mu_{\basic}^{r,r'}=\id$.
We show $\check{\zeta}_{\partial}$ is $\tb\times E_{\kf}^{r,r'}$-equivariant.
Clearly, $\tb\times\Sym\left(e_{r+1},...,e_{r+r'}\right)$ equivariance
follows from equivariance of $\left(\id-\acute{s}_{M_{\vp_{+}}}^{r'+1}\acute{\delta}_{M_{\vp_{+}}}^{r'+1}\right)\acute{\zeta}_{\basic}^{r}$.
Now let $\vp_{+}\in\pc_{\basic}^{r,r'+1}$ and consider a generator
$\mv_{x,e_{b}}\in E_{\kf}^{r,r'}$. We have 
\begin{equation}
\alpha\;\acute{\gamma}_{\vp_{+}}^{\emptyset}=\acute{\gamma}_{\vp_{+}}^{\emptyset}\;\left(\mv_{x,e_{b}}.\right)\label{eq:lift of E gen}
\end{equation}
where $\alpha:\acute{N}_{\vp_{+}}^{\left\{ R+1\right\} }\to\acute{N}_{\vp_{+}}^{\left\{ R+1\right\} }$
satisfies
\[
\grave{f}_{\vp_{+}}^{\left\{ R+1\right\} }\alpha=e.\,\grave{f}_{\vp_{+}}^{\left\{ R+1\right\} }
\]
for some $e\in\left\{ \mv_{x,e_{b}},\mv_{x,e_{b}}\circ\mv_{\sstar'_{b},e_{r+R+1}},\mv_{x,e_{b}}\circ\mv_{x,e_{r+R+1}}\right\} \subset E_{\kf}^{r,r'+1}$.
Similarly, 
\[
\alpha'\;\acute{\gamma}_{M_{\vp_{+}}}^{\emptyset}=\acute{\gamma}_{M_{\vp_{+}}}^{\emptyset}\;\left(\mv_{x,e_{b}}.\right)
\]
where $\alpha':\acute{M}_{\vp_{+}}^{\left\{ R+1\right\} }\to\acute{M}_{\vp_{+}}^{\left\{ R+1\right\} }$
is a $\grave{f}_{M_{\vp_{+}}}^{\left\{ R+1\right\} }$-lift of $e$.
We have  
\[
\alpha'\left(\id-\acute{s}_{M_{\vp_{+}}}^{r'+1}\acute{\delta}_{M_{\vp_{+}}}^{r'+1}\right)\,\acute{\zeta}_{\vp{}_{+}}=\left(\id-\acute{s}_{M_{\vp_{+}}}^{r'+1}\acute{\delta}_{M_{\vp_{+}}}^{r'+1}\right)\,\acute{\zeta}_{\vp{}_{+}}\alpha
\]
which implies that $\check{\zeta}_{\partial}$ is invariant under
$\mv_{x,e_{b}}$ and hence under $E_{\kf}^{r,r'}$.

We want to show there exists an extension $\check{\zeta}_{\basic}^{r,r'}$
of $\check{\zeta}_{\partial}$ inward. For this it suffices to check
that $\partial\zeta_{\partial}:\partial^{2}\check{N}_{\basic}^{r,r'}\to\partial^{2}\check{M}_{\basic}^{r,r'}$
is $\zz/2$ equivariant with respect to the action that switches boundary
components. The argument is similar to the one given in the proof
of Proposition \ref{prop:localizing eta exists}. We have a decomposition
$\partial^{2}\check{F}_{\basic}^{r,r'}=\partial_{+}^{2}\check{F}_{\basic}^{r,r'}\coprod\partial_{-}^{2}\check{F}_{\basic}^{r,r'}$
induced from the decomposition $\partial^{2}\check{\mm}_{\basic}^{r,r'}=\partial_{-}^{2}\check{\mm}_{\basic}^{r,r'}\coprod\partial_{+}^{2}\check{\mm}_{\basic}^{r,r'}$.
For any 2-flag\emph{ }$\left(\vp_{0},\vp_{1},\vp_{2}\right)$, $\vp_{2}\in\pc_{\basic}^{r,r'+2}$,
$\vp_{i-1}=\cnt\vp_{i}$, there are maps $\partial^{\vp_{2}}\partial^{\vp_{1}}\check{N}_{\vp}\xrightarrow{\gamma}\check{N}_{\vp_{2}}$
and $\partial^{\vp_{2}}\partial^{\vp_{1}}\check{M}_{\vp}\xrightarrow{\gamma_{M}}\check{M}_{\vp_{2}}$
induced by $g$, and we have 
\[
\gamma_{M}\partial^{\vp_{2}}\check{\zeta}_{\partial^{\vp_{1}}}=\underbrace{\left(\id-\check{s}_{M_{\vp_{2}}}^{r'+1}\check{\delta}_{M_{\vp_{2}}}^{r'+1}-\check{s}_{M_{\vp_{2}}}^{r'+2}\check{\delta}_{M_{\vp_{2}}}^{r'+2}\right)\check{\zeta}_{\vp{}_{2}}}_{\rho}\,\gamma,
\]
so $\zz/2$-equivariance of $\partial\check{\zeta}_{\partial}|_{\partial_{-}^{2}F_{\basic}^{r,r'}}$
follows from $\Sym\left(e_{R+1},e_{R+2}\right)$-equivariance of $\rho$.
$\zz/2$-equivariance of $\partial\check{\zeta}_{\partial}|_{\partial_{+}^{2}F_{\basic}^{r,r'}}$
follows from $\left\langle \mv_{a,e_{R+1}}\right\rangle $ equivariance
of \linebreak{}
${\left(\id-\check{s}_{M_{\vp_{+}}}^{r'+1}\check{\delta}_{M_{\vp_{+}}}^{r'+1}\right)\,\check{\zeta}_{\vp{}_{+}}}$.
We conclude that an extension $\check{\zeta}_{\basic}^{r,r'}$ of
$\check{\zeta}_{\partial}$ exists. 

By averaging, we can ensure that $\check{\zeta}_{\basic}^{r,r'}$
is $\tb\times E_{\kf}^{r,r'}$-equivariant. This completes the inductive
step and the proof of the proposition.
\end{proof}
For $\vp\in\pc_{\basic}^{r,r'}$, define ${\zeta_{\vp}:N_{\vp}\to M_{\vp}}$
by requiring ${f_{M_{\vp}}\zeta_{\vp}=\check{\zeta}_{\basic}^{r,r'}f_{\vp}}$
and set 
\[
\hat{\delta}_{\vp}=\prod_{i=1}^{r'}\hat{\delta}_{\vp}^{i}=\delta_{M_{\vp}}\circ\zeta_{\vp}.
\]

\subsection{\label{subsec:Regularizing-and-resumming}Regularizing and resumming}

A\emph{ flag $\vpb$ of length $l\left(\vp_{\bullet}\right)=l$ and
depth $r'\left(\vp_{\bullet}\right)=r'$} is a tuple 
\[
\left(\vp_{0},\vp_{1},...,\vp_{l}\right)\in\pc_{\basic}^{r,r'}\times\cdots\times\pc_{\basic}^{r,r'+l},
\]
such that $\vp_{i-1}=\cnt\vp_{i}$ for $1\leq i\leq l$. We say $\vpb$
\emph{lies over }$\underline{\vp}=\cnt^{r'}\vp_{0}\in\pc_{\basic}^{r,0}$.
We fix\emph{ }some \emph{$\underline{\vp}\in\pc_{\basic}^{r,0}$ }and
denote by $\Psi^{r',l}$ the set of flags of depth $r'$ and length
$l$ over $\underline{\vp}$.\emph{ }

$\Sym\left(r'\right)\times\Sym\left(l\right)<\Sym\left(r+r'+l\right)$
acts on $\Psi^{r',l}$ by relabeling the edges of $\vp_{l}$ (a flag
$\vpb$ is uniquely specified by $\vp_{l}$ and the length). Each
flag $\vpb$ specifies a clopen corner component $F_{\vpb}=\partial^{\vp_{l}}F_{\vect\phi_{0}}\subset\partial^{l}F_{\vp_{0}}$
whose local boundary faces correspond, in this order, to the edges
$\left(r+r'+1,...,r+r'+l\right)$ of $\tc\left(\vp_{l}\right)$. We
use similar notation for corner components of the other spaces, bundles
and maps, e.g. $\partial^{\vpb}N_{\vp_{0}}=\partial^{\vp_{l}}N_{\vp_{0}}\subset\partial^{l}N_{\vp_{0}}$.
The action of ${\Sym\left(r'\right)\times\Sym\left(l\right)}$ lifts
to $\coprod_{\vpb\in\Psi^{r,l'}}F_{\vpb}$. We set $N_{\vpb}:=N_{\vp_{0}}|_{F_{\vpb}}=\partial^{\vpb}N_{\vp_{0}}\xrightarrow{\pi_{\vpb}}F_{\vpb}$.

Let $I=\left[0,1\right]$ be the unit interval. For each $\vpb\in\Psi^{r',l}$
we define
\[
A_{\vpb}\left(T,t_{1},...,t_{l}\right):N_{\vpb}\times I_{T}\times I_{\left(t_{1},...,t_{l}\right)}^{l}\to N_{\vpb}\times I\times I^{l}
\]
by 
\begin{multline}
\gamma_{\vpb}A_{\vpb}\left(T,\vect t\right):=\left(\id+\left(1-T\right)\cdot\left[\sum_{i=1}^{r'}s_{\vp_{l}}^{i}\left(\hat{\delta}_{\vp_{l}}^{i}-\delta_{\vp_{l}}^{i}\right)+\sum_{i=1}^{l}t_{i}\cdot s_{\vp_{l}}^{r'+i}\hat{\delta}_{\vect\phi_{l}}^{r'+i}\right]\right)\gamma_{\vpb}\label{eq:def of A_psi(t,mu)-1}
\end{multline}

\begin{lem}
\label{lem:A properties}$A_{\vpb}$ satisfies the following properties.

(a) Let $\vpb=\left(\vp_{0},...,\vp_{l}\right)$ be a flag of length
$l$ and depth $r'$. We can write $A_{\vpb}\left(T,\boldsymbol{t}\right)$
as a product of commuting operators
\[
A_{\vpb}\left(T,\boldsymbol{t}\right)=\prod_{i=1}^{r'}A_{\vpb,i}\left(T\right)\cdot\prod_{j=r'+1}^{r'+l}A_{\vpb,j}\left(T,t_{j}\right)
\]
with 
\[
{A_{\vpb,i}\left(T\right):=\id+\left(1-T\right)s_{\vp_{l}}^{i}\left(\hat{\delta}_{\vp_{l}}^{i}-\delta_{\vp_{l}}^{i}\right)}
\]
 for $1\leq i\leq r'$ and
\[
{A_{\vpb,j}\left(T,t_{j}\right):=\id+\left(1-T\right)t_{j}\cdot s_{\vp_{l}}^{j}\left(\hat{\delta}_{\vp_{l}}^{j}\right)}
\]
for ${r'+1\leq j\leq r'+l}$. The operators $A_{\vpb,i}\left(T\right)$
for $1\leq i\leq r'$ are pulled back along ${\acute{f}_{\vpb}^{\left\{ i\right\} }:=\partial^{\vpb}\acute{f}_{\vp_{0}}^{\left\{ i\right\} }}$.

(b) Let $\vpb=\left(\vp_{0},...,\vp_{l+1}\right)$ be a flag of length
$l+1$, and let $\vpb^{\geq1}=\left(\vp_{1},...,\vp_{l+1}\right)$
and $\vpb^{\leq l}=\left(\vp_{0},...,\vp_{l}\right)$ be its truncations.
Then 

\[
A_{\vpb}\left(T,t_{1},...,t_{l},t_{l+1}=0\right)=\partial^{\vp_{l+1}}A_{\vpb^{\leq l}}\left(T,t_{1},...,t_{l}\right)
\]
and
\[
\partial^{\vpb^{\geq1}}\gamma_{\vp_{1}}\left[A_{\vpb}\left(T,t_{1}=1,t_{2},...,t_{l+1}\right)\right]=\left[A_{\vpb^{\geq1}}\left(T,t_{2},...,t_{l+1}\right)\right]\partial^{\vpb^{\geq1}}\gamma_{\vp_{1}}.
\]

(c) \textup{Let $\vpb\in\Psi^{r',l}$. for all $1\leq i\leq r'$ we
have}
\[
\partial^{\vpb}\delta_{\vp_{0}}^{i}\circ A_{\vpb}\left(0,t_{1},...,t_{l}\right)=\partial^{\vpb}\hat{\delta}_{\vp_{0}}^{i}.
\]
\end{lem}
\begin{proof}
Straightforward.
\end{proof}
We set $\delta_{\vpb}=\partial^{\vpb}\delta_{\vp_{0}}$ and define
$\widetilde{\delta}_{\vpb}$ and $\beta_{\vpb}$ so that there's
a cartesian diagram 
\[
\xymatrix{\widetilde{N}_{\vpb}\ar[r]^{\beta_{\vpb}}\ar[d]_{\widetilde{\delta}_{\vpb}=\prod_{i=1}^{r'}\widetilde{\delta}_{\vpb}^{i}} & N_{\vpb}\ar[d]^{\delta_{\vpb}}\\
\widetilde{N}_{0}^{r'}\times F_{\vpb}\ar[r]_{\beta_{0}^{r'}} & N_{0}^{r'}\times F_{\vpb}
}
\]
Setting

\[
\omega_{\vpb}=\gamma_{\vpb}^{*}\omega_{r+r'+l}=\left(i_{F_{\vp_{0}}}^{\partial^{\vpb}}\right)^{*}\omega_{r+r'}\text{ and }\eta_{\vpb}=\gamma_{\vpb}^{*}\eta_{\vp_{l}}=\left(i_{N_{\vp_{0}}}^{\partial^{\vpb}}\right)^{*}\eta_{\vp_{0}},
\]
we define an equivariant form on $\widetilde{N}_{\vpb}\times I\times I^{l}$
by

\begin{equation}
\ii_{\vpb}=\xi_{\vp_{0}}\beta_{\vpb}^{*}\left(e^{DA_{\vpb}^{-1}\left(T,t_{1},...,t_{l}\right)^{*}\eta_{\vpb}}\pi_{\vpb}^{*}\omega_{\vpb}\right)\widetilde{\delta}_{\vpb}^{*}\left(-\theta_{0}\right)^{r'}.\label{eq:resum integrand}
\end{equation}

The following proposition is the heart of the resummation result,
replacing $\text{Cont}_{\underline{\vp}}$'s original sum of singular
integrands corresponding to flags of increasing \emph{depth} by a
sum of regular, depth zero integrands indexed by flags of increasing
\emph{length}, whose domain of integration are corners of $N_{\underline{\vp}}$. 

In what follows $\widetilde{N}_{\vpb}$ is oriented using the outward
normal orientation, and ${\widetilde{N}_{\vpb}\times I\times I^{l}}$
is oriented as a product of oriented spaces.
\begin{prop}
\label{prop:regularizing}We have
\begin{align}
\text{Cont}\left(\underline{\vp}\right)=- & \sum_{r'\geq0}\frac{1}{r'!}\sum_{\vpb\in\Psi^{r',0}}\frac{1}{\boldsymbol{s}\left(\vp_{0}\right)!}\int_{\widetilde{N}_{\vpb\times\left\{ 1\right\} }}\ii_{\vpb}\nonumber \\
= & \sum_{l\geq0}\frac{1}{l!}\sum_{\vpb\in\Psi^{0,l}}\frac{1}{\boldsymbol{s}\left(\vp_{l}\right)!}\int_{N_{\vpb}\times0\times I^{l}}\ii_{\vpb}.\label{eq:prop 9}
\end{align}
\end{prop}
\begin{proof}
The first equality is immediate (the minus sign appears because the
orientation on $1\in\partial I$ is negative). By Lemma \ref{lem:int D I psi =00003D0}
below, 
\[
0=\sum_{\vpb\in\Psi^{r',l}}\frac{1}{r'!\cdot l!\cdot\boldsymbol{s}\left(\vp_{l}\right)!}\int_{\widetilde{N}_{\vpb}\times I\times I^{l}}D\ii_{\vpb}=\sum_{\vpb\in\Psi^{r',l}}\frac{1}{r'!\cdot l!\cdot\boldsymbol{s}\left(\vp_{l}\right)!}\int_{\partial\left(\widetilde{N}_{\vpb}\times I\times I^{l}\right)}\ii_{\vpb}.
\]
We have
\begin{multline*}
\partial\left(\widetilde{N}_{\vpb}\times I\times I^{l\left(\vpb\right)}\right)=\left(\left(\partial_{-}\widetilde{N}_{\vpb}\coprod\partial_{+}\widetilde{N}_{\vpb}\coprod\coprod_{i=1}^{r'}\left(\widetilde{\delta}_{\vpb}^{i}\right)^{-1}\left(S\left(N_{0}\right)\right)\right)\times I\times I^{l}\right)\coprod\\
\coprod\left(\widetilde{N}_{\vpb}\times\left\{ 0\right\} \times I^{l}\right)\coprod\left(\widetilde{N}_{\vpb}\times\left\{ 1\right\} \times I^{l}\right)\coprod\left(\widetilde{N}_{\vpb}\times I\times\partial\left(I^{l}\right)\right).
\end{multline*}
Here $\partial_{\pm}\widetilde{N}_{\vpb}$ denotes the $\beta_{\vpb}$-inverse
image of the decomposition \linebreak{}
${\partial N_{\vpb}=\partial_{+}^{f_{\vpb}}N_{\vpb}\coprod\partial_{-}^{f_{\vpb}}N_{\vpb}}$
into horizontal and vertical components w.r.t. $f_{\vpb}$. By Lemma
\ref{lem:partial + =00003D 0} we have ${\int_{\partial_{+}\widetilde{N}_{\vpb}\times I\times I^{l}}\ii_{\vpb}=0}$,
and by Lemmas \ref{lem:partial- cancels t=00003D0},\ref{lem:S canceled by t=00003D1}
the contributions of \linebreak{}
${\widetilde{N}_{\vpb}\times I\times\partial\left(I^{l}\right)}$,
${\partial_{-}\widetilde{N}_{\vpb}}$ and ${\coprod_{i=1}^{r'}\left(\widetilde{\delta}_{\vpb}^{i}\right)^{-1}\left(S\left(N_{0}\right)\right)}$
cancel out. 

$\int_{\widetilde{N}_{\vpb}\times\left\{ 1\right\} \times I^{l}}\ii_{\vpb}=0$
unless $l=0$, since $A_{\vpb}\left(T=1,t_{1},...,t_{l}\right)=\id$
is independent of $t_{1},...,t_{l}$ so $\ii_{\vpb}|_{T=1}$ is pulled
back along the projection $\widetilde{N}_{\vpb}\times\left\{ 1\right\} \times I^{l}\to\widetilde{N}_{\vpb}$.
The $l=0$ contributions give the middle expression in (\ref{eq:prop 9}),
without the minus sign.

$\int_{\widetilde{N}_{\vpb}\times\left\{ 0\right\} \times I^{l}}\ii_{\vpb}=0$
unless $r'=0$. Indeed, using Lemma \ref{lem:A properties}(c) we
find that
\[
\int_{\widetilde{N}_{\vpb}\times0\times I^{l}}\ii_{\vpb}=\int_{\widehat{N}_{\vpb}}A_{\vpb}\left(0,\boldsymbol{t}\right)^{*}\ii_{\vpb}
\]
where $\widehat{N}_{\vpb}=A_{\vpb}\left(0,\boldsymbol{t}\right)^{-1}\widetilde{N}_{\vpb}$
sits in following diagram where both squares are cartesian
\[
\xymatrix{\widehat{N}_{\vpb}\ar[r]\ar[d] & N_{\vpb}\ar[d]^{\hat{\delta}_{\vpb}}\ar[r] & \check{N}_{\vpb}\ar[d]^{\check{\delta}_{\vpb}}\\
\widetilde{N}_{0}^{r'}\times F_{\vpb}\ar[r] & N_{0}^{r'}\times F_{\vpb}\ar[r] & N_{0}^{r'}\times\check{F}_{\vpb}
}
\]
$A_{\vpb}\left(0,\boldsymbol{t}\right)^{*}\ii_{\vpb}$ is pulled
back along $\widehat{N}_{\vpb}\times\left\{ 0\right\} \times I^{l}\to\check{N}_{\vpb}\times\left\{ 0\right\} \times I^{l}$.
For $r'>0$, $\dim\check{N}_{\vpb}<\dim\widehat{N}_{\vpb}$. The $r'=0$
contributions give the right expression in (\ref{eq:prop 9}), and
this concludes the proof of the proposition.
\end{proof}
Let $\vpb=\left(\vp_{0},...,\vp_{l+1}\right)$ be a flag of length
$l+1$ and let $\vpb^{\geq1}=\left(\vp_{1},...,\vp_{l+1}\right)$
and $\vp^{\leq l}=\left(\vp_{0},...,\vp_{l}\right)$ be its truncations.
By Lemma \ref{lem:A properties}(b), 
\begin{equation}
\left(i_{\widetilde{N}_{\vpb^{\leq l}}}^{\partial^{\vp_{l+1}}}\right)^{*}\ii_{\vpb^{\leq l}}=\ii_{\vpb}|_{\left\{ t_{l+1}=0\right\} }.\label{eq:psi leq l and t=00003D0}
\end{equation}
Factor $\beta_{\vpb^{\geq1}}$ as
\[
\widetilde{N}_{\vpb^{\geq1}}\xrightarrow{\beta_{\vpb^{\geq1}}^{1}}N_{\psi_{\vpb^{\geq1}}}^{\neq1}\xrightarrow{\beta_{\vpb^{\geq1}}^{\neq1}}N_{\vpb^{\geq1}}
\]
where $\beta_{\vpb^{\geq1}}^{\neq1}=\left(\delta_{\vpb^{\geq1}}^{\neq1}\right)^{-1}\left(\beta_{0}^{r'}\right)$
denotes the partial blow up and $\beta_{\vpb^{\geq1}}^{1}$ is defined
similarly. Set
\[
\ii_{\vpb^{\geq1}}^{\neq1}:=\left(\beta_{\vpb^{\geq1}}^{\neq1}\right)^{*}\left(e^{DA_{\vpb^{\geq1}}^{-1}\left(T,t_{2},...,t_{l+1}\right)^{*}\eta_{\vpb^{\geq1}}}\pi_{\vpb^{\geq1}}^{*}\omega_{\vpb^{\geq1}}\right)\left(\widetilde{\delta}_{\vpb}^{\neq1}\right)^{*}\left(-\theta_{0}\right)^{r'},
\]
so that 
\begin{equation}
\ii_{\vpb^{\geq1}}=\left(-1\right)^{r'}\left(\beta_{\vpb^{\geq1}}^{1}\right)^{*}\ii_{\vpb^{\geq1}}^{\neq1}\cdot\left(\widetilde{\delta}_{\vpb^{\geq1}}^{1}\right)^{*}\left(-\theta_{0}\right).\label{eq:decompose ii}
\end{equation}
There's a map ${\widetilde{N}_{\vpb}\xrightarrow{\partial^{\vpb^{\geq1}}\widetilde{\gamma}_{\vp_{1}}}\widetilde{N}_{\vpb^{\geq1}}^{\neq1}}$
induced from $\partial^{\vpb^{\geq1}}\gamma_{\vp_{1}}$, and using
Lemma \ref{lem:A properties}(b) again, we have
\begin{equation}
\left(\partial^{\vpb^{\geq1}}\widetilde{\gamma}_{\vp_{1}}\right)^{*}\ii_{\vpb^{\geq1}}^{\neq1}=\ii_{\vpb}|_{\left\{ t_{1}=1\right\} }.\label{eq:psi geq 1 and t =00003D 1}
\end{equation}

\begin{lem}
\label{lem:int D I psi =00003D0}We have 
\[
\int_{N_{\vpb}\times I\times I^{l}}D\ii_{\vpb}=0
\]
\end{lem}
\begin{proof}
$D\left(e^{DA_{\vpb}^{-1}\left(T,t_{1},...,t_{l}\right)^{*}\eta_{\vpb}}\pi_{\vpb}^{*}\omega_{\vpb}\right)=0$
and $D\theta_{0}$ is a constant. Using $\Sym\left(r'\right)$ symmetry
we reduce to showing 
\[
\int_{N_{\vpb}^{\neq1}\times I\times I^{l}}\ii_{\vpb}^{\neq1}=\int_{N_{\vpb}^{\neq1}\times I\times I^{l}}A_{\vpb}^{1}\left(T,t_{1},...,t_{l}\right)^{*}\ii_{\vpb}^{\neq1}
\]
vanishes. But $A_{\vpb}^{1}\left(T,t_{1},...,t_{l}\right)^{*}\ii_{\vpb}^{\neq1}\in\im\left(\acute{f}_{\vpb}^{\left[r'+l\right]\backslash\left\{ i\right\} }\right)^{*}$
and the codomain of $\acute{f}_{\vpb}^{\left[r'+l\right]\backslash\left\{ i\right\} }$
has lower dimension than its domain, so the integral does vanish.
\end{proof}

\begin{lem}
\label{lem:partial + =00003D 0}We have 
\[
\int_{\partial_{+}\widetilde{N}_{\vpb}\times I\times I^{l}}\ii_{\vpb}=0.
\]
\end{lem}
\begin{proof}
The involution $\tau_{\basic}^{r}:\partial_{+}\mm_{\basic}^{r}\to\partial_{+}\mm_{\basic}^{r}$
(cf. (\ref{eq:or rev involution})) induce an orientation-reversing
automorphism $\tau_{\vpb}$ of $\partial_{+}N_{\vpb}\times I\times I^{l}$
which commutes with $f_{\vpb}$ and $\delta_{\vpb}$, and so lifts
to an orientation-reversing automorphism $\widetilde{\tau}_{\vpb}$
of $\partial_{+}\widetilde{N}_{\vpb}\times I\times I^{l}$, such that
$\beta_{\vpb}\widetilde{\tau}_{\vpb}=\tau_{\vpb}\beta_{\vpb}$ and
$\widetilde{\delta}_{\vpb}^{i}\widetilde{\tau}_{\vpb}=\widetilde{\delta}_{\vpb}^{i}$.
These properties of $\tau_{\vpb}$ and $\widetilde{\tau}_{\vpb}$
imply $\widetilde{\tau}_{\vpb}^{*}\ii_{\vpb}=\ii_{\vpb}$, and the
claim follows.
\end{proof}
\begin{lem}
\label{lem:partial- cancels t=00003D0}For $r',l\geq0$ we have 

\begin{multline}
\sum_{\vpb'\in\Psi^{r',l}}\frac{1}{l!\cdot\boldsymbol{s}\left(\vp_{l}\right)!}\int_{\partial_{-}\widetilde{N}_{\vpb'}\times I\times I^{l}}\ii_{\vpb'}=\\
=\sum_{\vpb\in\Psi^{r',l+1}}\frac{1}{\left(l+1\right)!\cdot s\left(\vp_{l+1}\right)!}\sum_{j=1}^{l+1}\int_{\widetilde{N}_{\vpb}\times I\times I^{j-1}\times\left\{ 0\right\} \times I^{l-j}}\ii_{\vpb}.\label{eq:lemma 12-1}
\end{multline}
\end{lem}
\begin{proof}
We can use (\ref{eq:psi leq l and t=00003D0}) to replace the left
hand side of (\ref{eq:lemma 12-1}) by 
\begin{equation}
\sum_{\vpb\in\Psi^{r',l+1}}\frac{1}{l!\cdot\boldsymbol{s}\left(\vp_{l+1}\right)!}\int_{\widetilde{N}_{\vpb}\times I\times I^{l}\times\left\{ 0\right\} }\ii_{\vpb}.\label{eq:full flags}
\end{equation}
To see this, consider the maps
\begin{align*}
\coprod_{\vpb\in\Psi^{r',l}}\partial_{-}F_{\vpb} & \xrightarrow{f}\partial_{-}^{l+1}F_{\basic}^{r+r'}\\
\coprod_{\vpb\in\Psi^{r',l+1}}F_{\vpb} & \xrightarrow{f_{+}}\partial_{-}^{l+1}F_{\basic}^{r+r'}
\end{align*}
For every point $\pt\xrightarrow{p}\partial_{-}^{l+1}F_{\basic}^{r+r'}$,
the essential fibers $f^{-1}\left(p\right)$ and $f_{+}^{-1}\left(p\right)$
are either both empty or have sizes $\boldsymbol{s}\left(\vp_{l}\right)!$
and $\boldsymbol{s}\left(\vp_{l+1}\right)!$, respectively. Finally,
(\ref{eq:full flags}) is equal to the right hand side of (\ref{eq:lemma 12-1})
by $\Sym\left(l+1\right)$-equivariance.
\end{proof}
\begin{lem}
\label{lem:S canceled by t=00003D1}For $r',l\geq0$, we have
\begin{multline}
\sum_{\vpb\in\Psi^{r'+1,l}}\frac{1}{\left(r'+1\right)!\,l!\,\boldsymbol{s}\left(\vp_{l}\right)!}\sum_{i=1}^{r'+1}\int_{\left(\widetilde{\delta}_{\vpb}^{i}\right)^{-1}\left(S\left(N_{0}\right)\right)\times I\times I^{l}}\ii_{\vpb}=\\
=\sum_{\vpb\in\Psi^{r',l+1}}\frac{1}{r'!\,\left(l+1\right)!\,\boldsymbol{s}\left(\vp_{l}\right)!}\sum_{j=1}^{l+1}\int_{\widetilde{N}_{\vpb}\times I\times I^{j-1}\times\left\{ 1\right\} \times I^{l-j}}\ii_{\vpb}\label{eq:eq:lemma 12}
\end{multline}
\end{lem}
\begin{proof}
We can write every $\vpb\in\Psi^{r'+1,l}$ uniquely as $\vpb^{\geq1}$
for $\vpb\in\Psi^{r',l+1}$, and use $\Sym\left(r'+1\right)$ invariance
and Eqs (\ref{eq:decompose ii},\ref{eq:psi geq 1 and t =00003D 1})
to express the left hand side of (\ref{eq:lemma 12}) as 
\[
\sum_{\vpb^{\geq1}\in\Psi^{r'+1,l}}\frac{1}{r'!\,l!\,\boldsymbol{s}\left(\vp_{l}\right)!}\int_{\left(\widetilde{\delta}_{\vpb^{\geq1}}^{1}\right)^{-1}\left(S\left(N_{0}\right)\right)\times I\times I^{l}}\left(-1\right)^{r'}\left(\beta_{\vpb^{\geq1}}^{1}\right)^{*}\ii_{\vpb^{\geq1}}^{\neq1}\cdot\left(\widetilde{\delta}_{\vpb^{\geq1}}^{1}\right)^{*}\left(-\theta_{0}\right)=
\]

\[
=\sum_{\vpb^{\geq1}\in\Psi^{r'+1,l}}\frac{1}{r'!\,l!\,\boldsymbol{s}\left(\vp_{l}\right)!}\int_{N_{\vpb}^{\neq1}\times I\times I^{l}}\left(\partial^{\vpb^{\geq1}}\widetilde{\gamma}_{\vp_{1}}\right)^{*}\ii_{\vpb^{\geq1}}^{\neq1}=\sum_{\vpb^{\geq1}\in\Psi^{r'+1,l}}\frac{1}{r'!}\frac{1}{l!}\int_{N_{\vpb}\times I\times I^{l}\times\left\{ 1\right\} }\ii_{\vpb}
\]
which by $\Sym\left(l\right)$-equivariance is equal to the right
hand side of (\ref{eq:lemma 12}).
\end{proof}

We turn now to resumming the right hand side of (\ref{eq:prop 9}).
Consider $\vp_{+}\in\pc_{\basic}^{r,1}$ with $\cnt\vp_{+}=\underline{\vp}$.
Since $\left(\id_{N_{0}}\times\For_{\vp_{+}}\right)\hat{\delta}_{\vp_{+}}=\check{\delta}_{\vp_{+}}f_{\vp_{+}}$
we get an induced cartesian map $\partial^{\vp_{+}}N_{\underline{\vp}}\xrightarrow{h_{\vp_{+}}}\ker\check{\delta}_{\vp_{+}}$
lying over $\For_{\vp_{+}}$, sitting in a map of short exact sequences

\[
\xymatrix{0\ar[r] & \partial^{\vp_{+}}N_{\underline{\vp}}\ar[rr]^{\hat{\gamma}_{\vp_{+}}=A^{-1}\gamma_{\vp_{+}}}\ar[d]_{h_{\vp_{+}}} &  & N_{\vp_{+}}\ar[rr]^{\hat{\delta}_{\vp_{+}}=\delta_{\vp_{+}}A}\ar[d]_{f_{\vp_{+}}} &  & N_{0}\times F_{\vp_{+}}\ar[r]\ar[d]^{\id\times\For_{\vp_{+}}} & 0\\
0\ar[r] & \ker\check{\delta}_{\vp_{+}}\ar[rr] &  & \check{N}_{\vp_{+}}\ar[rr] &  & N_{0}\times\check{F}_{\vp_{+}}\ar[r] & 0
}
.
\]
where $A=A_{\left(\vp_{+}\right)}\left(0\right)$. We have $\zeta_{\vp_{+}}\hat{\gamma}_{\vp_{+}}=\gamma_{M_{\vp_{+}}}\partial^{\vp_{+}}\zeta_{\underline{\vp}}$
and $f_{\vp_{+}}\zeta_{\vp_{+}}=\check{\zeta}_{\vp_{+}}f_{\vp_{+}}$
(in contrast to $\S$\ref{subsubsec:no pulled back splitting}) so
\begin{equation}
h_{\vp_{+}}=h_{M_{\vp_{+}}}\oplus\check{\gamma}_{S_{\vp_{+}}},\label{eq:bconds sum}
\end{equation}
for $h_{M_{\vp_{+}}},\check{\gamma}_{S_{\vp_{+}}}$ as in (\ref{eq:can theta_M})
and (\ref{eq:can theta_S}).
\begin{prop}
\label{prop:resumming}We have
\begin{equation}
\sum_{\vpb\in\Psi^{0,l}}\frac{\xi_{\underline{\vp}}}{l!\,\,\boldsymbol{s}\left(\vp_{l}\right)!}\int_{N_{\vpb}\times I^{l}}e^{DA_{\vpb}^{-1}\left(0,t_{1},...,t_{l}\right)^{*}\eta_{\vpb}}\pi_{\vpb}^{*}\omega_{\vpb}=\frac{\xi_{\underline{\vp}}}{\boldsymbol{s}\left(\underline{\vp}\right)}\int_{F_{\underline{\vp}}}e_{M}^{-1}\cdot e_{S}^{-1}\cdot\check{\omega}_{r}\label{eq:prop 14}
\end{equation}
where $e_{M}$, $e_{S}$ are canonical Euler forms for $M_{\underline{\vp}},S_{\underline{\vp}}$.
\end{prop}
\begin{proof}
Using a similar argument as in the proof of Proposition \ref{prop:splitting construction},
construct quadratic localizing forms $q_{M}$ and $q_{S}$ for $M_{\underline{\vp}}$
and $S_{\underline{\vp}}$, respectively, with $q_{M}|_{\partial^{\vp_{+}}M_{\underline{\vp}}}\in\im\left(h_{M_{\vp_{+}}}\right)$
and $q_{S}|_{\partial^{\vp_{+}}S_{\underline{\vp}}}\in\im\left(\check{\gamma}_{S_{\vp_{+}}}\right)$
for every $\vp_{+}\in\pc_{\basic}^{r,1}$ s.t. $\cnt\vp_{+}=\underline{\vp}$.
By (\ref{eq:bconds sum}) $q=q_{M}\oplus q_{S}$ is a quadratic localizing
form for $N_{\underline{\vp}}=M_{\underline{\vp}}\oplus S_{\underline{\vp}}$
which satisfies the hypothesis of Lemma \ref{lem:paddings collapse}
below. On the other hand, we have 
\begin{multline*}
\int_{N_{\underline{\vp}}}e^{Dq}\pi_{\underline{\vp}}^{*}\check{\omega}_{r}=\int_{F_{\underline{\vp}}}\left(\pi_{M_{\underline{\vp}}}\right)_{*}\left(e^{Dq_{M}}\right)\cdot\left(\pi_{S_{\underline{\vp}}}\right)_{*}\left(e^{Dq_{S}}\right)\check{\omega}_{r}=\int_{F_{\underline{\vp}}}e_{M}^{-1}\cdot e_{S}^{-1}\cdot\check{\omega}_{r},
\end{multline*}
where the first equality uses the projection formula and the second
uses Lemma \ref{lem:euler*segre =00003D 1}. This completes the proof.
\end{proof}
\begin{lem}
\label{lem:paddings collapse}We have 
\begin{equation}
\sum_{\vpb\in\Psi^{0,l}}\frac{1}{l!\,\boldsymbol{s}\left(\vp_{l}\right)!}\int_{N_{\vpb}\times I^{l}}e^{DA_{\vpb}^{-1}\left(0,t_{1},...,t_{l}\right)^{*}\eta_{\vpb}}\pi_{\vpb}^{*}\omega_{\vpb}=\frac{1}{\boldsymbol{s}\left(\underline{\vp}\right)!}\int_{N_{\underline{\vp}}}e^{Dq}\pi_{\underline{\vp}}^{*}\check{\omega}_{r}\label{eq:paddings collapse}
\end{equation}
where $q$ is any quadratic localizing form on $N_{\underline{\vp}}$
such that $q|_{\partial^{\vp_{+}}N_{\underline{\vp}}}\in\im\left(h_{\vp_{+}}^{*}\right)$
for every $\vp_{+}$ with $\cnt\vp_{1}=\underline{\vp}$.
\end{lem}
\begin{proof}
Consider 
\[
0=\sum_{\vpb\in\Psi^{0,l}}\frac{1}{l!\,\boldsymbol{s}\left(\vp_{l}\right)!}\int_{N_{\vpb}\times I_{u}\times I_{\left(t_{1},...,t_{l}\right)}^{l}}D\left[e^{D\left[\left(1-u\right)A_{\vpb}^{-1}\left(0,t_{1},...,t_{l}\right)^{*}\eta_{\vpb}+u\left(i_{N_{\underline{\vect\phi}}}^{\partial^{\vp_{l}}}\right)^{*}q\right]}\pi_{\vpb}^{*}\omega_{\vpb}\right]
\]

\[
=\sum_{\vpb\in\Psi^{0,l}}\frac{1}{l!\,\boldsymbol{s}\left(\vp_{l}\right)!}\int_{\partial\left(N_{\vpb}\times I_{u}\times I_{\left(t_{1},...,t_{l}\right)}^{l}\right)}e^{D\left[\left(1-u\right)A_{\vpb}^{-1}\left(0,t_{1},...,t_{l}\right)^{*}\eta_{\vpb}+u\left(i_{N_{\underline{\vect\phi}}}^{\partial^{\vp_{l}}}\right)^{*}q\right]}\pi_{\vpb}^{*}\omega_{\vpb}
\]
Write 
\begin{align*}
\partial\left(N_{\vpb}\times I_{u}\times I_{\left(t_{1},...,t_{l}\right)}^{l}\right)= & \left(\left(\partial_{+}N_{\vpb}\coprod\partial_{-}N_{\vpb}\right)\times I^{l+1}\right)\coprod\\
\coprod & N_{\vpb}\times\left\{ 0\right\} \times I^{l}\coprod N_{\vpb}\times\left\{ 1\right\} \times I^{l}\\
\coprod & N_{\vpb}\times I\times\partial I^{l}
\end{align*}
as in the proof of Proposition \ref{prop:regularizing}, the contribution
of $\partial_{+}N_{\vpb}$ vanishes and the contribution of $\coprod_{\vpb}\partial_{-}N_{\vpb}$
cancels with the contribution of 
\[
\coprod_{\vpb}\left(N_{\vpb}\times I\times\coprod_{j=1}^{l}\left(I^{j-1}\times\left\{ 0\right\} \times I^{l-j}\right)\right).
\]
The contribution of 
\[
\coprod_{\vpb=\left(\underline{\vp},\vp_{+}\right)\in\Psi^{0,1}}\left(N_{\vpb}\times I\times\left\{ 1\right\} \right).
\]
vanishes, since the integrand is pulled back along $h_{\vp_{+}}$,
which decreases the dimension of the total space. Similarly, the contribution
of $\coprod_{\vpb\in\Psi^{0,l}}\left(N_{\vpb}\times I\times I^{j-1}\times\left\{ 1\right\} \times I^{l-j}\right)$
vanishes. 

The contribution of $\coprod_{\vpb}N_{\vpb}\times0\times I^{l}$ gives
the left hand side of (\ref{eq:paddings collapse}).

The contribution of $\coprod_{\vpb\in\Psi^{0,l}}N_{\psi}\times\left\{ 1\right\} \times I^{l}$
vanishes if $l>0$ since the integrand is independent of $t_{1},...,t_{l}$.
If $l=0$ the contribution is 
\[
-\frac{1}{\boldsymbol{s}\left(\underline{\vp}\right)!}\int_{N_{\underline{\vp}}}e^{Dq}\pi_{\underline{\vp}}^{*}\check{\omega}_{r}=-\int_{F_{\underline{\vp}}}\sigma\cdot\check{\omega}_{r}.
\]
The claim immediately follows.
\end{proof}

\begin{proof}
[Proof of Proposition \ref{thm:resummation}]This follows immediately
from Propositions \ref{prop:regularizing} and \ref{prop:resumming}.
\end{proof}

\section{Appendix: orbifolds with corners}

\subsection{Review of notions from \cite{equiv-OGW-invts}}

We begin by a review of the appendix of \cite{equiv-OGW-invts}, where
more details can be found. Let $\manc$ denote the category of manifolds
with corners\footnote{by \emph{corners} we always mean ``ordinary'' corners, modeled on
$[0,\infty)^{k}\times\rr^{n-k}$; though we base our discussion on
\cite{joyce-generalized}, we will \emph{not} use manifolds with generalized
corners in this paper.} with smooth maps in the sense of \cite{joyce-generalized}. Let $\mathbf{PEG}$
be the category of proper étale groupoids in $\manc$.  A morphism
$\left(X_{1}\rightrightarrows X_{0}\right)\xrightarrow{R}\left(Y_{1}\rightrightarrows Y_{0}\right)$
in $\mathbf{PEG}$ is a \emph{refinement }if it is fully faithful
and essentially surjective, and the maps $R_{0}:X_{0}\to Y_{0},\;R_{1}:X_{1}\to Y_{1}$
of the object and morphisms spaces are étale. The category $\mathbf{Orb}$
of orbifolds with corners is the 2-localization of $\mathbf{PEG}$
by refinements. We usually denote orbifolds by caligraphic letters
$\xx,\yy,\mm$... They are given by proper étale groupoids. Maps $\xx\to\yy$
are given by fractions $F_{\bullet}|R_{\bullet}$ with $X_{\bullet}\xleftarrow{R_{\bullet}}X_{\bullet}'$
a refinement and $X'_{\bullet}\xrightarrow{F_{\bullet}}Y_{\bullet}$
a smooth functor. We refer the reader to \cite{pronk} for further
details, including the definition of the 2-cells, the composition
operations, etc. We say $f,f':\xx\to\yy$ are \emph{homotopic }if
there exists a 2-cell $f\Rightarrow f'$. 
\begin{defn}
\label{def:properties of orbi-maps}We say $f$ is \emph{weakly smooth,
smooth, strongly-smooth, étale, interior, b-normal, submersive, b-submersive,
simple }or\emph{ perfectly simple }if $F_{0}$ has the corresponding
property as a map of manifolds with corners. It is easy to check
that these properties are preserved by 2-cells (and thus are properties
of the homotopy class of $f$). The map $f$ is called a\emph{ b-fibration}
if it is b-normal and b-submersive (cf. \cite[Definition 4.3]{joyce-generalized}).

For $i=1,2$ let $f^{i}:\xx^{i}\to\yy$ be a strongly smooth map given
by the fraction $F^{i}|R^{i}$. $f^{1},f^{2}$ are said to be \emph{transverse
}(respectively, \emph{strongly transverse}) if $F_{0}^{1},F_{0}^{2}$
are transverse (resp. strongly transverse) in the sense of \cite[Definitions 6.1, 6.10]{joyce-fibered}.
\end{defn}
A map\emph{ }$f:\xx\to\yy$ is called a \emph{diffeomorphism} if it
is an equivalence in $\mathbf{Orb}$.

A manifold with corners $T$ defines an orbifold $\underline{T}$
given by the trivial groupoid $T\rightrightarrows T$. This gives
a full and faithful functor $\manc\to\text{\textbf{Orb}}$ and we
say a stack $\xx$ ``is a manifold'' if it lies in the essential
image of this functor. For a fixed orbifold $\xx=X_{1}\rightrightarrows X_{0}$
the category of arrows $\underline{T}\to\xx$ forms a geometric stack,
where the map $\underline{X_{0}}\to\xx$ is an \emph{atlas} for this
stack:
\begin{defn}
Let $\xx$ be an orbifold with corners. \emph{An atlas} \emph{for
$\xx$ }is a map $p:\underline{M}\to\xx$ where $M$ is some manifold
with corners, such that for any other map $f:\underline{N}\to\xx$
from a manifold with corners, $\underline{M}\times_{\xx}\underline{N}$
is a manifold with corners and the projection $\underline{M}\times_{\xx}\underline{N}\overset{p'}{\longrightarrow}\underline{N}$
is surjective and étale.
\end{defn}
Conversely, any such atlas defines an orbifold equivalent to $\xx$. 

If $G$ is a Lie group we define the 2-category of $G$-orbifolds,
$G$-equivariant maps, and $G$-2-cells as in \cite[Defnition 2.1]{stacky-action}.
We define the stack of fixed points $\xx^{G}$ and the quotient stack
$\xx_{G}=\xx/G$ as in \cite[Definition 2.3]{stacky-action} (sometimes
we refer to these as \emph{the homotopy fixed points} and \emph{the
stacky quotient}, respectively). If $\mm$ is a $G$-orbifold, we
denote the action of $g\in G$ by $g.^{\mm}:\mm\to\mm$ or just $g.:\mm\to\mm$.

To make the paper readable, we abuse notation and refer to maps which
are canonically isomorphic as equal. For example we may write $g.h.=\left(gh\right).$
The same goes for orbifolds which are canonically equivalent (that
is, with a given equivalence, or with an equivalence which is specified
up to a unique 2-cell). For example we may write
\[
\left(\mm_{1}\times_{L}\mm_{2}\right)\times_{L}\mm_{3}=\mm_{1}\times_{L}\left(\mm_{2}\times_{L}\mm_{3}\right).
\]

The notion of a sheaf on an orbifold $\xx$ is the same as the notion
of a sheaf on the underlying topological orbifold (see \cite{moerdijk-classifying-groupoids,pronk}).
A vector bundle $E$ on an orbifold with corners $\xx=X_{1}\overset{s,t}{\rightrightarrows}X_{0}$
is given by $\left(E_{0},\phi\right)$ where $E_{0}$ is a smooth
vector bundle on $X_{0}$ and 
\[
\phi:s^{*}E_{0}\to t^{*}E_{0}
\]
is an isomorphism satisfying some obvious compatibility requirements
with the groupoid structure. The sections of $\left(E_{0},\phi\right)$
form a sheaf over $\xx$. A \emph{local system }on an orbifold $\xx$
is a sheaf which is locally isomorphic to the constant sheaf $\underline{\zz}$.
We extend the conventions set forth in \cite[\S1.1, \S 6.1]{twA8}
to proper étale groupoids with corners in the obvious way\footnote{There we worked with $\cc$-valued local systems, whereas here we
work with $\zz$-valued local systems.}. In particular, for every vector bundle $E$ on $X_{\bullet}$ there's
a local system $\Or\left(E\right)$ on $X_{\bullet}$. The \emph{orientation
local system }of $X_{\bullet}$ is $\Or\left(TX_{\bullet}\right)$.
We have a local system isomorphism 
\begin{equation}
\iota_{\xx}^{\partial}:\Or\left(T\partial X{}_{\bullet}\right)\to\Or\left(TX_{\bullet}\right)\label{eq:boundary ls iso}
\end{equation}
lying over $i_{X_{\bullet}}^{\partial}:\partial X_{\bullet}\hookrightarrow X_{\bullet}$,
defined by appending the outward normal vector at the beginning of
the oriented base. Maps of local systems are always assumed to be
cartesian, so to specify a local system map ${\lc_{1}\xrightarrow{\ff}\lc_{2}}$
over ${\xx_{1}\xrightarrow{f}\xx_{2}}$ is equivalent to giving an
isomorphism $\lc_{1}\to f^{-1}\lc_{2}$.

Differential forms on $\xx$ are the global sections of the sheaf
of sections of $\bigwedge T\xx$; in particular they are invariant
under the local action of the isotropy group at each point. 

\subsection{Additional notions for manifolds with corners.}
\begin{defn}
A map $f:X\to Y$ of manifolds wtih corners is called a \emph{closed
immersion} if for every $p\in X$ there exists an open neighbourhood
$p\in U\subset X$, an open neighbourhood $f\left(U\right)\subset V\subset Y$,
and a submersion $h:V\to\rr^{N}$ for some integer $N\geq0$ such
that the following square is cartesian 
\[
\xymatrix{U\ar[r]^{f|_{U}}\ar[d] & V\ar[d]^{h}\\
0\ar[r] & \rr^{N}
}
\]
(it follows that $N=\dim Y-\dim X$). Note the fiber product exists
by \cite[Lemma 37]{equiv-OGW-invts} since $h$ is (vacuously) b-normal,
and $0\to\rr^{N}$ is strongly smooth and interior.
\end{defn}
\begin{rem}
\label{rem:b-submersive is enough}Any b-submersion to a manifold
without boundary is automatically a strongly smooth submersion, so
it suffices to assume that $h$ is a b-submersion. 
\end{rem}

\begin{defn}
A map $f:X\to Y$ of manifolds with corners is called \emph{a closed
embedding }if it is a closed immersion, has a closed image, and induces
a homeomorphism on its image.
\end{defn}

\begin{defn}
A map $f:X\to Y$ of manifolds with corners is an \emph{open embedding}
if it is étale and injective.
\end{defn}
\begin{lem}
If $i:X\to Y$ and $f:\mathcal{W}\to\yy$ are smooth maps of manifolds
with corners, with $i$ either a closed or an open embedding, and
$f\left(W\right)\subset i\left(X\right)$, then there is a unique
smooth map $g:W\to X$ with $f=i\circ g$. 
\end{lem}
\begin{proof}
This is Corollary 4.11 of \cite{joyce-generalized}.
\end{proof}
If we assume in addition that $f$ is a closed or an open embedding,
and that $f\left(W\right)=i\left(X\right)$, then $g$ is a diffeomorphism.

\subsection{Additional notions for orbifolds with corners.}
\begin{defn}
A map $F|R:\xx\to\yy$ of orbifolds with corners is a \emph{closed
immersion} if $F_{0}$ is a closed immersion. In this case, the same
holds for any map homotopic to $F|R$.
\end{defn}

\begin{defn}
\label{def:closed embedding}A map $f:\xx\to\yy$ of orbifolds with
corners is a \emph{closed} (respectively, \emph{open}) \emph{embedding}
if for some (hence any) atlas $p:\underline{M}\to\yy$, the 2-pullback
$\underline{M}\fibp{p}{f}\xx$ is a manifold with corners  and the
map $\underline{M}\fibp{p}{f}\xx\to\underline{M}$ is a closed (resp.
open) embedding of manifolds with corners.
\end{defn}
If $f:\xx\to\yy$ is a closed embedding we may refer to $\xx$ as
a \emph{suborbifold} of $\yy$.

\begin{defn}
\label{def:tubular ngbhd}Let $f:\xx\to\yy$ be a closed immersion
of orbifolds with corners, with normal bundle $N_{f}:=f^{*}T\yy/T\xx$.
A \emph{tubular neighbourhood} \emph{for $f$ }consists of a pair
of open embeddings 
\[
N_{f}\overset{j}{\longleftarrow}V\overset{\gamma}{\longrightarrow}\yy
\]
such that (i) there exists $i_{0}':\xx\to V$ such that $i_{0}=j\circ i_{0}'$
for $i_{0}:\xx\to N$ the closed embedding of the zero section, and
$f=\gamma\circ i_{0}'$. 

(ii) the map 
\[
N=\left(i_{0}\right)^{*}TN/T\xx=\left(i_{0}'\right)^{*}TV/T\xx\overset{\overline{d\gamma}}{\longrightarrow}f^{*}T\yy/T\xx=:N
\]
is the identity.

If we assume in addition that $f:\xx\to\yy$ is a $\tb$-equivariant
map, an \emph{equivariant tubular neighbourhood for }$f$ is a tubular
neighbourhood for $f$ such that $j$ and $\gamma$ are equivariant.
\end{defn}
\begin{defn}
\label{def:l-tubular ngbhd}If $\xx=\coprod\xx_{n}$ is an l-orbifold
with corners (cf. $\S$\ref{subsec:closed and open fps}) and $\yy$
is an orbifold with corners of dimension $r$, a closed embedding
$f:\xx\to\yy$ is just a coproduct of closed embeddings $\coprod_{n\geq0}f_{n}$.
The total space of the normal l-bundle $\coprod_{n\geq0}N_{f_{n}}$
is an (ordinary) orbifold with corners $N$ and a tubular neighborhood
for $f$ is given by a pair of maps of orbifolds with corners $N\xleftarrow{j}\mathcal{V}\xrightarrow{\gamma}\yy$,
satisfying the conditions above locally.
\end{defn}
\begin{lem}
\label{lem:tubular exists}Any closed embedding admits a tubular neighbourhood;
a $\tb$-equivariant embedding admits a $\tb$-equivariant tubular
neighbourhood.
\end{lem}
\begin{proof}
This is proven using geodesic flow along a b-metric, which may be
taken to be $\tb$-invariant, see \cite{melrose} (note that our notion
of a closed embedding of manifolds with corners corresponds to what
Melrose calls ``interior p-submanifold'').
\end{proof}

The following lemma is useful for extending constructions from the
boundary inward. Let $\xx$ be an orbifold with corners given by a
proper étale groupoid $X_{\bullet}=X_{1}\overset{s,t}{\rightrightarrows}X_{0}$.
Let $E$ be a vector bundle over $\xx$ given by $\left(E_{0},s^{*}E_{0}\xrightarrow{\phi}t^{*}E_{0}\right)$.
A section of $E$ is given by a section $\sigma:X_{0}\to E_{0}$ with
$t^{*}\sigma=\phi\circ s^{*}\sigma$. We can pull back bundles and
sections along maps of groupoids. 

Let $i_{1}$ (respectively, $i_{2}$) be the map $\partial^{2}\xx\to\partial\xx$
that remembers the first (respectively, second) local boundary component
(so $i_{1}=i_{\partial\xx}^{\partial}$).
\begin{lem}
\label{lem:extend inwards}Let ${\sigma_{\partial}}$ be a section
of $\left(i_{\xx}^{\partial}\right)^{*}E$ such that $i_{1}^{*}\sigma_{\partial}=i_{2}^{*}\sigma_{\partial}$.
Then there exists a section $\sigma$ of $E$ with $\left(i_{\xx}^{\partial}\right)^{*}\sigma=\sigma_{\partial}$.

If we assume that, in addition, $\xx$ is a $\tb$-orbifold, $E$
is a $\tb$-vector bundle, and $\sigma_{\partial}$ is a $\tb$-equivariant
section, then there exists a $\tb$-equivariant section $\sigma$
with $\left(i_{\xx}^{\partial}\right)^{*}\sigma=\sigma_{\partial}$.
\end{lem}
\begin{proof}
Using a partition of unity, i.e. a smooth function $\rho:X_{0}\to\rr$
with $t_{*}s^{*}\rho=1$, we reduce to the case $\xx=\rr_{k}^{n}$
(a manifold with corners). The claim for this case follows from Whitney's
theorem \cite{whitney}, though the argument we have is lengthy. Since
the result seems to be well-known to experts (see \cite[Proposition 4.41(b)]{joyce-generalized}),
we omit it.

The second claim follows from the first by averaging out the $\tb$-action
on any extension $\sigma$ of $\sigma_{\partial}$.
\end{proof}
\begin{lem}
\label{lem:extending forms inward}Let $\xx$ be a $\tb$-orbifold.
If $\theta_{\partial}\in\Omega\left(\partial\xx,\rr\right)^{\tb}$
satisfies $i_{1}^{*}\theta_{\partial}=i_{2}^{*}\theta_{\partial}$
then there exists $\theta\in\Omega\left(\xx,\rr\right)^{\tb}$ with
$i^{*}\theta=\theta_{\partial}$. 
\end{lem}
\begin{proof}
As in the proof of Lemma \ref{lem:extend inwards}, we may assume
without loss of generality that 
\[
\xx=\rr_{k}^{n}=\left\{ \left(x_{1},...,x_{n}\right)|x_{j}\geq0\text{ for }1\leq j\leq k\right\} 
\]
and $\tb$ acts trivially. Moreover, we may focus on some $J=\left\{ j_{1},...,j_{r}\right\} \subset\left[n\right]$,
and consider the coefficient $f$ of $dx^{J}$ in $\theta$. Let $D_{0}\subset2^{\left[k\right]}$
be defined by
\[
D_{0}=\left\{ I\subset\left[k\right]|\exists i\in\left[k\right]\backslash J|I\subset\left[k\right]\backslash\left\{ i\right\} \right\} 
\]
This represents the faces of $\rr_{k}^{n}$ for which the value of
$f$ is specified by $\theta_{\partial}$: the subset $S\subset\left[k\right]$
represents the face $F_{S}=\left\{ x_{a}=0|a\in\left[k\right]\backslash S\right\} $.
Define recursively
\[
D_{i+1}=\left\{ I\subset\left[k\right]|\forall a\in I:I\backslash\left\{ a\right\} \in D_{i}\right\} ,
\]
and prove by induction on $i$ that these are closed under taking
subsets and (using the previous lemma) that $f$ extends to $D_{i+1}\backslash D_{i}$.
We have
\[
\emptyset\in D_{\infty}\text{ and }D_{\infty}=\left\{ I\subset\left[k\right]|\forall a\in I:I\backslash\left\{ a\right\} \in D_{\infty}\right\} 
\]
which implies $D_{\infty}=2^{\left[k\right]}$ (consider $S\in2^{\left[k\right]}\backslash D_{\infty}$
with $\left|S\right|$ minimal). This shows that $f$ extends to $\rr_{k}^{n}$
and completes the proof.
\end{proof}

\bibliographystyle{amsabbrvc}
\bibliography{localization}

\end{document}